%% file: main.tex
\documentclass[10pt]{article}
\usepackage{hyperref}
\usepackage{amsmath}
\usepackage{amsfonts}
\usepackage{amssymb}
\usepackage{amsthm}

\usepackage[shortlabels]{enumitem}
\usepackage{verbatim}
\usepackage{graphicx}
\usepackage{microtype}
\usepackage{fullpage}
\usepackage{xparse}
\usepackage{mathrsfs}
\usepackage[capitalise]{cleveref}

\crefname{thm}{Theorem}{Theorems}
\Crefname{thm}{Theorem}{Theorems}
\crefname{conj}{Conjecture}{Conjectures}
\Crefname{conj}{Conjecture}{Conjectures}
\crefname{prop}{Proposition}{Propositions}
\Crefname{prop}{Proposition}{Propositions}
\crefname{cor}{Corollary}{Corollaries}
\Crefname{cor}{Corollary}{Corollaries}
\crefname{defn}{Definition}{Definitions}
\Crefname{defn}{Definition}{Definitions}
\crefname{rmk}{Remark}{Remarks}
\Crefname{rmk}{Remark}{Remarks}
\crefname{prob}{Problem}{Problems}
\Crefname{prob}{Problem}{Problems}

\crefname{enumi}{}{}
\Crefname{enumi}{}{}

\crefname{figure}{Figure}{Figures}
\Crefname{figure}{Figure}{Figures}

\crefformat{equation}{(#2#1#3)}
\crefrangeformat{equation}{(#3#1#4) to~(#5#2#6)}
\crefmultiformat{equation}{(#2#1#3)}%
{ and~(#2#1#3)}{, (#2#1#3)}{, and~(#2#1#3)}

\Crefformat{equation}{(#2#1#3)}
\Crefrangeformat{equation}{(#3#1#4) to~(#5#2#6)}
\Crefmultiformat{equation}{(#2#1#3)}%
{ and~(#2#1#3)}{, (#2#1#3)}{, and~(#2#1#3)}

\begin{document}

\NewDocumentCommand{\C}{}{{\mathbb{C}}}
\NewDocumentCommand{\R}{}{{\mathbb{R}}}
\NewDocumentCommand{\Q}{}{{\mathbb{Q}}}
\NewDocumentCommand{\Z}{}{{\mathbb{Z}}}
\NewDocumentCommand{\N}{}{{\mathbb{N}}}
\NewDocumentCommand{\M}{}{{\mathbb{M}}}
\NewDocumentCommand{\grad}{}{\nabla}
\NewDocumentCommand{\sA}{}{\mathcal{A}}
\NewDocumentCommand{\sF}{}{\mathcal{F}}
\NewDocumentCommand{\sH}{}{\mathcal{H}}
\NewDocumentCommand{\sD}{}{\mathcal{D}}
\NewDocumentCommand{\sB}{}{\mathcal{B}}
\NewDocumentCommand{\sC}{}{\mathcal{C}}
\NewDocumentCommand{\sL}{}{\mathcal{L}}
\NewDocumentCommand{\sT}{}{\mathcal{T}}
\NewDocumentCommand{\sO}{}{\mathcal{O}}
\NewDocumentCommand{\sP}{}{\mathcal{P}}
\NewDocumentCommand{\sQ}{}{\mathcal{Q}}
\NewDocumentCommand{\sR}{}{\mathcal{R}}
\NewDocumentCommand{\sM}{}{\mathcal{M}}
\NewDocumentCommand{\sI}{}{\mathcal{I}}
\NewDocumentCommand{\sK}{}{\mathcal{K}}
\NewDocumentCommand{\Span}{}{\mathrm{span}}
\NewDocumentCommand{\fM}{}{\mathfrak{M}}
\NewDocumentCommand{\fN}{}{\mathfrak{N}}
\NewDocumentCommand{\fX}{}{\mathfrak{X}}
\NewDocumentCommand{\fY}{}{\mathfrak{Y}}
\NewDocumentCommand{\gammat}{}{\tilde{\gamma}}
\NewDocumentCommand{\gammah}{}{\hat{\gamma}}
\NewDocumentCommand{\ct}{}{\tilde{c}}
\NewDocumentCommand{\bt}{}{\tilde{b}}
\NewDocumentCommand{\ch}{}{\hat{c}}
\NewDocumentCommand{\Ut}{}{\tilde{U}}
\NewDocumentCommand{\Vt}{}{\tilde{V}}
\NewDocumentCommand{\ah}{}{\hat{a}}
\NewDocumentCommand{\at}{}{\tilde{a}}
\NewDocumentCommand{\Yh}{}{\widehat{Y}}
\NewDocumentCommand{\Yt}{}{\widetilde{Y}}
\NewDocumentCommand{\Ah}{}{\widehat{A}}
\NewDocumentCommand{\At}{}{\widetilde{A}}
\NewDocumentCommand{\Vol}{m}{\mathrm{Vol}(#1)}
\NewDocumentCommand{\BVol}{m}{\mathrm{Vol}\left(#1\right)}
\NewDocumentCommand{\fg}{}{\mathfrak{g}}
\NewDocumentCommand{\Div}{}{\mathrm{div}}

\NewDocumentCommand{\Xa}{}{X^{(\alpha)}}
\NewDocumentCommand{\Va}{}{V^{(\alpha)}}

\NewDocumentCommand{\transpose}{}{\top}

\NewDocumentCommand{\ICond}{}{\sC}

\NewDocumentCommand{\LebDensity}{}{\sigma_{\mathrm{Leb}}}


\NewDocumentCommand{\Lie}{m}{\sL_{#1}}

\NewDocumentCommand{\ZygSymb}{}{\mathscr{C}}

\NewDocumentCommand{\Zyg}{m o}{\IfNoValueTF{#2}{\ZygSymb^{#1}}{\ZygSymb^{#1}(#2) }}
\NewDocumentCommand{\ZygX}{m m o}{\IfNoValueTF{#3}{\ZygSymb^{#2}_{#1}}{\ZygSymb^{#2}_{#1}(#3) }}

\NewDocumentCommand{\CSpace}{m o}{\IfNoValueTF{#2}{C(#1)}{C(#1;#2)}}

\NewDocumentCommand{\CjSpace}{m o o}{\IfNoValueTF{#2}{C^{#1}}{ \IfNoValueTF{#3}{ C^{#1}(#2)}{C^{#1}(#2;#3) } }  }

\NewDocumentCommand{\CXjSpace}{m m o}{\IfNoValueTF{#3}{C^{#2}_{#1}}{ C^{#2}_{#1}(#3) } }

\NewDocumentCommand{\HSpace}{m m o o}{\IfNoValueTF{#3}{C^{#1,#2}}{ \IfNoValueTF{#4} {C^{#1,#2}(#3)} {C^{#1,#2}(#3;#4)} }}

\NewDocumentCommand{\HXSpace}{m m m o}{\IfNoValueTF{#4}{C_{#1}^{#2,#3}}{  {C_{#1}^{#2,#3}(#4)}  }}

\NewDocumentCommand{\ZygSpace}{m o o}{\IfNoValueTF{#2}{\ZygSymb^{#1}}{ \IfNoValueTF{#3} { \ZygSymb^{#1}(#2) }{\ZygSymb^{#1}(#2;#3) } } }

\NewDocumentCommand{\ZygXSpace}{m m o}{\IfNoValueTF{#3}{\ZygSymb^{#2}_{#1}}{\ZygSymb^{#2}_{#1}(#3) }}

\NewDocumentCommand{\Norm}{m o}{\IfNoValueTF{#2}{\| #1\|}{\|#1\|_{#2} }}
\NewDocumentCommand{\BNorm}{m o}{\IfNoValueTF{#2}{\left\| #1\right\|}{\left\|#1\right\|_{#2} }}

\NewDocumentCommand{\CjNorm}{m m o o}{ \IfNoValueTF{#3}{ \Norm{#1}[\CjSpace{#2}]} { \IfNoValueTF{#4}{\Norm{#1}[\CjSpace{#2}[#3]]} {\Norm{#1}[\CjSpace{#2}[#3][#4]]}  }  }

\NewDocumentCommand{\CNorm}{m m}{\Norm{#1}[\CSpace{#2}]}

\NewDocumentCommand{\BCNorm}{m m}{\BNorm{#1}[\CSpace{#2}]}

\NewDocumentCommand{\CXjNorm}{m m m o}{\Norm{#1}[
\IfNoValueTF{#4}
{\CXjSpace{#2}{#3}}
{\CXjSpace{#2}{#3}[#4]}
]}

\NewDocumentCommand{\BCXjNorm}{m m m o}{\BNorm{#1}[
\IfNoValueTF{#4}
{\CXjSpace{#2}{#3}}
{\CXjSpace{#2}{#3}[#4]}
]}


\NewDocumentCommand{\BCjNorm}{m m o}{ \IfNoValueTF{#3}{ \BNorm{#1}[C^{#2}]} { \BNorm{#1}[C^{#2}(#3)]  }  }

\NewDocumentCommand{\HNorm}{m m m o o}{ \IfNoValueTF{#4}{ \Norm{#1}[\HSpace{#2}{#3}]} {
\IfNoValueTF{#5}
{\Norm{#1}[\HSpace{#2}{#3}[#4]]}
{\Norm{#1}[\HSpace{#2}{#3}[#4][#5]] }
}  }

\NewDocumentCommand{\HXNorm}{m m m m o}{ \IfNoValueTF{#5}{ \Norm{#1}[\HXSpace{#2}{#3}{#4}]} {
{\Norm{#1}[\HXSpace{#2}{#3}{#4}[#5]]}
}  }

\NewDocumentCommand{\BHXNorm}{m m m m o}{ \IfNoValueTF{#5}{ \BNorm{#1}[\HXSpace{#2}{#3}{#4}]} {
{\BNorm{#1}[\HXSpace{#2}{#3}{#4}[#5]]}
}  }

\NewDocumentCommand{\ZygNorm}{m m o o}{ \IfNoValueTF{#3}{ \Norm{#1}[\ZygSpace{#2}]} {
\IfNoValueTF{#4}
{\Norm{#1}[\ZygSpace{#2}[#3]]}
{\Norm{#1}[\ZygSpace{#2}[#3][#4]]}
}  }

\NewDocumentCommand{\ZygXNorm}{m m m o}{\Norm{#1}[
\IfNoValueTF{#4}
{\ZygXSpace{#2}{#3}}
{\ZygXSpace{#2}{#3}[#4]}
]}

\NewDocumentCommand{\BZygXNorm}{m m m o}{\BNorm{#1}[
\IfNoValueTF{#4}
{\ZygXSpace{#2}{#3}}
{\ZygXSpace{#2}{#3}[#4]}
]}

\NewDocumentCommand{\ComegaSpace}{m o o}{\IfNoValueTF{#2}{\CjSpace{\omega,#1}}{
\IfNoValueTF{#3}
{\CjSpace{\omega,#1}[#2]}
{\CjSpace{\omega,#1}[#2][#3]}
}
}

\NewDocumentCommand{\CXomegaSpace}{m m o o}{\IfNoValueTF{#3}{\CXjSpace{#1}{\omega,#2}}{
\IfNoValueTF{#4}
{\CXjSpace{#1}{\omega,#2}[#3]}
{\CXjSpace{#1}{\omega,#2}[#3][#4]}
}
}

\NewDocumentCommand{\ComegaNorm}{m m o o}{\IfNoValueTF{#3}{ \Norm{#1}[\ComegaSpace{#2}] }
{
\IfNoValueTF{#4}
{\Norm{#1}[\ComegaSpace{#2}[#3]] }
{\Norm{#1}[\ComegaSpace{#2}[#3][#4]] }
}
}

\NewDocumentCommand{\CXomegaNorm}{m m m o o}{\IfNoValueTF{#4}{ \Norm{#1}[\CXomegaSpace{#2}{#3}] }
{
\IfNoValueTF{#5}
{\Norm{#1}[\CXomegaSpace{#2}{#3}[#4]] }
{\Norm{#1}[\CXomegaSpace{#2}{#3}[#4][#5]] }
}
}


\NewDocumentCommand{\diff}{o m}{\IfNoValueTF{#1}{\frac{\partial}{\partial #2}}{\frac{\partial^{#1}}{\partial #2^{#1}} }}

\NewDocumentCommand{\dt}{o}{\IfNoValueTF{#1}{\diff{t}}{\diff[#1]{t} }}

\NewDocumentCommand{\Zygad}{m}{\{ #1\}}

\NewDocumentCommand{\Zygmad}{m}{\Zygad{#1,-1}}
\NewDocumentCommand{\ZygEad}{m}{\Zygad{#1{:}\: {\mathrm{E}}}}

\NewDocumentCommand{\ZeroE}{}{0{:}\mathrm{E}}

\NewDocumentCommand{\Had}{m}{\langle #1\rangle}
\NewDocumentCommand{\Hmad}{m m}{\Had{ #1, -1, #2}}
\NewDocumentCommand{\HEad}{m}{\Had{#1{:}\: \mathrm{E}}}

\NewDocumentCommand{\DiffOp}{m}{\Delta_{#1}}

\NewDocumentCommand{\InnerProduct}{m m}{\left\langle #1, #2\right\rangle}

\NewDocumentCommand{\Real}{}{\mathrm{Re}}
\NewDocumentCommand{\Imag}{}{\mathrm{Im}}

\newtheorem{thm}{Theorem}[section]
\newtheorem{cor}[thm]{Corollary}
\newtheorem{prop}[thm]{Proposition}
\newtheorem{lemma}[thm]{Lemma}
\newtheorem{conj}[thm]{Conjecture}
\newtheorem{prob}[thm]{Problem}

\theoremstyle{remark}
\newtheorem{rmk}[thm]{Remark}

\theoremstyle{definition}
\newtheorem{defn}[thm]{Definition}

\theoremstyle{definition}
\newtheorem{assumption}[thm]{Assumption}

\theoremstyle{remark}
\newtheorem{example}[thm]{Example}

\numberwithin{equation}{section}

\title{Coordinates Adapted to Vector Fields: Canonical Coordinates}
\author{Betsy Stovall and Brian Street}
\date{}

\maketitle

\begin{abstract}
\input{abstract}
\end{abstract}

\tableofcontents

\section{Introduction}
\input{intro2}


\section{Function Spaces}\label{Section::FuncSpace}
\input{funcspaces}

    \subsection{Function Spaces on Euclidean Space}\label{Section::FuncSpace::Euclid}
    \input{funcspacesrn}

    \subsection{Function Spaces on Manifolds}\label{Section::FuncSpace::Manif}
    \input{funcspacesm}

        \subsubsection{Beyond Manifolds}\label{Section::FuncSpace::BeyondManif}
        \input{funcspacesbeyond}

\section{Overview of the Series}\label{Section::Overview}
\input{series}

    \subsection{Qualitative Results}\label{Section::Series::Qual}
    \input{seriesqual}

    \subsection{Quantitative Results}\label{Section::Series::Quant}
    \input{seriesquant}

    	\subsubsection{Diffeomorphism Invariance}\label{Section::Series::DiffInv}
	\input{seriesdiffinv}

\section{Main Results of this Paper}\label{Section::Results}
\input{results}

    \subsection{More on the assumptions}\label{Section::Results::MoreAssump}
    \input{moreassump}

\section{Wedge Products}\label{Section::Wedge}
\input{wedge}

\section{Densities}\label{Section::Densities}
\input{densities}

\section{Scaling and other consequences}\label{Section::Scale}
\input{scale}

    \subsection{Classical sub-Riemmanian geometries and the work of Nagel, Stein, and Wainger}\label{Section::Scale::NSW}
    \input{nsw}

        \subsubsection{H\"ormander's condition}
        \input{hormander}

    \subsection{Multi-parameter Balls}\label{Section::Scale::MultiParam}
    \input{mulitparam}

    \subsection{Generalized sub-Riemannian geometries}\label{Section::Scale::GenSubR}
    \input{gensubr}

    \subsection{Diffeomorphism Invariance and Nonsmooth Vector Fields}
    \input{nonsmooth}

    \subsection{Several Complex Variables}\label{Section::Scale::SCV}
    \input{scv}


\section{Function Spaces, revisited}
\input{funcspacesrev}

    \subsection{Comparison with Euclidean Function Spaces}\label{Section::FuncSpaceRev::Compare}
    \input{funcspacesrevcomp}

\section{Proofs}
\input{pfintro}

    \subsection{An ODE}\label{Section::Proofs::ODE}
    \input{pfode}

        \subsubsection{Derivation of the ODE}\label{Section::DerivODE}
        \input{pfodederive}

        \subsubsection{Regularity Properties}\label{Section::RegularODE}
        \input{pfoderegularity}

    \subsection{An Inverse Function Theorem}\label{Section::Proofs::IFT}
    \input{pfift}

    \subsection{Proof of the main result}\label{Section::Proofs::Main}
    \input{pfmain}

        \subsubsection{Linearly Independent}\label{Section::Proofs::LI}
        \input{pfli}

        \subsubsection{Linearly Dependent}
        \input{pfld}

    \subsection{Densities}\label{Section::Proofs::Densities}
    \input{pfdensities}

    \subsection{More on the assumptions}\label{Section::Proofs::MoreAssump}
    \input{pfmoreassump}

\appendix

\section{\texorpdfstring{Proof of \cref{Prop::ResQual::Mmanif}}{Orbits}}\label{Appendix::ProofOfImmerse}
\input{proofimmers}


\bibliographystyle{amsalpha}

\bibliography{coords}

\center{\it{University of Wisconsin-Madison, Department of Mathematics, 480 Lincoln Dr., Madison, WI, 53706}}

\center{\it{stovall@math.wisc.edu, street@math.wisc.edu}}

\center{MSC 2010:  58A30 (Primary), 57R55 and 53C17 (Secondary)}

\end{document}

%% file: abstract.tex

Given a finite collection of $C^1$ vector fields on a $C^2$ manifold which span the tangent space at every point,
we consider the question of when there is locally a coordinate system in which these
vector fields have a higher level of smoothness.
For example, when is there a coordinate system in which the vector fields are smooth, or real analytic,
or have Zygmund regularity of some finite order?
We address this question in a quantitative way, which strengthens and generalizes
previous works on the quantitative theory of sub-Riemannian (aka Carnot-Carath\'eodory)
geometry due to Nagel, Stein, and Wainger, Tao and Wright, the second author, and others.
Furthermore, we provide a diffeomorphism invariant version of these theories.
This is the first part in a three part series of papers.  
In this paper, we study a particular coordinate system 
adapted to a collection of vector fields (sometimes called canonical coordinates)
and present results related to the above questions which are not quite sharp; these
results form the backbone of the series.
The methods of this paper are based on techniques from ODEs.
In the second paper, we use additional methods from PDEs to obtain the sharp results.
In the third paper, we prove  results concerning real analyticity and
use methods from ODEs.

%% file: intro2.tex

Let $X_1,\ldots, X_q$ be $C^1$ vector fields on a $C^2$ manifold $M$, which span the tangent space at every point of $M$.
For $s>0$ let $\ZygSpace{s}$ denote the Zygmund space of order $s$ (see \cref{Section::FuncSpace::Euclid}),
let $\ZygSpace{\infty}$ denote the space of smooth functions, and let $\ZygSpace{\omega}$ denote the space
of real analytic functions.
In this three part series of papers, we investigate the following closely related questions for $s\in (1,\infty]\cup\{\omega\}$:\footnote{We define $\infty+1=\infty+2=\infty$ and $\omega+1=\omega+2=\omega$.}
\begin{enumerate}[(i)]
    \item\label{Item::Intro::LocalQual} When is there a coordinate system near a fixed point $x_0\in M$ such that the vector fields $X_1,\ldots, X_q$
    are $\Zyg{s+1}$ in this coordinate system?
    \item\label{Item::Intro::GlobalQual} When is there a $\Zyg{s+2}$ manifold structure on $M$, compatible with its $C^2$ structure, such that $X_1,\ldots, X_q$
    are $\Zyg{s+1}$ with respect to this structure?  When such a structure exists, we will see it is unique.
    \item\label{Item::Intro::Charts} When there is a coordinate system as in \cref{Item::Intro::LocalQual}, how can we pick
    it so that $X_1,\ldots, X_q$ are ``normalized'' in this coordinate system in a quantitative way which
    is useful for applying techniques from analysis?
\end{enumerate}
We present necessary and sufficient, coordinate free, conditions for \cref{Item::Intro::LocalQual,Item::Intro::GlobalQual}
and, under these conditions, give a quantitative answer to \cref{Item::Intro::Charts}.
See \cref{Section::Overview} for an overview of the results of this series.
The outline of this series is as follows:
\begin{enumerate}[(I)]
\item In this paper, we study a particular explicit coordinate system adapted to a collection of vector fields.
This coordinate system is sometimes known (at least in the setting of Lie groups) as canonical coordinates of the first kind.
This builds on previous work of Nagel, Stein, and Wainger \cite{NagelSteinWaingerBallsAndMetrics}, Tao and Wright \cite[Section 4]{TaoWrightLpImproving}, and the second author \cite{S}.  To study these canonical coordinates, we use methods from ODEs.
Unfortunately, the results given by these methods are one derivative short of being optimal (see \cref{Rmk::Results::LackOfOptimal}).

\item In the second paper \cite{StreetII}, we obtain the optimal results (in terms of Zygmund spaces) by introducing a new (implicitly defined) coordinate system.
The second paper takes as a starting point the main result of this paper, and then uses methods from PDEs
to obtain the sharp results.  These PDE methods were inspired by, and are closely related to, Malgrange's celebrated
proof of the Newlander-Nirenberg theorem \cite{MalgrangeSurLIntegbrabilite}.

\item While the second paper obtains optimal regularity in terms of Zgymund spaces, the methods there
are not applicable to the real analytic setting.  In the third paper \cite{StreetIII}, we  return to canonical coordinates and methods from ODEs to obtain
results regarding real analyticity.   The third paper takes the main results of this paper as its starting point.
\end{enumerate}

To help explain the sorts of questions we investigate, we consider a trivial example.
\begin{example}\label{Example::Intro::OneV}
Let $X$ be a $C^1$ vector field on a $C^2$ manifold $\fM$ with $X(x_0)\ne 0$ for some $x_0\in \fM$.
Let $M$ be the integral curve of $X$ passing through $x_0$.  It is well known that there is a unique $C^2$ manifold
structure on $M$ which sees $M\hookrightarrow \fM$ as a $C^2$ injective immersion (see \cref{Prop::ResQual::Mmanif});
$X$ spans the tangent space to $M$ at every point.  Set $\Phi(t):=e^{tX}x_0$ and let $I\subseteq \R$ be a maximal
open interval containing $0$ such that $\Phi$ is defined on $I$ and $\Phi:I\rightarrow M$ is injective.
It is easy to see that $\Phi\big|_I$ is a $C^2$ diffeomorphism onto its image, and therefore $\Phi$
defines a coordinate chart on $M$ near $x_0$.  In this coordinate system $X$ equals $\diff{t}$; more precisely,
$\Phi^{*}X=\diff{t}$.  Thus, we have not only picked a coordinate chart on $M$ in which $X$ is smooth,
but we have also chosen it so that $X$ is ``normalized'' to be $\diff{t}$.
\end{example}

It is straightforward to generalize \cref{Example::Intro::OneV} to a finite collection of vector fields, so long
as the vector fields are assumed to commute.  The purpose of this series of papers
is to consider similar results when the vector fields are not assumed to commute; in which case it is not always
possible to pick a coordinate system in which the vector fields are smooth.  Indeed, we present necessary
and sufficient conditions for when one can pick a coordinate system giving the vector fields a desired level
of regularity.

The coordinate charts developed in this series can be viewed as scaling maps in a wide variety of problems; this is described in more detail
in \cref{Section::Scale}.
Seen in this perspective, these results are the latest, most general, and sharpest in a series
of papers on the quantitative theory of sub-Riemannian (or Carnot-Carath\'eodory) geometry.
This started with the foundational work of Nagel, Stein, and Wainger \cite{NagelSteinWaingerBallsAndMetrics}
and the closely related work of C.\ Fefferman and S\'anchez-Calle \cite{FeffermanSanchezCalleFundamentalSoltuions}.
Tao and Wright \cite[Section 4]{TaoWrightLpImproving} furthered the results of Nagel, Stein, and Wainger
and offered a new proof based on methods from ODEs (see \cref{Section::Proofs::ODE} for a detailed discussion of
the primary ODE they studied).  The second author unified these two approaches to prove more general results
in \cite{S}.
This series of papers can be seen as strengthening and generalizing these theories and casting them in a way which is
completely ``coordinate free'' in the sense that all of our assumptions and estimates
are invariant under arbitrary $C^2$ diffeomorpisms.
The most basic version of this scaling perspective can be seen in \cref{Example::Intro::OneV},
as the next example shows.

\begin{example}
We take the setting of \cref{Example::Intro::OneV} with $\fM=\R$, $x_0=0$, $X=\delta \diff{x}$,
for some fixed $\delta>0$.  In this case $\Phi(t)=\delta t$; thus the pullback via $\Phi$
is the usual Euclidean dilation of vector fields.
We can therefore think of \cref{Example::Intro::OneV} as a generalization of the usual dilation
maps on $\R$.
\end{example}

As described above, the main results of this series have two facets:
\begin{itemize}
\item They provide a coordinate system in which given $C^1$ vector fields have an optimal degree of smoothness.
\item They provide a coordinate system in which given vector fields are normalized in a way which is useful for applying techniques from analysis.
\end{itemize}
These two facets, along with some applications, are described in more detail in \cref{Section::Scale}.

Despite the fact that the results in the second paper of this series are sharp in terms of regularity, and the
results in this paper are one-derivative off from being optimal, we believe the methods and results of this
paper have several advantages over those in the second paper.  Some of these advantages are:
\begin{enumerate}[(a)]
\item The coordinate system defined in this paper is explicit, while it is only defined implicitly in the second paper.
\item The proofs in this paper are simpler.  Indeed, the second paper requires all of the results of this paper, plus
additional methods from PDEs.
\item Despite having a simpler proof, the main results of this paper are still useful in many applications.
Indeed, they are stronger, sharper, and more general than the previous works on this subject
\cite{NagelSteinWaingerBallsAndMetrics,TaoWrightLpImproving,S} which have had many applications; see \cref{Section::Scale} for further details.  However, they are not strong enough
to obtain some of the most interesting consequences of the results in the second paper; for example, the results stated in \cref{Section::Series::Quant}.
The PDE methods will also be necessary for future work of the second author in the complex setting; see \cref{Section::Scale::SCV}.
\item Because the methods of this paper are based on ODEs, as opposed to the PDEs in the second paper, they
are in some ways more robust, and will likely be easier to adapt to other settings.  For example, in the third paper of the series we see that these ODE methods can be used
to study the real analytic setting.
\end{enumerate}

\noindent\textbf{Acknowledgements:}
We thank the referee whose detailed comments improved the exposition.
Stovall was partially supported by National Science Foundation Grant No.\ 1600458.
Street was partially supported by National Science Foundation Grant Nos.\ 1401671 and 1764265.
This material is partially based upon work supported by the National Science Foundation
under Grant No.\ 1440140, while the authors were in residence at the Mathematical Sciences Research Institute in Berkeley, California, during the spring semester of 2017.


%% file: funcspaces.tex

Before we can state any results, we need to introduce the function spaces we use.
We make a distinction between function spaces on subsets of $\R^n$ and function
spaces on a $C^2$ manifold $M$.  On $\R^n$, we have access to the standard coordinate system (and its induced smooth structure)
and we can define all of the usual function spaces and their norms in terms of this coordinate system.
On $M$, we do not have access to any such natural coordinates, and it does not make
sense to talk about, for example, $C^\infty$ functions on $M$; as this would depend
on a choice of coordinate system or smooth structure.  However, if we are given a finite collection
of vector fields on $M$, it does make sense to talk about functions which 
are $C^\infty$ with respect to these vector fields, and this is how we shall proceed.

%% file: funcspacesrn.tex

Let $\Omega\subset \R^n$ be a bounded, connected, open set (we will almost always
be considering the case when $\Omega$ is a ball in $\R^n$).  We have the following
classical Banach spaces of functions on $\Omega$:
\begin{equation*}
    \CSpace{\Omega}=\CjSpace{0}[\Omega]:=\{f:\Omega\rightarrow \C \:\big|\: f\text{ is continuous and bounded}\}, \quad \CNorm{f}{\Omega}=\CjNorm{f}{0}[\Omega]:=\sup_{x\in \Omega} |f(x)|.
\end{equation*}
For $m\in \N$ (throughout the paper we take the convention $0\in \N$),
\begin{equation*}
    \CjSpace{m}[\Omega]:=\{ f\in \CjSpace{0}[\Omega] \:\big|\: \partial_x^{\alpha} f\in \CjSpace{0}[\Omega],\forall |\alpha|\leq m\}, \quad \CjNorm{f}{m}[\Omega]:=\sum_{|\alpha|\leq m} \CjNorm{\partial_x^{\alpha} f}{0}[\Omega].
\end{equation*}
Next we define the classical Lipschitz-H\"older spaces.  For $s \in [0,1]$,
\begin{equation}\label{Eqn::FuncSpaceRn::DefnHolder}
    \HNorm{f}{0}{s}[\Omega]:=\CjNorm{f}{0}[\Omega]+\sup_{\substack{x,y\in \Omega\\ x\ne y }} |x-y|^{-s} |f(x)-f(y)|, \quad \HSpace{0}{s}[\Omega]:= \{ f\in \CjSpace{0}[\Omega] : \HNorm{f}{0}{s}[\Omega]<\infty\}.
\end{equation}
For $m\in \N$, $s\in [0,1]$,
\begin{equation*}
    \HNorm{f}{m}{s}[\Omega]:=\sum_{|\alpha|\leq m} \HNorm{\partial_x^{\alpha} f}{0}{s}[\Omega], \quad \HSpace{m}{s}[\Omega]:=\{f\in \CjSpace{m}[\Omega] : \HNorm{f}{m}{s}[\Omega]<\infty\}.
\end{equation*}
Next, we turn to the Zygmund-H\"older spaces.  Given $h\in \R^n$ define
$\Omega_h:=\{x\in \R^n : x,x+h,x+2h\in \Omega\}$.
For $s\in (0,1]$ set
\begin{gather*}
    \ZygNorm{f}{s}[\Omega]:=\HNorm{f}{0}{s/2}[\Omega]+\sup_{\substack{0\ne h\in \R^n \\ x\in \Omega_h}} |h|^{-s} \left|f(x+2h)-2f(x+h)+f(x)\right|,
    \\\ZygSpace{s}[\Omega]:=\{f\in \CjSpace{0}[\Omega] : \ZygNorm{f}{s}[\Omega]<\infty\}.
\end{gather*}
For $m\in \N$, $s\in (0,1]$, set
\begin{equation*}
    \ZygNorm{f}{m+s}[\Omega]:=\sum_{|\alpha|\leq m} \ZygNorm{\partial_x^{\alpha} f}{s}[\Omega], \quad \ZygSpace{s+m}[\Omega]:=\{ f\in \CjSpace{m}[\Omega] : \ZygNorm{f}{s+m}[\Omega]<\infty\}.
\end{equation*}
We set,
\begin{equation*}
    \ZygSpace{\infty}[\Omega]:=\bigcap_{s>0} \ZygSpace{s}[\Omega], \quad \CjSpace{\infty}[\Omega]:= \bigcap_{m\in \N} \CjSpace{m}[\Omega].
\end{equation*}
When $\Omega$ is a ball, $\ZygSpace{\infty}[\Omega]=\CjSpace{\infty}[\Omega]$.

\begin{rmk}\label{Rmk::FuncSpace::NonstandardNorm}
The term $\HNorm{f}{0}{s/2}[\Omega]$ in the definition of $\ZygNorm{f}{s}[\Omega]$ is somewhat unusual, and
in the literature is usually replaced by $\CjNorm{f}{0}[\Omega]$.  As is well-known, if $\Omega$ is a bounded Lipschitz domain, these two options yield
equivalent\footnote{This equivalence  follows easily from \cite[Theorem 1.118 (i)]{TriebelTheoryOfFunctionSpacesIII}.  We will usually use these norms in the case when $\Omega$ is a ball in Euclidean space, and is therefore
a bounded Lipschitz domain.} 
(but not equal) norms--and therefore the space $\ZygSpace{s}[\Omega]$ is the usual Zygmund-H\"older
space of order $s$.  However, the constants involved in this equivalence of norms depend on the size of $\Omega$,
and the above choice is more convenient for our purposes.  For an example of the convenience offered
by this choice of norm, see \cref{Rmk::FuncSpaceRev::NonstandardNorm}.
\end{rmk}

Finally, we turn to spaces of real analytic functions.  Given $r>0$ we define:
\begin{equation*}
    \ComegaNorm{f}{r}[\Omega]:=\sum_{\alpha\in \N^n} \frac{\CNorm{\partial_x^{\alpha} f}{\Omega}}{\alpha!} r^{|\alpha|},\quad \ComegaSpace{r}[\Omega]:=\{f\in \CjSpace{\infty}[\Omega] : \ComegaNorm{f}{r}[\Omega]<\infty\}.
\end{equation*}
We set $\CjSpace{\omega}[\Omega]:=\bigcup_{r>0} \ComegaSpace{r}[\Omega]$.
For notational convenience, we set $\ZygSpace{\omega}[\Omega]:=\CjSpace{\omega}[\Omega]$.

Throughout the paper, if we say $\CjNorm{f}{m}[\Omega]<\infty$, it means that $f\in \CjSpace{m}[\Omega]$,
and similarly for any other function space.

For a Banach space $V$ we define the same spaces taking values in $V$ by the obvious modifications and
write $\CjSpace{m}[\Omega][V]$, $\HSpace{m}{s}[\Omega][V]$,  $\ZygSpace{m+s}[\Omega][V]$,  $\ComegaSpace{r}[\Omega][V]$, and $\CjSpace{\omega}[\Omega][V]$ to denote
these spaces.
When we have a vector field $X$ on $\Omega$, we identify $X$ with a function $X:\Omega\rightarrow \R^n$
by writing $X=\sum_{j=1}^n a_j(x) \diff{x_j}$ and treating $X$ as the function $X(x)=(a_1(x),\ldots, a_n(x))$.  Thus, it makes sense to consider norms like $\ZygNorm{X}{s}[\Omega][\R^n]$ and $\HNorm{X}{m}{s}[\Omega][\R^n]$.

%% file: funcspacesm.tex

Let $X_1,\ldots, X_q$ be $C^1$ vector fields on a connected $C^2$ manifold $M$.
Define the Carnot-Carath\'eodory ball associated to $X_1,\ldots, X_q$, centered
at $x\in M$, of radius $\delta>0$, by
\begin{equation}\label{Eqn::FundSpaceM::DefnCCBall}
    \begin{split}
        B_X(x,\delta):=\Bigg\{
        y\in M\: \bigg|\: &\exists \gamma:[0,1]\rightarrow M, \gamma(0)=x, \gamma(1)=y,
        \gamma'(t)=\sum_{j=1}^q a_j(t) \delta X_j(\gamma(t)),
        \\&a_j\in L^\infty([0,1]), \BNorm{\sum_{j=1}^q|a_j|^2 }[L^\infty]<1
        \Bigg\},
    \end{split}
\end{equation}
and for $y\in M$, set
\begin{equation}\label{Eqn::FucsSpaceM::rho}
    \rho(x,y):=\inf\{\delta>0 : y\in B_X(x,\delta)\}.
\end{equation}

When $X_1,\ldots, X_q$ are smooth vector fields on a smooth connected manifold $M$, if the Lie algebra generated by $X_1,\ldots, X_q$ spans the tangent space at every point of $M$, $\rho$ is a metric on $M$--sometimes
known as a \textit{sub-Riemannian metric} or a \textit{Carnot-Carath\'eodory metric}.
In this case, the metric toplogy induced by $\rho$ is the same as the topology on $M$.
If the Lie algebra generated by $X_1,\ldots, X_q$ does not span
the tangent space at some point, then $\rho$ may or may not be a metric:  it is possible that $\rho(x,y)=\infty$
for some $x,y$.
If $\rho(x,y)=\infty$, we make the convention that $\rho(x,y)^{-s}=0$ for $s>0$ and $\rho(x,y)^{0}=1$.
In the nonsmooth setting, we will usually be considering the special case when $X_1,\ldots, X_q$ span the tangent space at every point of $M$, and in this case
$\rho$ is a metric, and the metric topology induced by $\rho$ is the same as the topology on $M$.

We use ordered multi-index notation: $X^{\alpha}$.  Here, $\alpha$ denotes a list of elements of $\{1,\ldots, q\}$
and $|\alpha|$ denotes the length of the list.  For example, $X^{(2,1,3,1)}=X_2X_1X_3X_1$ and $|(2,1,3,1)|=4$.

Associated to the vector fields $X_1,\ldots, X_q$, we have the following
Banach spaces of functions on $M$.
\begin{equation*}
    \CSpace{M}=\CXjSpace{X}{0}[M]:=\{f:M\rightarrow \C\:\big|\: f\text{ is continuous and bounded}\}, \quad
    \CNorm{f}{M}=\CXjNorm{f}{X}{0}[M]:=\sup_{x\in M} |f(x)|.
\end{equation*}
For $m\in \N$, we define
\begin{equation*}
    \CXjSpace{X}{m}[M]:=\{f\in \CSpace{M}\: \big|\: X^{\alpha} f\text{ exists and } X^\alpha f\in \CSpace{M},\forall |\alpha|\leq m\}, \quad \CXjNorm{f}{X}{m}[M]:=\sum_{|\alpha|\leq m} \CNorm{X^{\alpha} f}{M}.
\end{equation*}
For $s\in [0,1]$, we define the Lipschitz-H\"older space associated to $X$
by
\begin{equation*}
    \HXNorm{f}{X}{0}{s}[M]:=\CNorm{f}{M}+\sup_{\substack{x,y\in M \\ x\ne y}} \rho(x,y)^{-s} |f(x)-f(y)|, \quad \HXSpace{X}{0}{s}[M]:=\{f\in \CSpace{M} : \HXNorm{f}{X}{0}{s}[M]<\infty\}.
\end{equation*}
For $m\in \N$ and $s\in [0,1]$, set
\begin{equation*}
    \HXNorm{f}{X}{m}{s}[M]:=\sum_{|\alpha|\leq m} \HXNorm{X^{\alpha}f}{X}{0}{s}[M], \quad \HXSpace{X}{m}{s}[M]:=\{f\in \CXjSpace{X}{m}[M] : \HXNorm{f}{X}{m}{s}[M]<\infty\}.
\end{equation*}

We turn to the Zygmund-H\"older spaces.  For this, we use the H\"older spaces
$\HSpace{0}{s}[[a,b]]$ for a closed interval $[a,b]\subset \R$; $\HNorm{\cdot}{0}{s}[[a,b]]$
is defined via the formula \cref{Eqn::FuncSpaceRn::DefnHolder}.
Given $h>0$, $s\in (0,1)$ define
\begin{equation*}
    \sP_{X,s}^M(h):=\left\{ \gamma:[0,2h]\rightarrow M\:\bigg|\: \gamma'(t)=\sum_{j=1}^q d_j(t) X_j(\gamma(t)), d_j\in \HSpace{0}{s}[[0,2h]], \sum_{j=1}^q \HNorm{d_j}{0}{s}[[0,2h]]^2<1 \right\}.
\end{equation*}
For $s\in (0,1]$ set
\begin{equation*}
    \ZygXNorm{f}{X}{s}[M]:=\HXNorm{f}{X}{0}{s/2}[M]+\sup_{\substack{h>0 \\ \gamma\in \sP^M_{X,s/2 }(h) }}h^{-s}\left|f(\gamma(2h))-2f(\gamma(h))+f(\gamma(0))\right|,
\end{equation*}
and for $m\in \N$,
\begin{equation*}
    \ZygXNorm{f}{X}{m+s}[M]:=\sum_{|\alpha|\leq m} \ZygXNorm{X^{\alpha}f}{X}{s}[M],
\end{equation*}
and we set
\begin{equation*}
    \ZygXSpace{X}{s+m}[M]:=\{f\in \CXjSpace{X}{m}[M] : \ZygXNorm{f}{X}{m+s}[M]<\infty\}.
\end{equation*}
Set
\begin{equation*}
    \ZygXSpace{X}{\infty}[M]:=\bigcap_{s>0} \ZygXSpace{X}{s}[M]\text{ and } \CXjSpace{X}{\infty}[M]:=\bigcap_{m\in \N} \CXjSpace{X}{m}[M].
\end{equation*}
It is a consequence of \cref{Lemma::FuncSpaces::Inclu} that $\ZygXSpace{X}{\infty}[M]=\CXjSpace{X}{\infty}[M]$;
indeed, $\ZygXSpace{X}{\infty}[M]\subseteq \CXjSpace{X}{\infty}[M]$ is clear while the reverse containment
follows from \cref{Lemma::FuncSpaces::Inclu}.

We introduce the following counter-intuitive, but convenient, definitions.
\begin{defn}\label{Defn::FuncSpacesM::NegativeSpaces}
For $m<0$, $s\in [0,1]$, we define $\HXSpace{X}{m}{s}[M]:=\CSpace{M}$ with equality of norms.
For $s\in (-1,0]$, we define $\ZygXSpace{X}{s}[M]:=\HXSpace{X}{0}{(s+1)/2}[M]$, with equality of norms.
For $s\in (-\infty,-1]$, we define $\ZygXSpace{X}{s}[M]:=\CSpace{M}$ with equality of norms.
\end{defn}

Finally, for $r>0$ we introduce a space of functions which are ``real analytic with respect to $X$''.
\begin{equation*}
    \CXomegaNorm{f}{X}{r}[M]:=\sum_{m=0}^\infty \frac{r^m}{m!} \sum_{|\alpha|=m} \CNorm{X^{\alpha} f}{M}, \quad \CXomegaSpace{X}{r}[M]:=\{f\in \CXjSpace{X}{\infty}[M] : \CXomegaNorm{f}{X}{r}[M]<\infty\}.
\end{equation*}
This definition was introduced in greater generality by Nelson \cite{NelsonAnalyticVectors}.
We set $\CXjSpace{X}{\omega}[M]:=\bigcup_{r>0} \CXomegaSpace{X}{r}[M]$, and for notational convenience set
$\ZygXSpace{X}{\omega}[M]:=\CXjSpace{X}{\omega}[M]$.  We refer the reader to the third paper in the series for a more detailed
discussion of the spaces $\ComegaSpace{r}[\Omega]$ and $\CXomegaSpace{X}{r}[M]$.

Importantly, all of the above spaces are invariant under diffeomorphisms.  In fact, we have the following
result.

\begin{prop}\label{Prop::FuncSpacesM::DiffInv}
Let $N$ be another $C^2$ manifold, let $\Phi:M\rightarrow N$ be a $C^2$ diffeomorphism,
and let $\Phi_{*}X$ denote the list of vector fields $\Phi_* X_1,\ldots, \Phi_* X_q$.
Then the map $f\mapsto f\circ \Phi$ is an isometric isomorphism between the following Banach spaces:
$\CXjSpace{\Phi_* X}{m}[N]\rightarrow \CXjSpace{X}{m}[M]$, $\HXSpace{\Phi_* X}{m}{s}[N]\rightarrow \HXSpace{X}{m}{s}[M]$,  $\ZygXSpace{\Phi_* X}{s}[N]\rightarrow \ZygXSpace{X}{s}[M]$, and $\CXomegaSpace{\Phi_* X}{r}[N]\rightarrow \CXomegaSpace{X}{r}[M]$.
\end{prop}
\begin{proof}This is immediate from the definitions.\end{proof}

\begin{rmk}
Some of the above definitions deserve some additional remarks.
\begin{itemize}
\item In \cref{Eqn::FundSpaceM::DefnCCBall}, $\gamma'(t)$ is defined as follows.
In the case that $M$ is an open subset $\Omega\subseteq \R^n$ and $\gamma:[a,b]\rightarrow \Omega$,
$\gamma'(t)=\sum_{j=1}^q a_j(t) X_j(\gamma(t))$ is defined to mean $\gamma(t)=\gamma(a)+\int_a^t \sum_{j}a_j(s) X_j(\gamma(s))\: ds$; note that this definition is local in $t$.
For an abstract $C^2$ manifold $M$, this is interpreted locally.  I.e., if $\gamma:[a,b]\rightarrow M$, we say $\gamma'(t)=\sum_{j=1}^q a_j(t) X_j(\gamma(t))$ if $\forall t_0\in [a,b]$,
there is an open neighborhood $N$ of $\gamma(t_0)$ and a $C^2$ diffeomorphism $\Psi:N\rightarrow \Omega$, where $\Omega\subseteq \R^n$ is open,
such that $(\Psi\circ \gamma)'(t) = \sum_{j=1}^{q} a_j(t) (\Psi_{*} X_j)(\Psi\circ \gamma(t))$ for $t$ near $t_0$ ($t\in [a,b]$).


\item When we write $V f$ for a $C^1$ vector field $V$ and $f:M\rightarrow \R$, we define this as $Vf(x):=\frac{d}{dt}\big|_{t=0} f(e^{t V} x)$.  When we say $Vf$ exists, it means that this derivative exists in the classical sense, $\forall x$.  If we have several $C^1$ vector fields $V_1, V_2, \ldots, V_L$, we define $V_1 V_2 \cdots V_L f:=V_1(V_2(\cdots V_L(f)))$ and to say that this exists means that at each stage the derivatives exist.


\end{itemize}
\end{rmk}

%% file: funcspacesbeyond.tex

For certain subsets of $M$ which are not themselves manifolds, we can still define the
above norms.  Indeed, let $X_1,\ldots, X_q$ be $C^1$ vector fields on a $C^2$ manifold $M$ and fix $\xi>0$.
In this setting, $B_X(x_0,\xi)$ might not be a manifold (though it sometimes is--see \cref{Prop::ResQual::Mmanif}).
$B_X(x_0,\xi)$ is a metric space, with the metric $\rho$.
For a function $f:B_X(x_0,\xi)\rightarrow \C$ and $x\in B_X(x_0,\xi)$, it makes sense to consider
$X_j f(x):=\frac{d}{dt}\big|_{t=0} f(e^{tX_j }x)$.
Using this, we can define the spaces
$\HXSpace{X}{m}{s}[B_X(x_0,\xi)]$, $\ZygXSpace{X}{s}[B_X(x_0,\xi)]$,
and $\CXomegaSpace{X}{r}[B_X(x_0,\xi)]$
and their corresponding norms, with the same formulas as above.

%% file: series.tex

In this section, we present the main results of this three part series of papers;
though we will offer a more detailed presentation of these results in the later papers.
We separate the results into two parts:  the qualitative results (i.e.,
\cref{Item::Intro::LocalQual,Item::Intro::GlobalQual} from the introduction)
and the quantitative results (i.e., \cref{Item::Intro::Charts}).
The quantitative results are the most useful for applications, and the qualitative results
are simple consequences of the quantitative ones.  
The proofs will not be completed until the later papers--though in this paper
we prove a slightly weaker version of the quantitative results (see \cref{Section::Results}).
We begin by stating
the qualitative results, as they are easier to understand.

%% file: seriesqual.tex

Let $X_1,\ldots, X_q$ be $C^1$ vector fields on a $C^2$ manifold $\fM$.
For $x,y\in \fM$, define $\rho(x,y)$ as in \cref{Eqn::FucsSpaceM::rho}.
Fix $x_0\in \fM$ and let $Z:=\{y\in \fM : \rho(x_0,y)<\infty\}$.
$\rho$ is a metric on $Z$, and we give $Z$ the topology induced by $\rho$ (this
is finer\footnote{See \cref{Lemma::ProofImmerse::FinerTop} for a proof
 that this topology is finer than the subspace topology.} than the topology as a subspace of $\fM$, and may be strictly finer).
Let $M\subseteq Z$ be a connected open subset of $Z$ containing $x_0$.  We
give $M$ the topology of a subspace of $Z$.
We begin with a classical result to set the stage.

\begin{prop}\label{Prop::ResQual::Mmanif}
Suppose $[X_i, X_j]=\sum_{k=1}^q c_{i,j}^k X_k$, where $c_{i,j}^k:M\rightarrow \R$ are locally bounded.
Then, there is a $C^2$ manifold structure on $M$ (compatible with its topology) such that:
\begin{itemize}
\item The inclusion $M\hookrightarrow \fM$ is a $C^2$ injective
immersion.
\item $X_1,\ldots, X_q$ are $C^1$ vector fields tangent to $M$.
\item $X_1,\ldots, X_q$ span the tangent space at every point of $M$.
\end{itemize}
Furthermore, this $C^2$ structure is unique in the sense that if $M$ is given another $C^2$ structure (compatible with its topology)
such that the inclusion map $M\hookrightarrow \fM$ is a $C^2$ injective immersion, then the identity map $M\rightarrow M$ is a $C^2$
diffeomorphmism between these two structures.
\end{prop}
For a proof of \cref{Prop::ResQual::Mmanif} see \cref{Appendix::ProofOfImmerse}.
Henceforth, we assume the conditions of  \cref{Prop::ResQual::Mmanif} so that $M$
is a $C^2$ manifold and $X_1,\ldots, X_q$ are $C^1$ vector fields on $M$ which span the tangent space
at every point.
We write $n=\dim\Span \{X_1(x_0),\ldots, X_q(x_0)\}$, so that $\dim M=n$.

\begin{rmk}
If $X_1(x_0),\ldots, X_q(x_0)$ span $T_{x_0} \fM$, then $M$ is an open submanifold of $\fM$.  If $X_1,\ldots, X_q$ span the tangent space at every point of $\fM$ and $\fM$ is connected,
one may take $M=\fM$.
\end{rmk}

\begin{thm}[The Local Theorem]\label{Thm::Series::LocalQual}
For $s\in (1,\infty] \cup \{\omega\}$, the following three conditions are equivalent:
\begin{enumerate}[(i)]
\item\label{Item::Series::LocalQual::Chart} There is an open neighborhood $V\subseteq M$ of $x_0$ and a $C^2$ diffeomorphism $\Phi:U\rightarrow V$ where
$U\subseteq \R^n$ is open, such that $\Phi^{*}X_1,\ldots, \Phi^{*}X_q\in \ZygSpace{s+1}[U][\R^n]$.

\item\label{Item::Series::LocalQual::Basis}
Re-order the vector fields so that $X_1(x_0),\ldots, X_n(x_0)$ are linearly
independent.
There is an open neighborhood $V\subseteq M$ of $x_0$ such that:
    \begin{itemize}
    \item $[X_i,X_j]=\sum_{k=1}^n \ch_{i,j}^k X_k$, $1\leq i,j\leq n$, where $\ch_{i,j}^k\in \ZygXSpace{X}{s}[V]$.
    \item For $n+1\leq j\leq q$, $X_j=\sum_{k=1}^n b_j^k X_k$, where $b_j^k\in \ZygXSpace{X}{s+1}[V]$.
    \end{itemize}

\item\label{Item::Series::LocalQual::Spanning} There exists an open neighborhood $V\subseteq M$ of $x_0$ such that
$[X_i,X_j]=\sum_{k=1}^q c_{i,j}^k X_k$, $1\leq i,j\leq q$, where $c_{i,j}^k\in \ZygXSpace{X}{s}[V]$.
\end{enumerate}
\end{thm}

\begin{rmk}
\cref{Item::Series::LocalQual::Basis} and \cref{Item::Series::LocalQual::Spanning} of \cref{Thm::Series::LocalQual}
are similar but have slightly different advantages.
In \cref{Item::Series::LocalQual::Basis}, because $X_1,\ldots, X_n$ form a basis for the tangent space of $M$ near $x_0$, the functions
$\ch_{i,j}^k$ and $b_j^k$
are uniquely determined (so long as $V$ is
chosen sufficiently small), and
one can directly check to see if \cref{Item::Series::LocalQual::Basis} holds by computing
these functions.\footnote{The computation can be done in any coordinate system, as the conditions
are invariant under a change of coordinate system--see \cref{Prop::FuncSpacesM::DiffInv}.}
If $q>n$, $X_1,\ldots, X_q$ are linearly
dependent, so the $c_{i,j}^k$ in \cref{Item::Series::LocalQual::Spanning} are not unique--and
\cref{Item::Series::LocalQual::Spanning} only asks that there exists a choice of $c_{i,j}^k$
satisfying the conditions in \cref{Item::Series::LocalQual::Spanning}.
Despite this lack of uniqueness, \cref{Item::Series::LocalQual::Spanning} is the setting which
usually arises in applications.  
\end{rmk}

\begin{rmk}
\Cref{Thm::Series::LocalQual} is stated for $s\in (1,\infty]$.  It is reasonable to expect
the same result for $s\in (0,\infty]$, however our proof runs into some technical issues
when $s\in (0,1]$.  We refer the reader to the second paper for a further discussion of this.
A similar remark holds for \cref{Thm::Series::GlobalQual}, below.
\end{rmk}

\begin{thm}[The Global Theorem]\label{Thm::Series::GlobalQual}
For $s\in (1,\infty]$, the following three conditions are equivalent:
\begin{enumerate}[(i)]
\item\label{Item::Series::Global::Atlas} There exists a $\ZygSpace{s+2}$ atlas on $M$, compatible with its $C^2$ structure,
such that $X_1,\ldots, X_q$ are $\ZygSpace{s+1}$ with respect to this atlas.
\item For each $x_0\in M$, any of the three equivalent conditions \cref{Item::Series::LocalQual::Chart}, \cref{Item::Series::LocalQual::Basis}, or \cref{Item::Series::LocalQual::Spanning} from \cref{Thm::Series::LocalQual} holds for this choice of $x_0$.
\item $[X_i,X_j]=\sum_{k=1}^q c_{i,j}^k X_k$, $1\leq i,j\leq q$, where $\forall x_0\in M$, $\exists V\subseteq M$
open with $x_0\in V$ such that $c_{i,j}^k\big|_V\in \ZygXSpace{X}{s}[V]$, $1\leq i,j,k\leq q$.
\end{enumerate}
Furthermore, under these conditions, the $\ZygSpace{s+2}$ manifold structure on $M$ induced by the atlas in \cref{Item::Series::Global::Atlas}
is unique, in the sense that if there is another $\ZygSpace{s+2}$ atlas on $M$, compatible with its $C^2$
structure, and such that
$X_1,\ldots, X_q$ are $\ZygSpace{s+1}$ with respect to this second atlas, then the identity
map $M\rightarrow M$ is a $\ZygSpace{s+2}$ diffeomorphism between these two $\ZygSpace{s+2}$ manifold structures
on $M$.

Also, the following two conditions are equivalent:
\begin{enumerate}[(a)]
\item\label{Item::Series::Global::Atlas::RA} There is a real analytic atlas on $M$, compatible with its $C^2$ structure, such that $X_1,\ldots, X_q$
are real analytic with respect to this atlas.
\item For each $x_0\in M$, any of the three equivalent conditions \cref{Item::Series::LocalQual::Chart}, \cref{Item::Series::LocalQual::Basis}, or \cref{Item::Series::LocalQual::Spanning}
from \cref{Thm::Series::LocalQual}
hold for this choice of $x_0$ (with $s=\omega$).
\end{enumerate}
Furthermore, under these conditions, the real analytic manifold structure on $M$ induced by the atlas in \cref{Item::Series::Global::Atlas::RA} is unique,
in the sense that if there is another real analytic atlas on $M$, compatible with its $C^2$ structure and such that $X_1,\ldots, X_q$
are real analytic with respect to this second atlas, then the identity map $M\rightarrow M$ is a real analytic
diffeomorphism between these two real analytic structures on $M$.
\end{thm}

%% file: seriesquant.tex

\Cref{Thm::Series::LocalQual} gives necessary and sufficient conditions for a certain type of
coordinate chart to exist.  For applications in analysis, it is essential to have
quantitative control of this coordinate chart.
In the second part to this series, these quantitative charts are studied in the setting of Zygmund spaces,
while in the third part they are studied in the real analytic setting.  In this section,
we present the results on Zygmund spaces, and refer the reader to the third paper for the corresponding
real analytic results.

Because we need to keep track of what each constant depends on for applications in analysis (see \cref{Section::Scale}),
the statements of the results in this section, later in the paper, and in the subsequent papers in this series,
are quite technical.  To help simplify matters, we define various notions of ``admissible constants''.  These will
be constants that can only depend on certain parameters.  While these definitions are somewhat unwieldy,
they greatly simplify the statements of the results in the rest of this series.
In each instance, it will be clear what notion of admissible constants we are using.

First we need some new notation.
$B^n(\eta)$ denotes the Euclidean ball of radius $\eta>0$ centered at $0\in \R^n$.
Let $X_1,\ldots, X_q$ be $C^1$ vector fields on a
$C^2$ manifold $\fM$.

\begin{defn}
For $x_0\in \fM$, $\eta>0$,
and $U\subseteq \fM$, we say the list $X=X_1,\ldots, X_q$ satisfies $\ICond(x_0,\eta,U)$
if for every $a\in B^q(\eta)$ the expression
$$e^{a_1 X_1+\cdots+a_q X_q}x_0$$
exists in $U$.  More precisely, consider the differential equation
\begin{equation*}
    \diff{r} E(r)=a_1 X_1(E(r))+\cdots+ a_q X_q(E(r)), \quad E(0)=x_0.
\end{equation*}
We assume that a solution to this differential equation exists up to $r=1$, $E:[0,1]\rightarrow U$.
We have $E(r) = e^{ra_1 X_1+\cdots+ra_q X_q}x_0$.
\end{defn}

For $1\leq n\leq q$, we let
\begin{equation*}
    \sI(n,q):=\{ (i_1,i_2,\ldots, i_n) : i_j\in \{1,\ldots, q\}\}, \quad \sI_0(n,q):=\{i\in \sI(n,q) : 1\leq i_1< i_2< \cdots < i_n\leq q\}.
\end{equation*}
For $J=(j_1,\ldots, j_n)\in \sI(n,q)$ we write $X_J$ for the list of vector fields $X_{j_1},\ldots, X_{j_n}$.
We write $\bigwedge X_J = X_{j_1}\wedge X_{j_2}\wedge \cdots \wedge X_{j_n}$.

Fix $x_0\in \fM$, let $n=\dim \Span\{X_1(x_0),\ldots, X_q(x_0)\}$.  Fix $\xi,\zeta\in (0,1]$.
We assume that on $B_X(x_0,\xi)$, the $X_j$'s satisfy
\begin{equation*}
    [X_j,X_k]=\sum_{l=1}^q c_{j,k}^l X_l, \quad c_{j,k}^l\in C(B_X(x_0,\xi)),
\end{equation*}
where $B_X(x_0,\xi)$ is given the metric topology induced by $\rho$ from \cref{Eqn::FucsSpaceM::rho}.
\Cref{Prop::ResQual::Mmanif} applies to show that $B_X(x_0,\xi)$ is an $n$-dimensional, $C^2$,
injectively immersed submanifold of $\fM$. $X_1,\ldots, X_q$ are $C^1$ vector fields on $B_X(x_0,\xi)$
and span the tangent space at every point.  Henceforth, we treat $X_1,\ldots, X_q$
as vector fields on $B_X(x_0,\xi)$.

Let $J_0\in \sI(n,q)$ be such that $\bigwedge X_{J_0}(x_0)\ne 0$ and moreover
\begin{equation}\label{Eqn::Series::Quant::J0}
    \max_{J\in \sI(n,q)}\left|\frac{\bigwedge X_J(x_0) }{\bigwedge X_{J_0}(x_0)}\right|\leq \zeta^{-1},
\end{equation}
see \cref{Section::Wedge} for the definition of this quotient.  Note that such a $J_0\in \sI(n,q)$ always exists--indeed, we may choose
$J_0$ so that the left hand side of \cref{Eqn::Series::Quant::J0} equals $1$.
Without loss of generality, reorder the vector fields so that $J_0=(1,\ldots, n)$.

\begin{itemize}
    \item Let $\eta>0$ be such that $X_{J_0}$ satisfies $\sC(x_0,\eta,\fM)$.
    \item Let $\delta_0>0$ be such that for $\delta\in(0,\delta_0]$ the following holds:
    if $z\in B_{X_{J_0}}(x_0,\xi)$ is such that $X_{J_0}$ satisfies
    $\sC(z,\delta, B_{X_{J_0}}(x_0,\xi))$ and if $t\in B^{n}(\delta)$ is such that
    $e^{t_1 X_1+\cdots + t_n X_n} z=z$ and if $X_1(z),\ldots, X_n(z)$ are linearly
    independent, then $t=0$.
\end{itemize}

\begin{rmk}
Using that the vector fields $X_1,\ldots, X_n$ are $C^1$, it follows that there exist $\eta$
and $\delta_0$ as above (which are small depending on, among other things, the $C^1$ norms of $X_1,\ldots, X_n$
in a fixed coordinate system); see \cref{Prop::MoreAssumpt}.  However, it is possible that the $C^1$ norms of $X_1,\ldots, X_q$
can be very large while $\eta$ and $\delta_0$ are not small.  Furthermore, the quantities
$\eta$ and $\delta_0$ are invariant under $C^2$ diffeomorphisms, while the $C^1$ norms
of $X_1,\ldots, X_n$ depend on the choice of coordinate system.  Thus, we present our results
in terms of $\eta$ and $\delta_0$.  
\end{rmk}

\begin{rmk}
For a more detailed discussion of $\eta$ and $\delta_0$ see \cref{Section::Results::MoreAssump}.
\end{rmk}

Fix $s_0>1$.

\begin{defn}
For $s\geq s_0$ if we say $C$ is an $\Zygad{s}$-admissible constant, it means that
we assume $c_{j,k}^l\in \ZygXSpace{X_{J_0}}{s}[B_{X_{J_0}}(x_0,\xi)]$ for $1\leq j,k,l\leq q$.
$C$ is then allowed to depend on $s$, $s_0$, lower bounds $>0$ for $\zeta$, $\xi$, $\eta$, and $\delta_0$,
and upper bounds for $q$ and $\ZygXNorm{c_{j,k}^l}{X_{J_0}}{s}[B_{X_{J_0}}(x_0,\xi)]$, $1\leq j,k,l\leq q$.
We write $A\lesssim_{\Zygad{s}} B$ for $A\leq CB$ where $C$ is a positive $\Zygad{s}$-admissible constant.
We write $A\approx_{\Zygad{s}} B$ for $A\lesssim_{\Zygad{s}} B$ and $B\lesssim_{\Zygad{s}} A$.
\end{defn}

\begin{thm}[The Quantitative Theorem]\label{Thm::Series::MainThm}
Suppose $c_{i,j}^k\in \ZygXSpace{X}{s_0}[B_{X_{J_0}}(x_0,\xi)]$, $1\leq i,j,k\leq q$.  Then,
there exists a map $\Phi:B^n(1)\rightarrow B_{X_{J_0}}(x_0,\xi)$ and $\Zygad{s_0}$-admissible
constants $\xi_1,\xi_2>0$ such that the following hold:
\begin{enumerate}[label=(\roman*),series=overviewtheoremenumeration]
\item $\Phi(B^n(1))\subseteq B_X(x_0,\xi)$ is an open subset of the $C^2$ manifold $B_X(x_0,\xi)$.
\item $\Phi:B^n(1)\rightarrow \Phi(B^n(1))$ is a $C^2$ diffeomorphism.
\item $B_X(x_0,\xi_2)\subseteq B_{X_{J_0}}(x_0,\xi_1)\subseteq \Phi(B^n(1))\subseteq B_X(x_0,\xi)$.
\end{enumerate}
Let $Y_j=\Phi^{*} X_j$.
There exists an $\Zygad{s_0}$-admissible constant $K\approx_{\Zygad{s_0}} 1$ and a matrix
 $A\in \ZygSpace{s_0}[B^n(1)][\M^{n\times n}]$ such that:\footnote{Here, and in the rest of the paper,
 $\M^{n\times n}$ denotes the space of $n\times n$ real matrices endowed with the usual operator norm of a matrix.}
\begin{enumerate}[resume*=overviewtheoremenumeration]
\item $Y_{J_0}=K(I+A)\grad$, where $\grad$ denotes the gradient in $\R^n$ (thought of as a column vector) and we are identifying $Y_{J_0}$ with the column vector of vector fields
$\begin{bmatrix} Y_1,Y_2,\ldots, Y_n\end{bmatrix}^{\transpose}$.
\item $A(0)=0$, $\sup_{t\in B^n(1)} \Norm{A(t)}[\M^{n\times n}]\leq \frac{1}{2}$.
\item\label{Item::Series::YReg} For all $s\geq s_0$, $1\leq j\leq q$, $\ZygNorm{Y_j}{s+1}[B^n(1)][\R^n]\lesssim_{\Zygad{s}} 1$.
\end{enumerate}
\end{thm}

\begin{rmk}
In the second paper, we discuss further details of the map $\Phi$ from \cref{Thm::Series::MainThm}.
For example, we describe how to understand $\Phi^{*}\nu$ where $\nu$ is a density on $B_X(x_0,\xi)$.
\end{rmk}

%% file: seriesdiffinv.tex

The results in this series are invariant under arbitrary $C^2$ diffeomorphisms.
In light of \cref{Prop::FuncSpacesM::DiffInv} this is obvious for the qualitative results (\cref{Thm::Series::LocalQual,Thm::Series::GlobalQual}).
It is true for the quantitative results as well (e.g., \cref{Thm::Series::MainThm}).

Indeed, let $X_1,\ldots, X_q$ be $C^1$ vector fields on a $C^2$ manifold $\fM$, as in \cref{Thm::Series::MainThm}, and fix $x_0\in \fM$.
Let $\Psi:\fM\rightarrow \fN$ be a $C^2$ diffeomorphism.  Then, $X_1,\ldots, X_q$ satisfy the conditions of \cref{Thm::Series::MainThm}
at the point $x_0$ if and only if $\Psi_{*} X_1,\ldots, \Psi_{*} X_q$ satisfy the conditions at $\Psi(x_0)$.
Moreover, $\Zygad{s}$-admissible constants as defined in terms of $X_1,\ldots, X_q$ are the same as $\Zygad{s}$-admissible
constants when defined in terms of $\Psi_{*} X_1,\ldots, \Psi_{*} X_q$.
Finally, if $\Phi$ is the map guaranteed by \cref{Thm::Series::MainThm} when applied to $X_1,\ldots, X_q$, then
$\Psi\circ \Phi$ is the map guaranteed by \cref{Thm::Series::MainThm} when applied to $\Psi_{*} X_1,\ldots, \Psi_{*} X_q$ (as can be seen by tracing through the proof).
The same remarks hold for \cref{Thm::Results::MainThm}, below.

%% file: results.tex

We now turn to the results of this paper, which amount to a slightly weaker
version of \cref{Thm::Series::MainThm}.  We take the same setup as \cref{Thm::Series::MainThm};
so that we have $X_1,\ldots, X_q$, $C^1$ vector fields on a $C^2$ manifold $\fM$.
Fix $x_0\in M$ and set $n=\dim \Span\{X_1(x_0),\ldots, X_q(x_0)\}$.
As before, we assume that on $B_X(x_0,\xi)$, the $X_j$'s satisfy
\begin{equation*}
    [X_j,X_k]=\sum_{l=1}^q c_{j,k}^l X_l, \quad c_{j,k}^l\in C(B_X(x_0,\xi)),
\end{equation*}
where $B_X(x_0,\xi)$ is given the metric topology induced by $\rho$ from \cref{Eqn::FucsSpaceM::rho}.
\Cref{Prop::ResQual::Mmanif} applies to show that $B_X(x_0,\xi)$ is an $n$-dimensional, $C^2$,
injectively immersed submanifold of $\fM$. $X_1,\ldots, X_q$ are $C^2$ vector fields on $B_X(x_0,\xi)$
and span the tangent space at every point.  Henceforth, we treat $X_1,\ldots, X_q$
as vector fields on $B_X(x_0,\xi)$.
Let $J_0\in \sI(n,q)$ be such that $\bigwedge X_{J_0}(x_0)\ne 0$ and moreover
\begin{equation*}
    \max_{J\in \sI(n,q)}\left|\frac{\bigwedge X_J(x_0) }{\bigwedge X_{J_0}(x_0)}\right|\leq \zeta^{-1},
\end{equation*}
see \cref{Section::Wedge} for the definition of this quotient.\footnote{One may always choose $J_0$ so that $\zeta=1$.  However, the flexibility
to take $\zeta<1$ is essential for
some applications.  It will prove to be particularly important when we turn to analogous results in the complex setting in a future
paper.}
Without loss of generality, reorder the vector fields so that $J_0=(1,\ldots, n)$.
Let $\eta,\delta_0>0$ be as in \cref{Section::Series::Quant}.

\begin{defn}
We say $C$ is a $0$-admissible constant if $C$ can be chosen to depend only on
upper bounds for $q$, $\zeta^{-1}$, $\xi^{-1}$, and $\CNorm{c_{j,k}^l}{B_{X_{J_0}}(x_0,\xi)}$, $1\leq j,k,l\leq q$.
\end{defn}

\begin{defn}
If we say $C$ is a $1$-admissible constant, it means that we assume $c_{j,k}^l\in \CXjSpace{X}{1}[B_{X_{J_0}}(x_0,\xi)]$
for $1\leq j,k\leq n$, $1\leq l\leq q$.  $C$ is then allowed to depend on anything a $0$-admissible constant can depend on,
lower bounds $>0$ for $\eta$ and $\delta_0$, and upper bounds for $\CXjNorm{c_{j,k}^l}{X}{1}[B_{X_{J_0}}(x_0,\xi)]$, $1\leq j,k\leq n$, $1\leq l\leq q$.
\end{defn}

\begin{defn}\label{Defn::Results::HXAdmiss}
For $m_1,m_2\in \Z$ and $s\in [0,1]$ if we say $C$ is an $\Had{m_1,m_2,s}$-admissible constant, it means
that we assume:
\begin{itemize}
\item $c_{j,k}^l\in \HXSpace{X_{J_0}}{m_1}{s}[B_{X_{J_0}}(x_0,\xi)]$, $1\leq j,k\leq n$, $1\leq l\leq q$.

\item $c_{j,k}^l\in \HXSpace{X_{J_0}}{m_2}{s}[B_{X_{J_0}}(x_0,\xi)]$, $1\leq j,k,l\leq q$.
\end{itemize}
$C$ can then be chosen to depend only on upper bounds for $m_1$, $m_2$, $q$, $\zeta^{-1}$, $\xi^{-1}$,
$\HXNorm{c_{j,k}^l}{X_{J_0}}{m_1}{s}[B_{X_{J_0}}(x_0,\xi)]$, $1\leq j,k\leq n$, $1\leq l\leq q$,
and $\HXNorm{c_{j,k}^l}{X_{J_0}}{m_2}{s}[B_{X_{J_0}}(x_0,\xi)]$, $1\leq j,k,l\leq q$.
\end{defn}

\begin{defn}
For $s_1,s_2\in \R$ if we say $C$ is an $\Zygad{s_1,s_2}$-admissible constant,
it means that we assume:
\begin{itemize}
\item $c_{j,k}^l\in \ZygXSpace{X_{J_0}}{s_1}[B_{X_{J_0}}(x_0,\xi)]$, $1\leq j,k\leq n$, $1\leq l\leq q$.

\item $c_{j,k}^l\in \ZygXSpace{X_{J_0}}{s_2}[B_{X_{J_0}}(x_0,\xi)]$, $1\leq j,k,l\leq q$.
\end{itemize}
$C$ can then be chosen to depend only on $s_1$, $s_2$ and upper bounds for $q$, $\zeta^{-1}$, $\eta^{-1}$, $\xi^{-1}$,
$\ZygXNorm{c_{j,k}^l}{X_{J_0}}{s_1}[B_{X_{J_0}}(x_0,\xi)]$, $1\leq j,k\leq n$, $1\leq l\leq q$,
and $\ZygXNorm{c_{j,k}^l}{X_{J_0}}{s_2}[B_{X_{J_0}}(x_0,\xi)]$, $1\leq j,k,l\leq q$.
\end{defn}

\begin{rmk}
 $0$ and $1$-admissible constants are the most basic type of admissible constants, and nearly all of our estimates
 depend on those quantities used in $0$-admissible constants, while many depend on the
 stronger $1$-admissible constants.
Admissible constants using the braces $\Had{\cdot}$ are used when working
with estimates relating to H\"older norms, while those using
 $\Zygad{\cdot}$ are used for estimates relating to Zygmund norms. 
In \cref{Section::Densities}, we introduce a density $\nu$ and admissible constants
that take into account this density.  To indicate this, we will decorate the notions of
admissible constants by writing, e.g.,  $\Had{m_1,m_2,s;\nu}$-admissible constants
and $\Zygad{s_1,s_1;\nu}$-admissible constants.
Finally, in \cref{Section::FuncSpaceRev::Compare} we will prove some technical results
for vector fields which are defined on Euclidean space.  To indicate
the corresponding admissible constants, we will use notation like $\HEad{m_1,s}$ and $\ZygEad{s}$,
where $\mathrm{E}$ stands for ``Euclidean''.
\end{rmk}

\begin{rmk}
In the various definitions of admissible constants in this section, we treat $c_{j,k}^l$ differently depending
on whether $1\leq j,k\leq n$ or $1\leq j,k\leq q$.  This is likely an artifact of the proof.
Indeed, this lack of symmetry disappears when we move to the sharp results in the second
paper in the series; see \cref{Thm::Series::MainThm}.
\end{rmk}

We write $A\lesssim_0 B$ for $A\leq CB$ where $C$ is a positive $0$-admissible constant.  We write
$A\approx_0 B$ for $A\lesssim_0 B$ and $B\lesssim_0 A$.
We similarly define $\lesssim_1$, $\approx_1$, $\lesssim_{\Had{m_1,m_2,s}}$, $\approx_{\Had{m_1,m_2,s}}$,
$\lesssim_{\Zygad{s_1,s_2}}$, and $\approx_{\Zygad{s_1,s_2}}$.

Because $X_{J_0}$ satisfies $\sC(x_0,\eta,\fM)$, by hypothesis, we may define the map, for $t\in B^n(\eta)$,
\begin{equation}\label{Eqn::Results::DefinePhi}
    \Phi(t):=e^{t_1 X_1+\cdots + t_n X_n}x_0.
\end{equation}
Let $\eta_0:=\min\{\eta,\xi\}$ so that $\Phi:B^n(\eta_0)\rightarrow B_{X_{J_0}}(x_0,\xi)$.
Note that, a priori, $\Phi$ is $C^1$, since $X_1,\ldots, X_n$ are $C^1$.

\begin{thm}\label{Thm::Results::MainThm}
There exists a $0$-admissible constant $\chi\in (0,\xi]$ such that:
\begin{enumerate}[label=(\alph*),series=maintheoremenumeration]
\item\label{Item::Results::WedgeNonzero} $\forall y\in B_{X_{J_0}}(x_0,\chi)$, $\bigwedge X_{J_0}(y)\ne 0$.
\item\label{Item::Results::BigWedge} $\forall y\in B_{X_{J_0}}(x_0,\chi)$,
$$\max_{J\in \sI(n,q)} \left|\frac{\bigwedge X_J(y)}{\bigwedge X_{J_0}(y)}\right|\approx_0 1.$$
\item\label{Item::Results::Open} $\forall \chi'\in (0,\chi]$, $B_{X_{J_0}}(x_0,\chi')$ is an open subset of
$B_X(x_0,\xi)$ and is therefore a submanifold.
\end{enumerate}
For the rest of the theorem, we assume $c_{j,k}^l\in \CXjSpace{X_{J_0}}{1}[B_{X_{J_0}}(x_0,\xi)]$
for $1\leq j,k\leq n$, $1\leq l\leq q$.
There exist $1$-admissible constants $\eta_1,\xi_1,\xi_2>0$ such that:
\begin{enumerate}[resume*=maintheoremenumeration]
\item $\Phi(B^n(\eta_1))$ is an open subset of $B_{X_{J_0}}(x_0,\chi)$, and is therefore a submanifold of $B_X(x_0,\xi)$.
\item $\Phi:B^n(\eta_1)\rightarrow \Phi(B^n(\eta_1))$ is a $C^2$ diffeomorphism.
\item\label{Item::Results::xi2} $B_X(x_0,\xi_2)\subseteq B_{X_{J_0}}(x_0,\xi_1)\subseteq \Phi(B^n(\eta_1))\subseteq B_{X_{J_0}}(x_0,\chi)\subseteq B_X(x_0,\xi)$.
\end{enumerate}
Let $Y_j=\Phi^{*}X_j$ and write $Y_{J_0}=(I+A)\grad$, where $Y_{J_0}$ denotes the column vector of vector fields
$Y_{J_0}=\begin{bmatrix}Y_1,Y_2,\ldots,Y_n\end{bmatrix}^{\transpose}$, $\grad$ denotes the gradient in $\R^n$
thought of as a column vector, and $A\in \CSpace{B^n(\eta_1)}[\M^{n\times n}]$.
\begin{enumerate}[resume*=maintheoremenumeration]
\item\label{Item::Results::ASize} $A(0)=0$ and $\sup_{t\in B^n(\eta_1)} \Norm{A(t)}[\M^{n\times n}]\leq \frac{1}{2}$.
\item\label{Item::Results::YReg} We have the following regularity on $Y_j$, $1\leq j\leq q$:
\begin{itemize}
\item $\HNorm{Y_j}{m}{s}[B^n(\eta_1)][\R^n]\lesssim_{\Had{m,m-1,s}} 1$, for $m\in \N$, $s\in [0,1]$.
\item $\ZygNorm{Y_j}{s}[B^n(\eta_1)][\R^n]\lesssim_{\Zygad{s,s-1}} 1$, for $s>0$.
\end{itemize}
\item\label{Item::Results::bklReg} 
There exist $b_{k}^l\in \CjSpace{1}[B^n(\eta_1)]$, $n+1\leq k\leq q$, $1\leq l\leq n$, such that
$Y_k=\sum_{l=1}^n b_k^l Y_l$ and
\begin{equation*}
    \HNorm{b_k^l}{m}{s}[B^n(\eta_1)]\lesssim_{\Had{m-1,m-1,s}} 1,\quad m\in \N, s\in [0,1],
\end{equation*}
\begin{equation*}
    \ZygNorm{b_k^l}{s}[B^n(\eta_1)]\lesssim_{\Zygad{s-1,s-1}} 1,\quad s>0.
\end{equation*}
\item\label{Item::Results::cjklReg} For $1\leq j,k\leq n$, $[Y_j,Y_k]=\sum_{l=1}^n \ct_{j,k}^l Y_l$, where
\begin{equation*}
    \HNorm{\ct_{j,k}^l}{m}{s}[B^n(\eta_1)]\lesssim_{\Had{m,m-1,s}} 1, \quad m\in \N, s\in [0,1],
\end{equation*}
\begin{equation*}
    \ZygNorm{\ct_{j,k}^l}{s}[B^n(\eta_1)]\lesssim_{\Zygad{s,s-1}} 1, \quad s>0.
\end{equation*}
\item\label{Item::Results::EquivNorms} We have the following equivalence of norms, for $f\in C(B^n(\eta_1))$,
\begin{itemize}
\item $\HNorm{f}{m}{s}[B^n(\eta_1)]\approx_{\Had{m-1,m-2,s}} \HXNorm{f}{Y_{J_0}}{m}{s}[B^n(\eta_1)]\approx_{\Had{m-1,m-2,s}} \HXNorm{f}{Y}{m}{s}[B^n(\eta_1)]$, for $m\in \N$, $s\in [0,1]$.
\item $\ZygNorm{f}{s}[B^n(\eta_1)]\approx_{\Zygad{s-1,s-2}} \ZygXNorm{f}{Y_{J_0}}{s}[B^n(\eta_1)]\approx_{\Zygad{s-1,s-2}} \ZygXNorm{f}{Y}{s}[B^n(\eta_1)]$, for $s>2$.
\end{itemize}
\item\label{Item::Results::AbstractNorm} We have, for $f\in \CSpace{B_{X_{J_0}}(x_0,\chi)}$,
\begin{itemize}
\item $\HNorm{f\circ \Phi}{m}{s}[B^n(\eta_1)]\lesssim_{\Had{m-1,m-2,s}} \HXNorm{f}{X_{J_0}}{m}{s}[B_{X_{J_0}}(x_0,\chi)]$, $m\in \N$, $s\in [0,1]$.
\item $\ZygNorm{f\circ \Phi}{s}[B^n(\eta_1)]\lesssim_{\Zygad{s-1,s-2}} \ZygXNorm{f}{X_{J_0}}{s}[B_{X_{J_0}}(x_0,\chi)]$, $s\in (0,\infty)$.
\end{itemize}
\end{enumerate}
\end{thm}

\begin{rmk}\label{Rmk::Results::LackOfOptimal}
The lack of optimality of \cref{Thm::Results::MainThm} can be seen by comparing
\cref{Thm::Results::MainThm} \cref{Item::Results::YReg} and \cref{Thm::Series::MainThm} \cref{Item::Series::YReg};
in the later one can estimate $\ZygNorm{Y_j}{s+1}$ in terms of an $\Zygad{s}$-admissible constant,
while in the former, one can only estimate $\ZygNorm{Y_j}{s}$ in terms of the similar $\Zygad{s,s-1}$-admissible constants.  Because of this, \cref{Thm::Results::MainThm} ``loses one derivative'' and is not powerful
enough to conclude necessary and sufficient results like \cref{Thm::Series::LocalQual,Thm::Series::GlobalQual}.
\end{rmk}

\begin{rmk}\label{Rmk::Results::LeverageBetterOptimality}
By comparing \cref{Item::Results::YReg} and \cref{Item::Results::cjklReg}, we see that
the functions $\ct_{j,k}^l$
have the same regularity as $Y_1,\ldots, Y_n$.  If one only knew the regularity of $Y_1,\ldots, Y_n$,
one could only conclude the regularity of $\ct_{j,k}^l$ for one fewer derivative.
Similarly, \cref{Item::Results::bklReg} gives one more derivative regularity on $b_{k}^l$
than we get from merely considering the regularity of $Y_1,\ldots, Y_q$.
In the second paper of this series, we will leverage this extra regularity to prove
\cref{Thm::Series::MainThm}.
\end{rmk}

\begin{rmk}
Because the methods in this paper are based on ODEs, it is possible to prove
versions of \cref{Thm::Results::MainThm} for some function spaces other than $\HSpace{m}{s}$
or $\ZygSpace{s}$, with the same methods as in this paper.
However, once we turn to the second paper in the series, where PDEs are used,
we are forced to work with more specialized spaces--and that is the main
motivation for using Zygmund spaces in this context.
\end{rmk}

\begin{rmk}
In the context of Lie groups, the coordinates  given by $\Phi$ are sometimes
called canonical coordinates of the first kind.
\end{rmk}

%% file: moreassump.tex

We further consider the constants $\eta>0$ and $\delta_0>0$ which were
introduced in \cref{Section::Series::Quant}.  First we present two examples
which show why these constants cannot be dispensed with in our results,
and then we state a result which shows such constants always exist.

\begin{example}
This example demonstrates the importance of $\eta$.  Let $\fM=\R$, $q=1$, $x_0>0$, and let $X_1=x^2\diff{x}$.
In this case, $\eta$ can be taken no larger than $1/x_0$--i.e., $X_1$ satisfies $\sC(x_0,x_0^{-1},\R)$
but does not satisfy $\sC(x_0,\eta', \R)$ for any $\eta'>x_0^{-1}$
(because the ODE $\dot{\gamma}(t)=\gamma(t)^2$, $\gamma(0)=x_0$ exists only for $t<\frac{1}{x_0}$).
If \cref{Thm::Results::MainThm} 
held with constants
independent of $\eta$ 
(and
therefore independent of $x_0$),
then we could conclude that $X_1$ satisfied $\sC(x_0,\eta',\R)$ for some $\eta'$ independent of $x_0$.
This is because the condition $\sC$ is invariant under a change of coordinates, and we can therefore
check it in the coordinate system given by $\Phi$ in \cref{Thm::Results::MainThm}.
This is a contradiction, showing
$\eta$ must play a role
in \cref{Thm::Results::MainThm}.\footnote{For a similar example, one could take $\fM=(-\epsilon,\epsilon)$, $q=1$, $x_0=0$, and $X_1 = \diff{x}$.  Then, $X$
satisfies $\sC(0,\epsilon, (-\epsilon, \epsilon))$, but does not satisfy
$\sC(0, \eta', (-\epsilon, \epsilon))$ for any $\eta'> \epsilon$.}
\end{example}

\begin{example}
This example demonstrates the importance of $\delta_0$--and also
shows its topological nature.  The point of $\delta_0$ is to ensure the map
$\Phi$ in \cref{Thm::Results::MainThm} is injective.\footnote{In fact, by inspecting the proof
of \cref{Thm::Results::MainThm}, it is easy to see that one can prove similar results,
independent of $\delta_0$, so long as one allows $\Phi$ to not be injective.}
Let $\fM=S^1$, $q=1$, $x_0\in S^1$, and let $X_1=K\diff{\theta}$ for some large constant $K$.
For this example, we must take $\delta_0\leq 2\pi /K$.  If the constants in \cref{Thm::Results::MainThm}
did not depend on $\delta_0$, they would also not depend on $K$.  We could then conclude
that $\delta_0$ could be taken independent of $K$--this is because $\delta_0$ is invariant
under a change of coordinates and we can check it in the coordinate system given by $\Phi$
in \cref{Thm::Results::MainThm}--see also \cref{Prop::MoreAssumpt}.
This shows that $\delta_0$ must play a role in \cref{Thm::Results::MainThm}.
\end{example}

Now we state a result which shows that such a $\delta_0$ and $\eta$ always exist for $C^1$ vector fields.
Let $X_1,\ldots, X_q$ be $C^1$ vector fields on a $C^2$ manifold $\fM$,
and let $X$ denote the list $X_1,\ldots, X_q$.

\begin{prop}\label{Prop::MoreAssumpt}
\begin{itemize}
    \item $\forall x_0\in \fM$, $\exists \eta>0$, such that $X$ satisfies $\sC(x_0,\eta,\fM)$.
    \item Let $K\Subset \fM$ be a compact set.  Then, $\exists \delta_0>0$ such that
    $\forall \theta\in S^{q-1}$ if $x\in K$ is such that $\theta_1 X_1(x)+\cdots+ \theta_qX_q(x)\ne 0$,
    then $\forall r\in (0,\delta_0]$,
    \begin{equation*}
        e^{r\theta_1 X_1+\cdots + r\theta_q X_q}x\ne x.
    \end{equation*}
\end{itemize}
\end{prop}

For the proof, see \cref{Section::Proofs::MoreAssump}.  \Cref{Prop::MoreAssumpt} shows that there always exist $\eta$ and $\delta_0$
as in \cref{Section::Series::Quant}.  However, the $\eta$ and $\delta_0$
guaranteed by \cref{Prop::MoreAssumpt} depend on the $C^1$ norms of $X_1,\ldots, X_q$
in some fixed coordinate system, and this is not invariant under diffeomorphisms.
It is important for some applications that $\eta$ and $\delta_0$
can be taken to be large in some settings even when the $C^1$ norms of $X_1,\ldots, X_q$
are large.  The next example gives a simple setting where this is the case.

\begin{example}
Take $q=1$, $\fM=\R$, $X_1=K\diff{x}$, for any $K\in \R\setminus \{0\}$ (we think of $K$ as large).
Then one can take $\eta=\delta_0=\infty$ in the assumptions in \cref{Section::Series::Quant}.
\end{example}

%% file: wedge.tex

Let $Z$ be a one dimensional real vector space.
For $x,y\in Z$, $x\ne 0$ we define $\frac{y}{x}\in \R$ by
$\frac{y}{x}:=\frac{\lambda(y)}{\lambda(x)}$ where $\lambda:Z\rightarrow \R$
is any nonzero linear functional.  It is easy to see that $\frac{y}{x}$ is independent of the
choice of $\lambda$.

This allows us to formulate a ``coordinate free'' version of Cramer's rule.
Let $V$ be an $n$-dimensional vector space, so that $\bigwedge^n V$ is a one dimensional
vector space.  Let $x_1,\ldots, x_n\in V$ be a basis for $V$.  For any $y\in V$, we have
\begin{equation}\label{Eqn::Wedge::Cramer}
    y = \frac{y\wedge x_2\wedge x_3\wedge \cdots \wedge x_n}{x_1\wedge x_2\wedge \cdots \wedge x_n} x_1 + \frac{x_1\wedge y\wedge x_3\wedge \cdots \wedge x_n}{x_1\wedge x_2\wedge \cdots \wedge x_n} x_2+\cdots + \frac{x_1\wedge x_2\wedge \cdots \wedge x_{n-1} \wedge y}{x_1\wedge x_2\wedge \cdots \wedge x_n} x_n.
\end{equation}

Let $M$ be a $C^2$ manifold of dimension $n$.  Let $Y_1,\ldots, Y_n$ be $C^1$ vector fields in on $M$.
For another $C^1$ vector field $Z$, the Lie derivative of $Y_1\wedge Y_2\wedge \cdots \wedge Y_n$
with respect $Z$ is given by
\begin{equation*}
    \Lie{Z}(Y_1\wedge Y_2\wedge \cdots \wedge Y_n) = [Z,Y_1]\wedge Y_2\wedge Y_3\wedge \cdots\wedge Y_n
    + Y_1\wedge [Z,Y_2]\wedge Y_3\wedge \cdots \wedge Y_n +\cdots + Y_1\wedge Y_2\wedge \cdots \wedge Y_{n-1}\wedge [Z,Y_n].
\end{equation*}
Let $X_1,\ldots, X_n$ be $C^1$ vector fields on $M$ which span the tangent space near a point $x_0$.
Thus, near $x_0$, we may define a real valued function by
\begin{equation*}
    \frac{Y_1\wedge Y_2\wedge \cdots \wedge Y_n}{X_1\wedge X_2\wedge \cdots\wedge X_n}.
\end{equation*}
The derivative of this function with respect to $Z$ is exactly what one would expect as the
next lemma shows.
\begin{lemma}\label{Lemma::Wedge::DerivFrac}
$$Z \frac{Y_1\wedge Y_2\wedge \cdots \wedge Y_n}{X_1\wedge X_2\wedge \cdots\wedge X_n}
=\frac{\Lie{Z}(Y_1\wedge Y_2\wedge \cdots \wedge Y_n)}{X_1\wedge X_2\wedge \cdots \wedge X_n}
-\frac{Y_1\wedge Y_2\wedge \cdots \wedge Y_n}{X_1\wedge X_2\wedge \cdots \wedge X_n} \frac{\Lie{Z}(X_1\wedge X_2\wedge \cdots \wedge X_n)}{X_1\wedge X_2\wedge \cdots \wedge X_n}.$$
\end{lemma}
\begin{proof}
Let $\fX=X_1\wedge X_2\wedge \cdots \wedge X_n$ and $\fY=Y_1\wedge Y_2\wedge \cdots \wedge Y_n$.
Let $\nu$ be any $C^1$ $n$-form which is nonzero near $x_0$, so that by definition
\begin{equation*}
    \frac{\fY}{\fX}=\frac{\nu(\fY)}{\nu(\fX)}.
\end{equation*}
Because $\nu$ is nonzero near $x_0$ (and the space of $n$-forms is one dimensional at each point), we may write $\Lie{Z}\nu = f\nu$ for some continuous function $f$ (near $x_0$); where here and in the rest of the paper
$\Lie{Z}$ denotes the Lie derivative with respect to $Z$.
Using \cite[Proposition 18.9]{LeeIntroToSmoothManifolds}, we have
\begin{equation*}
    Z \nu(\fY) = (\Lie{Z}\nu)(\fY) + \nu(\Lie{Z}\fY)
    = f\nu(\fY) + \nu(\Lie{Z}\fY).
\end{equation*}
and similarly with $\fY$ replaced by $\fX$.
We conclude
\begin{equation*}
\begin{split}
    Z\frac{\fY}{\fX}=Z\frac{\nu(\fY)}{\nu(\fX)}
    &=\frac{Z\nu(\fY)}{\nu(\fX)}
    -\frac{\nu(\fY)}{\nu(\fX)} \frac{Z\nu(\fX)}{\nu(\fX)}
    =\frac{f\nu(\fY)+\nu(\Lie{Z}\fY)}{\nu(\fX)} - \frac{\nu(\fY)}{\nu(\fX)}\frac{f\nu(\fX)+\nu(\Lie{Z}\fX) }{\nu(\fX)}
    \\&=\frac{\nu(\Lie{Z}\fY)}{\nu(\fX)} - \frac{\nu(\fY)}{\nu(\fX)} \frac{\nu(\Lie{Z}\fX)}{\nu(\fX)}
    =\frac{\Lie{Z}\fY}{\fX} - \frac{\fY}{\fX} \frac{\Lie{Z}\fX}{\fX},
\end{split}
\end{equation*}
completing the proof.
\end{proof}

%% file: densities.tex

Let $\chi\in (0,\xi]$ be as in \cref{Thm::Results::MainThm}.
In many applications, one is given a density on $B_{X_{J_0}}(x_0,\chi)$ and it is of interest
to measure certain sets with respect to this density.  For a quick introduction
to the basics of densities, we refer the reader to Guillemin's lecture notes \cite{GuilleminNotes}.

Let $\nu$ be a $C^1$ density on $B_{X_{J_0}}(x_0,\chi)$.  Suppose
\begin{equation}\label{Eqn::Desnity::Defnfj}
\Lie{X_j} \nu = f_j \nu, \quad 1\leq j\leq n, \quad f_j\in \CSpace{B_{X_{J_0}}(x_0,\chi)}.
\end{equation}
Our goal is to understand $\Phi^{*}\nu$ and $\nu(B_{X}(x_0,\xi_2))$ where $\Phi$ and $\xi_2$
are as in \cref{Thm::Results::MainThm}.

\begin{defn}
We say $C$ is a $0;\nu$-admissible constant if $C$ is a $0$-admissible constant which is also
allowed to depend on upper bounds for $\CNorm{f_j}{B_{X_{J_0}}(x_0,\chi)}$, $1\leq j\leq n$.
\end{defn}

\begin{defn}\label{Defn::Densities::1nuAdmissible}
We say $C$ is a $1;\nu$-admissible constant if $C$ is a $1$-admissible constant
which is also allowed to depend on upper bounds for $\CNorm{f_j}{B_{X_{J_0}}(x_0,\chi)}$, $1\leq j\leq n$.
\end{defn}

\begin{defn}
For $m_1,m_2\in \Z$, $s\in [0,1]$ if we say $C$ is an $\Had{m_1,m_2,s;\nu}$-admissible constant,
it means that we assume $f_j\in \HXSpace{X_{J_0}}{m_1}{s}[B_{X_{J_0}}(x_0,\chi)]$, and $C$
is an $\Had{m_1,m_2,s}$-admissible constant which is also allowed to depend on
upper bounds for $\HXNorm{f_j}{X_{J_0}}{m_1}{s}[B_{X_{J_0}}(x_0,\chi)]$, $1\leq j\leq n$.
\end{defn}

\begin{defn}
For $s_1>0$, $s_2\in \R$, if we say $C$ is an $\Zygad{s_1,s_2;\nu}$-admissible constant, it means
that we assume $f_j\in \ZygXSpace{X_{J_0}}{s_1}[B_{X_{J_0}}(x_0,\chi)]$, and $C$
is an $\Zygad{s_1,s_2}$-admissible constant which is also allowed to depend on
upper bounds for $\ZygXNorm{f_j}{X_{J_0}}{s_1}[B_{X_{J_0}}(x_0,\chi)]$, $1\leq j\leq n$.
For $s_1\leq 0$, $s_2\in \R$, if we say $C$ is an $\Zygad{s_1,s_2;\nu}$-admissible constant, it means
 $C$
is an $\Zygad{s_1,s_2}$-admissible constant which is also allowed to depend on
upper bounds for $\CNorm{f_j}{B_{X_{J_0}}(x_0,\chi)}$, $1\leq j\leq n$.
\end{defn}

We write $A\lesssim_{0;\nu} B$ for $A\leq C B$ where $C$ is a positive $0;\nu$-admissible constant,
and write $A \approx_{0;\nu} B$ for $A\lesssim_{0;\nu} B$ and $B\lesssim_{0;\nu} A$.
We define $\lesssim_{1;\nu}$, $\approx_{1;\nu}$, $\lesssim_{\Had{m_1,m_2,s;\nu}}$, $\approx_{\Had{m_1,m_2,s;\nu}}$, $\lesssim_{\Zygad{s_1,s_2;\nu}}$,
and $\approx_{\Zygad{s_1,s_2;\nu}}$ similarly.

To help understand $\nu$, we use a distinguished density $\nu_0$ on $B_{X_{J_0}}(x_0,\chi)$:
\begin{equation}\label{Eqn::Density::Defnnu0}
    \nu_0(Z_1,\ldots, Z_n) := \left|\frac{Z_1\wedge Z_2\wedge \cdots \wedge Z_n}{X_1\wedge X_2\wedge \cdots \wedge X_n}\right|,
\end{equation}
note that $\nu_0$ is defined since $X_1\wedge X_2\wedge \cdots \wedge X_n$ is never zero on $B_{X_{J_0}}(x_0,\chi)$ by \cref{Thm::Results::MainThm} \cref{Item::Results::WedgeNonzero}; $\nu_0$ is clearly a density.

\begin{thm}\label{Thm::Density::MainThm}
There exists $g\in \CSpace{B_{X_{J_0}}(x_0,\chi)}$ such that
$\nu = g\nu_0$ and
\begin{enumerate}[label=(\roman*),series=densitytheoremenumeration]
\item\label{Item::Desnity::gconst} $g(x)\approx_{0;\nu} g(x_0)=\nu(X_1,\ldots, X_n)(x_0)$, $\forall x\in B_{X_{J_0}}(x_0,\chi)$.  In particular, $g$ always has the same sign,
and is either never zero or always zero.
\item\label{Item::Desnity::greg} We have the following regularity on $g$:
\begin{itemize}
\item For $m\in \N$, $s\in [0,1]$, we have $\HXNorm{g}{X_{J_0}}{m}{s}[B_{X_{J_0}}(x_0,\chi)]\lesssim_{\Had{m-1,m-1,s;\nu}} |\nu(X_1,\ldots, X_n)(x_0)|$.
\item For $s>0$,  we have $\ZygXNorm{g}{X_{J_0}}{s}[B_{X_{J_0}}(x_0,\chi)]\lesssim_{\Zygad{s-1,s-1;\nu}} |\nu(X_1,\ldots, X_n)(x_0)|$.
\end{itemize}
\end{enumerate}
Define $h\in \CjSpace{1}[B^n(\eta_1)]$ by $\Phi^{*}\nu=h \LebDensity$, where $\LebDensity$ denotes
the usual Lebesgue density on $\R^n$.
\begin{enumerate}[resume*=densitytheoremenumeration]
\item\label{Item::Density::hconst} $h(t)\approx_{0;\nu} \nu(X_1,\ldots, X_n)(x_0)$, $\forall t\in B^n(\eta_1)$.  In particular, $h$
always has the same sign and is either never zero or always zero.
\item\label{Item::Density::hReg} We have the following regularity on $h$:
\begin{itemize}
\item For $m\in \N$, $s\in [0,1]$, $\HNorm{h}{m}{s}[B^{n}(\eta_1)]\lesssim_{\Had{m,m-1,s;\nu}} |\nu(X_1,\ldots, X_n)(x_0)|$.
\item For $s>0$, $\ZygNorm{h}{s}[B^n(\eta_1)]\lesssim_{\Zygad{s,s-1;\nu}} |\nu(X_1,\ldots, X_n)(x_0)|$.
\end{itemize}
\end{enumerate}
\end{thm}

\begin{cor}\label{Cor::Desnity::MeasureSets}
Let $\xi_2$ be as in \cref{Thm::Results::MainThm}.  Then,
\begin{equation}\label{Cor::Density::ToShowMeasure1}
\nu(B_{X_{J_0}}(x_0,\xi_2))\approx_{1;\nu} \nu(B_X(x_0,\xi_2))\approx_{1;\nu} \nu(X_1,\ldots, X_n)(x_0),
\end{equation}
and therefore,
\begin{equation}\label{Cor::Density::ToShowMeasure2}
|\nu(B_{X_{J_0}}(x_0,\xi_2))|\approx_{1;\nu} |\nu(B_X(x_0,\xi_2))|\approx_{1;\nu} |\nu(X_1,\ldots, X_n)(x_0)|
\approx_{0} \max_{(j_1,\ldots, j_n)\in \sI(n,q)} |\nu(X_{j_1},\ldots, X_{j_n})(x_0)|.
\end{equation}
\end{cor}

%% file: scale.tex

The main results of this series have two facets:
\begin{itemize}
\item (Smoothness) They provide a coordinate system in which given $C^1$ vector fields have an optimal degree of smoothness.
\item (Scaling) They provide a coordinate system in which given vector fields are normalized in a way which is useful for applying techniques from analysis.
\end{itemize}
In both cases, the results are in many ways optimal:  they provide necessary and sufficient, diffeomorphic invariant conditions under which one can obtain such coordinate charts.
In this section, we describe these two facets.

When viewed as providing a coordinate system in which vector fields have an optimal level of smoothness, these results seem to be of a new type.
When viewed as scaling maps, these results
take their roots in the quantitative study of sub-Riemannian
(aka Carnot-Carath\'eodory) geometry initiated by Nagel, Stein, and Wainger
\cite{NagelSteinWaingerBallsAndMetrics}.  Since Nagel, Stein, and Wainger's original work,
these ideas have had a significant impact on various questions in harmonic analysis
(see the discussion at the end of Chapter 2 of \cite{StreetMultiParamSingInt}
for a detailed history of these ideas).
Following Nagel, Stein, and Wainger's work, Tao and Wright \cite{TaoWrightLpImproving}
generalized Nagel, Stein, and Wainger's ideas and provided a new approach to proving their results.
In \cite{S}, the second author combined these two approaches to prove results in more
general settings; these more general results have already had several applications, for example
\cite{SteinStreetA,SteinStreetI,SteinStreetII,SteinStreetIII,SteinStreetS,StreetMultiParamSingInt,GressmanScalarOscillatoryIntegrals,StovallUniformLpImprovingForWeightedAveragesOnCurves}.

%% file: nsw.tex

In this section, we describe the foundational work of Nagel, Stein, and Wainger
\cite{NagelSteinWaingerBallsAndMetrics}, and see how it is a special case
of \cref{Thm::Results::MainThm}.  This provides the simplest non-trivial setting where the results
in this paper can be seen as providing scaling maps adapted to a sub-Riemannian geometry.
In \cref{Section::Scale::GenSubR}, we generalize these results to more general geometries.

Let $X_1,\ldots, X_q$ be $C^\infty$ vector fields on an open set $\Omega\subseteq \R^n$;
we assume $X_1,\ldots, X_q$ span the tangent space at every point of $\Omega$.
To each $X_j$ assign a formal degree $d_j\in [1,\infty)$.
We assume
\begin{equation}\label{Eqn::NSW::MainAssump}
    [X_j,X_k]=\sum_{d_l\leq d_j+d_k} c_{j,k}^l X_l, \quad c_{j,k}^l\in C^\infty(\Omega).
\end{equation}

We write $(X,d)$ for the list $(X_1,d_1),\ldots, (X_q,d_q)$ and for $\delta>0$
write $\delta^dX$ for the list of vector fields $\delta^{d_1} X_1,\ldots, \delta^{d_q} X_q$.
The sub-Riemannian ball associated to $(X,d)$ centered at $x_0\in \Omega$
of radius $\delta>0$ is defined by
\begin{equation*}
    B_{(X,d)}(x_0,\delta):=B_{\delta^d X}(x_0,1),
\end{equation*}
where the later ball is defined by \cref{Eqn::FundSpaceM::DefnCCBall}.  $B_{(X,d)}(x_0,\delta)$
is an open subset of $\Omega$.  It is easy to see that the balls $B_{(X,d)}(x,\delta)$ are metric balls.



Define, for $x\in \Omega$, $\delta\in (0,1]$,
\begin{equation*}
	\Lambda(x,\delta):=\max_{j_1,\ldots, j_n\in \{1,\ldots, q\}}  \left|\det \left( \delta^{d_{j_1}} X_{j_1}(x) | \cdots | \delta^{d_{j_n}} X_{j_n}(x)\right)\right|.
\end{equation*}
For each $x\in \Omega$, $\delta\in (0,1]$, pick $j_1=j_1(x,\delta),\ldots, j_n=j_n(x,\delta)$ so that
\begin{equation*}
	\left|\det \left( \delta^{d_{j_1}} X_{j_1}(x) | \cdots | \delta^{d_{j_n}} X_{j_n}(x)\right)\right|=\Lambda(x,\delta).
\end{equation*}
For this choice of $j_1=j_1(x,\delta),\ldots, j_n=j_n(x,\delta)$,
set
\begin{equation*}
    \Phi_{x,\delta}(t_1,\ldots, t_n):=\exp\left(t_1 \delta^{d_{j_1}} X_{j_1}+ \cdots + t_n \delta^{d_{j_n}} X_{j_n}\right)x.
\end{equation*}

\begin{thm}[\cite{NagelSteinWaingerBallsAndMetrics}]\label{Thm::NSW::MainThm}
Fix a compact set $\sK\Subset \Omega$.\footnote{Here, and in the rest of the paper, we write $\sK\Subset \Omega$ to mean that $\sK$ is a relatively compact subset of $\Omega$.}  
In what follows, we write $A\lesssim B$ for $A\leq CB$ where $C$ is a positive constant which may depend on $\sK$,
but does not depend on the particular point $x\in \sK$ or the scale $\delta\in (0,1]$.
There exist $\eta_1,\xi_0\approx 1$, such that $\forall x\in \sK$,
\begin{enumerate}[(i)]
\item\label{Item::NSW::EstVol} $\LebDensity(B_{(X,d)}(x,\delta))\approx \Lambda(x,\delta)$, $\forall \delta\in (0,\xi_0]$.
\item\label{Item::NSW::Doubling} $\LebDensity(B_{(X,d)}(x,2\delta))\lesssim \LebDensity(B_{(X,d)}(x,\delta))$, $\forall \delta\in (0,\xi_0/2]$.
\item $\forall \delta\in (0,1]$, $\Phi_{x,\delta}(B^n(\eta_1))\subseteq \Omega$ is open and $\Phi_{x,\delta}:B^n(\eta_1)\rightarrow \Phi_{x,\delta}(B^n(\eta_1))$ is a $C^\infty$ diffeomorphism.
\item\label{Item::NSW::Jacobian} $|\det d\Phi_{x,\delta}(t)|\approx \Lambda(x,\delta)$, $\forall t\in B^n(\eta_1)$.
\item\label{Item::NSW::Image} $B_{(X,d)}(x,\xi_0\delta)\subseteq \Phi_{x,\delta}(B^n(\eta_1))\subseteq B_{(X,d)}(x,\delta)$, $\forall \delta\in (0,1]$.
\item\label{Item::NSW::Pullback} Let $Y_j^{x,\delta}:=\Phi^{*}_{x,\delta} \delta^{d_j} X_j$, so that $Y_j^{x,\delta}$ is a $C^\infty$ vector field
on $B^n(\eta_1)$.  We have
\begin{equation*}
    \BCjNorm{Y_j^{x,\delta}}{m}[B^n(\eta_1);\R^n]\lesssim 1,\quad \forall m\in \N,
\end{equation*}
where the implicit constant depends on $m$, by not on $x\in \sK$ or $\delta\in (0,1]$.   Finally, $Y_1^{x,\delta}(u),\ldots, Y_q^{x,\delta}(u)$ span
$T_uB^n(\eta_1)$,
uniformly in $x$, $\delta$, and $u$, in the sense that
\begin{equation*}
    \max_{j_1,\ldots, j_n\in \{1,\ldots, q\}} \inf_{u\in B^n(\eta_1)} \left|\det \left(Y_{j_1}^{x,\delta}(u)| \cdots | Y_{j_n}^{x,\delta}(u)\right)\right|\approx 1.
\end{equation*}
\end{enumerate}
\end{thm}
\begin{proof}
This result is a special case of \cref{Thm::GenSubr::MainThm}, below.  To see this, 
for $\delta\in (0,1]$ we multiply both sides of \cref{Eqn::NSW::MainAssump} by $\delta^{d_j+d_k}$ to
obtain
\begin{equation*}
    [\delta^{d_j} X_j, \delta^{d_k}X_k]=\sum_{d_l\leq d_j+d_k} \delta^{d_j+d_k-d_l} c_{j,k}^l \delta^{d_l} X_l,
\end{equation*}
so that if we set
\begin{equation*}
    X_j^{\delta}:=\delta^{d_j} X_j, \quad c_{j,k}^{l,\delta}:=\begin{cases}
    \delta^{d_j+d_k-d_l} c_{j,k}^l, & d_l\leq d_j+d_k,\\
    0,&\text{otherwise},
    \end{cases}
\end{equation*}
then we have
\begin{equation*}
    [X_j^{\delta}, X_k^{\delta}] = \sum_{l} c_{j,k}^{l,\delta} X_l^{\delta}.
\end{equation*}
Furthermore, $c_{j,k}^{l,\delta}\in C^\infty$ and $X_l^{\delta}\in C^\infty$ \textit{uniformly in }$\delta$.
From here it is straightforward to verify that $X_1^{\delta},\ldots, X_q^{\delta}$ satisfy all the hypotheses of  \cref{Thm::GenSubr::MainThm};
in the application of \cref{Thm::GenSubr::MainThm}, we replace $\Omega$ with $\Omega'$ where $\sK\Subset \Omega'\Subset \Omega$.
\end{proof}

\begin{rmk}
It is easy to see that the balls $B_{(X,d)}(x,\delta)$ are metric balls.\footnote{This uses that
$d_j\geq 1$, $\forall j$.  If $d_j\in (0,\infty)$, they are quasi-metric balls.}
\Cref{Thm::NSW::MainThm} \cref{Item::NSW::Doubling} is the main estimate
needed to show these balls (when paired with $\LebDensity$) form a space
of homogeneous type.  Thus, one can obtain a theory of singular integrals
associated with these balls.  Such singular integrals have a long history
and have proven to be quite useful in a variety of contexts.
The history of these ideas is detailed at the end of
\cite[Chapter 2]{StreetMultiParamSingInt}.
\end{rmk}


%% file: hormander.tex

The main way that \cref{Thm::NSW::MainThm} arises is via vector fields which satisfy
H\"ormander's condition.  Suppose $V_1,\ldots, V_r$ are $C^\infty$ vector fields
on an open set $\Omega\subseteq \R^n$.  We assume that $V_1,\ldots, V_r$ satisfy
H\"ormander's condition of order $m$ on $\Omega$.  I.e., we assume that the finite
list of vector fields
\begin{equation*}
    V_1,\ldots, V_r,\ldots, [V_i,V_j],\ldots, [V_i,[V_j,V_k]],\ldots,\ldots, \text{commutators of order }m,
\end{equation*}
span the tangent space at every point of $\Omega$.

To each $V_1,\ldots, V_r$, we assign the formal degree $1$.  If $Z$ has formal degree $e$,
we assign to $[V_j,Z]$ the formal degree $e+1$.  Let $(X_1,d_1),\ldots, (X_q,d_q)$
denote the finite list of vector fields with formal degree $d_j\leq m$.
H\"ormander's condition implies $X_1,\ldots, X_q$ span the tangent space at every point of $\Omega$.

We claim that \cref{Eqn::NSW::MainAssump} holds, and therefore \cref{Thm::NSW::MainThm}
applies to $(X_1,d_1),\ldots, (X_q,d_q)$.  Indeed, if $d_j+d_k\leq m$ we have
\begin{equation*}
    [X_j,X_k]=\sum_{d_l= d_j+d_k} c_{j,k}^l X_l,
\end{equation*}
where $c_{j,k}^l$ are constants by the Jacobi identity.  If $d_j+d_k>m$ then,
since $X_1,\ldots, X_q$ span the tangent space at every point, we have
\begin{equation*}
    [X_j,X_k]=\sum_{l=1}^q c_{j,k}^l X_l=\sum_{d_l\leq d_j+d_k} c_{j,k}^l X_l, \quad c_{j,k}^l\in \CjSpace{\infty}[\Omega].
\end{equation*}
Thus,  \cref{Eqn::NSW::MainAssump} holds and \cref{Thm::NSW::MainThm} applies.

Let $\sK\Subset \Omega$ be a compact set.  Applying \cref{Thm::NSW::MainThm}, for $\delta\in (0,1]$, $x\in \sK$, we obtain
$\eta_1>0$ and $\Phi_{x,\delta}:B^n(\eta_1)\rightarrow B_{(X,d)}(x,\delta)$ as in that theorem.
Set $V_{j}^{x,\delta}:=\Phi_{x,\delta}^{*} \delta V_j$, $1\leq j\leq r$.

If $d_k=l$, then
\begin{equation*}
    X_k=[V_{j_1}, [V_{j_2},\cdots,[V_{j_{l-1}}, V_{j_l}]\cdots]],
\end{equation*}
and so
\begin{equation*}
    \Phi_{x,\delta}^{*} \delta^{d_k}X_k =\Phi_{x,\delta}^{*} [\delta V_{j_1},[\delta V_{j_2},\ldots, [\delta V_{j_{l-1}}, \delta V_{j_l}]\ldots]]=[V_{j_1}^{x,\delta},[V_{j_2}^{x,\delta},\ldots,[V_{j_{l-1}}^{x,\delta},V_{j_l}^{x,\delta}]\ldots]].
\end{equation*}
\Cref{Thm::NSW::MainThm} implies that the vector fields $\Phi_{x,\delta}^{*} \delta^{d_k} X_k$ are smooth and
span the tangent space, uniformly for $x\in \sK$, $\delta\in (0,1]$.
We conclude that the vector fields $V_1^{x,\delta},\ldots, V_r^{x,\delta}$ are smooth
and satisfy H\"ormander's condition, uniformly for $x\in \sK$, $\delta\in (0,1]$.
In short, the map $\Phi_{x,\delta}^{*}$ takes $\delta V_1,\ldots, \delta V_r$
to $V_1^{x,\delta},\ldots, V_r^{x,\delta}$ which satisfy H\"ormander's condition ``uniformly'';
i.e., it takes the case of $\delta$ small and ``rescales'' it to the case $\delta=1$.

\begin{rmk}
In the above, we multiplied $V_1,\ldots, V_r$ all by the same small number $\delta$.
Similar results hold (with the same proofs) for $\delta_1 V_1,\ldots, \delta_r V_r$
where $\delta_1,\ldots, \delta_r$ are small, provided they are ``weakly-comparable.''
I.e., provided $\exists N,\kappa$ such that
$\delta_j^N\leq \kappa \delta_k$, for all $j,k$.
This was first noted and used by Tao and Wright \cite{TaoWrightLpImproving}.
See \cite[Section 5.2.1]{S} for further details.
\end{rmk}

\begin{rmk}
It is possible for \cref{Eqn::NSW::MainAssump} to hold (for a sufficiently large $m$) even if $V_1,\ldots, V_r$ do not
satisfy H\"ormander's condition.  In this case, with the same proof one can obtain similar results; however,
now the ball $B_{(X,d)}(x,\delta)$ lies on an injectively immersed submanifold of $\R^n$ as discussed in
\cref{Prop::ResQual::Mmanif}.
An important setting where this arises is when $V_1,\ldots, V_r$ are real analytic; see \cite[Section 2.15.5]{StreetMultiParamSingInt} for details.
\end{rmk}

%% file: mulitparam.tex

In a generalization of the work of Nagel, Stein, and Wainger, the second author studied multi-parameter
sub-Riemannian balls in \cite{S}.  The main result of \cite{S} is a special case of 
\cref{Thm::Results::MainThm,Thm::Density::MainThm,Cor::Desnity::MeasureSets}.
We refer the reader to \cite{S} for the detailed assumptions used in that paper,
which are very similar to the assumptions of \cref{Thm::Results::MainThm}.
We give a few comments here to help the reader understand how the main
result of \cite{S} (namely \cite[Theorem 4.1]{S}) is a special case
of the results in this paper.

The main differences between \cite[Theorem 4.1]{S} and the setting of this paper are:
\begin{itemize}
    \item $\fM$ is taken to be an open subset of $\R^N$ in \cite{S}.
    \item In \cite{S}, the various kinds of admissible constants are allowed to depend on upper
    bounds for quantities like $\CjNorm{X_j}{m}$.  This quantity is not invariant under diffeomorphisms,
    and the norm is defined in terms of the fixed standard coordinate system on $\R^N$.
    \item Instead of an abstract density as is used in \cref{Thm::Density::MainThm,Cor::Desnity::MeasureSets},
    \cite{S} uses the usual Lebesgue measure on submanifolds of $\R^N$.
    \item In \cite{S}, the existence of $\delta_0$ is not assumed.  Instead, one uses bounds on
    $\CjNorm{X_j}{1}$ to prove that such a $\delta_0$ exists (as in \cref{Prop::MoreAssumpt}).  This
    process is not invariant under diffeomorphisms.
    \item The constants in \cref{Thm::Results::MainThm} have better dependence on various
    quantities than they do in \cite[Theorem 4.1]{S}.  For example, the methods
    in \cite{S} do not imply that $\eta_1$ is a $1$-admissible constant.
    \item In \cite{S}, only the spaces $\CXjSpace{X}{m}$ (and not $\HXSpace{X}{m}{s}$ or $\ZygXSpace{X}{s}$)
    were used.
\end{itemize}

We include a lemma, whose straightforward proof we omit, which will allow the reader
to more easily translate the results of \cite{S} into the language of this paper.
For an $N\times n$ matrix we write $\det_{n\times n} B$ to be the vector consisting
of determinants of $n\times n$ submatricies of $B$.

\begin{lemma}
Let $L$ be an $n$-dimensional injectively immersed submanifold of $\R^N$, and give $L$
the induced Riemannian metric.  Let $\nu$ denote the Riemannian volume density on $L$.
For vector fields $Z_1,\ldots, Z_n$ on $\R^N$ which are tangent to $L$, let $Z$ denote the $N\times n$ matrix
whose columns are $Z_1,\ldots, Z_n$.  Then,
\begin{equation*}
    \left|\det_{n\times n} Z\right| = \nu(Z_1,\ldots, Z_n).
\end{equation*}
Furthermore, if $\Phi:B^n(\eta)\rightarrow L\subseteq \R^N$, and if $\Phi^{*} \nu=h(t)\LebDensity$,
then we have
\begin{equation*}
    h(t)=\left|\det_{n\times n} d\Phi(t)\right|,
\end{equation*}
where $d\Phi(t)$ is computed by thinking of $\Phi$ as a map $B^n(\eta)\rightarrow \R^N$.
\end{lemma}

Using this lemma and the above remarks, \cite[Theorem 4.1]{S} follows easily from
the results in this paper.  We refer the reader to \cite{S,StreetMultiParamSingInt,SteinStreetA,SteinStreetI,SteinStreetII,SteinStreetIII,SteinStreetS}
for examples of how these ideas can be used as scaling maps.

%% file: gensubr.tex
The results described in \cref{Section::Scale::NSW} concern the classical setting of sub-Riemannian geometry.  When applied to partial differential equations
defined by vector fields, this is the geometry which arises in the important case of \textit{maximally hypoelliptic} operators.
Maximal hypoellipticity is a far reaching generalization of ellipticity, which was first introduced (implicitly) by
Folland and Stein \cite{FollandSteinEstimatesForTheDbarComplex}; see
\cite[Chapter 2]{StreetMultiParamSingInt} for a discussion of these ideas as well as a detailed history.
When one moves beyond the setting of maximal hypoellipticity, other more general sub-Riemannian geometries can arise.
These are defined by choosing different vector fields at each scale.  A particularly transparent setting where this arises
is in the work of Charpentier and Dupain on the Bergman and Szeg\"o projections \cite{CharpentierDupainExtremalBases}.
The theory in this paper allows us to easily understand what properties one requires on these vector fields so that the induced quasi-metrics
give rise to a space of homogeneous type; furthermore, our theory provides generalized scaling maps adapted to these geometries.
See \cref{Section::Scale::SCV} for some further comments on the relationship between
 the results in this paper and several complex variables.

Fix an open set $\Omega\subseteq \R^n$, and for each $\delta\in (0,1]$, let $X^{\delta} = X_1^{\delta},\ldots, X_q^{\delta}$ be a list of $C^1$ vector fields on $\Omega$,
which span the tangent space at every point.  For $x\in \Omega$, $\delta\in (0,1]$ set $B(x,\delta):=B_{X^{\delta}}(x,1)$, where $B_{X^{\delta}}(x,1)$
is defined by \cref{Eqn::FundSpaceM::DefnCCBall}.  Our goal is to give conditions on $X^{\delta}$ so that the balls $B(x,\delta)$, when paired
with Lebesgue measure on $\Omega$ (denoted $\LebDensity$), locally form a space of homogeneous type (see \cite{SteinHarmonicAnalysis} for the definition
we are using
of a space of homogeneous type).  The conditions we give can be thought of as infinitesimal versions of the axioms of a space of homogeneous type.
In what follows, we write $X^{\delta}$ for the column vector of vector fields $[X_1^{\delta},\ldots, X_q^{\delta}]^{\transpose}$.
Because of this, if we are given a matrix $A:\Omega\rightarrow \M^{q\times q}$, it makes sense to consider
$A(x)X^{\delta}(x)$ which again gives a column vector of vector fields on $\Omega$.

We assume:
\begin{enumerate}[(I)]
\item\label{Item::GenSubr::Assume::Span} $\forall \delta\in (0,1]$, $x\in \Omega$, we have $\Span\{X_1^{\delta}(x),\ldots, X_q^{\delta}(x)\} = T_x\Omega$.
\item\label{Item::GenSubr::Assume::C1} $\sup_{\delta\in (0,1]} \CjNorm{X_j^{\delta}}{1}[\Omega][\R^n]<\infty$.
\item\label{Item::GenSubr::Assume::Decreasing} $X_j^{\delta}\rightarrow 0$, as $\delta\rightarrow 0$, uniformly on compact subsets of $\Omega$.
\item\label{Item::GenSubr::Assume::Containment} $\forall 0<\delta_1\leq \delta_2\leq 1$, $X^{\delta_1}=T_{\delta_1,\delta_2} X^{\delta_2}$, where $T_{\delta_1,\delta_2}\in L^\infty(\Omega;\M^{q\times q})$,
and $\Norm{T_{\delta_1,\delta_2}}[L^\infty(\Omega;\M^{q\times q})]\leq 1$.
\item\label{Item::GenSubr::Assume::Engulfing} $\exists B_1, B_2\in (1,\infty)$, $b_1,b_2\in (0,1)$, such that $\forall \delta\in (0,1/B_1]$, $\exists S_\delta\in L^\infty(\Omega;\M^{q\times q})$
and $\forall \delta\in (0,1/B_2]$, $\exists R_\delta\in L^\infty(\Omega;\M^{q\times q})$ with
$S_\delta X^{B_1 \delta} = X^{\delta}$, $R_{\delta} X^{\delta} = X^{B_2\delta}$, and
$$\sup_{0<\delta\leq 1/B_1} \Norm{S_{\delta}}[L^\infty(\Omega;\M^{q\times q})]\leq b_1, \quad \sup_{0<\delta\leq 1/B_2} \Norm{R_{\delta}}[L^\infty(\Omega;\M^{q\times q})]\leq b_2^{-1}.$$
\item\label{Item::GenSubr::Assume::Smooth} $\forall \delta\in (0,1]$, $[X_j^{\delta}, X_k^{\delta}] =\sum_{l=1}^q c_{j,k}^{l,\delta} X_l^{\delta}$, where $c_{j,k}^{l,\delta}\in \CSpace{\Omega}$ and $\forall m\in \N$
\begin{equation*}
\sup_{\delta\in (0,1],x\in \Omega} \CXjNorm{c_{j,k}^{l,\delta}}{X^{\delta}}{m}[B(x,\delta)]<\infty.
\end{equation*}
\end{enumerate}

Define, for $x\in \Omega$, $\delta\in (0,1]$,
\begin{equation*}
\Lambda(x,\delta):= \max_{j_1,\ldots, j_n\in \{1,\ldots, q\}} \left|\det \left(X_{j_1}^{\delta}(x)| \cdots | X_{j_n}^{\delta}(x)\right)\right|.
\end{equation*}
For each $x\in \Omega$, $\delta\in (0,1]$, pick $j_1=j_1(x,\delta),\ldots, j_n=j_n(x,\delta)\in \{1,\ldots, q\}$ so that
\begin{equation*}
\left|\det \left(X_{j_1}^{\delta}(x)| \cdots | X_{j_n}^{\delta}(x)\right)\right| = \Lambda(x,\delta),
\end{equation*}
and set (for this choice of $j_1=j_1(x,\delta),\ldots, j_n=j_n(x,\delta)$),
\begin{equation*}
\Phi_{x,\delta}(t_1,\ldots, t_n) = \exp\left(t_1 X_{j_1}^\delta + \cdots + t_n X_{j_n}^{\delta}\right)x.
\end{equation*}

\begin{thm}\label{Thm::GenSubr::MainThm}
\begin{enumerate}[label=(\roman*),series=gensubrtheoremenumeration]
\item\label{Item::GenSubr::Containment} $B(x,\delta_1)\subseteq B(x,\delta_2)$, $\forall x\in \Omega$, $0<\delta_1\leq \delta_2\leq 1$.
\item\label{Item::GenSubr::Decreasing} $\bigcap_{\delta\in (0,1]} \overline{B(x,\delta)} = \{x\}$, $\forall x\in \Omega$.
\item\label{Item::GenSubr::Engulfing} $B(x,\delta)\cap B(y,\delta)\ne \emptyset\Rightarrow B(y,\delta)\subseteq B(x,C\delta)$, $\forall \delta\in (0,1/C]$,
where $C=B_1^k$ and $k$ is chosen so that $b_1^k\leq \frac{1}{3}$.
\item\label{Item::GenSubr::Cont} For each $U\Subset \Omega$ with $U$ open, $\delta\in (0,1]$, the map $x\mapsto \LebDensity(U\cap B(x,\delta))$ is continuous.
\end{enumerate}
Fix a compact set $\sK\Subset \Omega$.  In what follows we write $A\lesssim B$ for $A\leq CB$ where $C$ is a positive constant which  may depend on $\sK$,
but does not depend on the particular point $x\in \sK$ or the scale $\delta\in (0,1]$.  We write $A\approx B$ for $A\lesssim B$ and $B\lesssim A$.
There exist $\eta_1, \xi_0\approx 1$ such that $\forall x\in \sK$:
\begin{enumerate}[resume*=gensubrtheoremenumeration]
\item\label{Item::GenSubr::VolEst} $\LebDensity(B(x,\delta))\approx \Lambda(x,\delta)$, $\forall \delta\in (0,\xi_0]$.
\item\label{Item::GenSubr::Doubling} $\LebDensity(B(x,2\delta))\lesssim \LebDensity(B(x,\delta))$, $\forall \delta\in (0,\xi_0/2]$.
\item\label{Item::GenSubr::Diffeo} $\forall \delta\in (0,1]$, $\Phi_{x,\delta}(B^n(\eta_1))\subseteq \Omega$ is open and $\Phi_{x,\delta}:B^n(\eta_1)\rightarrow \Phi_{x,\delta}(B^n(\eta_1))$ is a $C^2$ diffeomorphism.
\item\label{Item::GenSubr::Jacobian} $\left|\det d\Phi_{x,\delta}(t)\right|\approx \Lambda(x,\delta)$, $\forall t\in B^n(\eta_1)$, $\delta\in (0,1]$.
\item\label{Item::GenSubr::Image} $B(x,\xi_0\delta)\subseteq \Phi_{x,\delta}(B^n(\eta_1))\subseteq B(x,\delta)$, $\forall \delta\in (0,1]$.
\item\label{Item::GenSubr::Pullback} Let $Y_j^{x,\delta}:=\Phi_{x,\delta}^{*} X_j^{\delta}$, so that $Y_j^{x,\delta}$ is a vector field on $B^n(\eta_1)$.
Then $Y_j^{x,\delta}\in \CjSpace{\infty}[B^n(\eta_1)][\R^n]$ and
\begin{equation*}
\CjNorm{Y_j^{x,\delta}}{m}[B^n(\eta_1)][\R^n]\lesssim 1,\quad \forall m\in \N,
\end{equation*}
where the implicit constant may depend on $m$, but does not depend on $x\in \sK$ or $\delta\in (0,1]$.  Furthermore, $Y_1^{x,\delta}(u),\ldots, Y_q^{x,\delta}(u)$
span $T_uB^n(\eta_1)$, uniformly in $x$, $\delta$, and $u$ in the sense that
\begin{equation*}
\max_{j_1,\ldots, j_n\in \{1,\ldots, q\}} \inf_{u\in B^n(\eta_1)} \left| \det \left( Y_{j_1}^{x,\delta}(u) | \cdots | Y_{j_n}^{x,\delta}(u)\right) \right|\approx 1.
\end{equation*}
\end{enumerate}
\end{thm}
\begin{proof}
To facilitate the proof, we introduce some new notation.  For $y\in \Omega$, $y\in B(x,\delta)=B_{X^{\delta}}(x,1)$ if and only if
$\exists \gamma:[0,1]\rightarrow \Omega$, $\gamma(0)=x$, $\gamma(1)=y$, $\gamma'(t)=\InnerProduct{a(t)}{X^{\delta}(\gamma(t))}$,
where $a\in L^\infty([0,1];\R^q)$ with $\Norm{a}[L^\infty([0,1];\R^q)]<1$, we have identified $X^{\delta}$ with the vector of vector fields
$X^\delta=(X^{\delta}_1,\ldots, X^{\delta}_q)$, and $\InnerProduct{\cdot}{\cdot}$ denotes the usual inner product on $\R^q$.

\Cref{Item::GenSubr::Containment}:  Let $0<\delta_1\leq \delta_2\leq 1$.  Take $y\in B(x,\delta_1)$ so that
$\exists \gamma:[0,1]\rightarrow \Omega$, $\gamma(0)=x$, $\gamma(1)=y$,
$\gamma'(t)=\InnerProduct{a(t)}{X^{\delta_1}(\gamma(t))}$, $\Norm{a}[L^\infty([0,1];\R^q)]<1$.
We have
\begin{equation*}
\gamma'(t) = \InnerProduct{a(t)}{X^{\delta_1}(\gamma(t))} = \InnerProduct{a(t)}{T_{\delta_1,\delta_2}(\gamma(t))X^{\delta_2}(\gamma(t))}
= \InnerProduct{T_{\delta_1,\delta_2}(\gamma(t))^{\transpose} a(t)}{X^{\delta_2}(\gamma(t))}.
\end{equation*}
Since $\Norm{T_{\delta_1,\delta_2}(\gamma(t))^{\transpose} a}[L^\infty([0,1];\R^q)]\leq \Norm{a}[L^\infty([0,1];\R^q)]<1$, this proves
$y\in B(x,\delta_2)$, completing the proof of \cref{Item::GenSubr::Containment}.

\Cref{Item::GenSubr::Decreasing} follows from the hypothesis \cref{Item::GenSubr::Assume::Decreasing}.

\Cref{Item::GenSubr::Engulfing}:  Suppose $B(x,\delta)\cap B(y,\delta)\ne \emptyset$.
This is equivalent to $B_{X^{\delta}}(x,1)\cap B_{X^{\delta}}(y,1)\ne \emptyset$.  Since the balls $B_{X^{\delta}}(x,\cdot)$ are metric balls,
this implies $B(y,\delta)=B_{X^{\delta}}(y,1)\subseteq B_{X^{\delta}}(x,3)$.  Thus it suffices to show $B_{X^{\delta}}(x,3)\subseteq B(x,C\delta)$.
Suppose $z\in B_{X^{\delta}}(x,3)$, so that $\exists \gamma:[0,1]\rightarrow \Omega$, $\gamma(0)=x$, $\gamma(1)=z$,
$\gamma'(t) = \InnerProduct{a(t)}{3X^{\delta}(\gamma(t))}$, where $\Norm{a}[L^\infty([0,1];\R^q)]<1$.

Take $k$ so large that $b_1^{k}\leq \frac{1}{3}$.  Then, for $\delta\in (0,B_1^{-k}]$,
$\gamma'(t)=\InnerProduct{a(t)}{A(t)X^{B_1^k \delta}(\gamma(t))} = \InnerProduct{A(t)^{\transpose} a(t)}{X^{B_1^k \delta}(\gamma(t)}$,
where
\begin{equation*}
A(t) = 3 S_{\delta}(\gamma(t)) S_{B_1 \delta}(\gamma(t)  ) \cdots S_{B_1^{k-1} \delta}(\gamma(t)).
\end{equation*}
Since $\Norm{A}[L^\infty([0,1];\M^{q\times q})]\leq 3b_1^k\leq 1$, it follows that $\Norm{A^{\transpose} a}[L^\infty([0,1];\R^q)]\leq \Norm{a}[L^\infty([0,1];\R^q)]<1$,
and therefore $z=\gamma(1)\in B(x,B_1^k \delta)=B(x,C\delta)$, completing the proof of \cref{Item::GenSubr::Engulfing}.

\Cref{Item::GenSubr::Cont} follows from standard ODE results.

For the remaining parts, the goal is to apply \cref{Thm::Results::MainThm,Thm::Density::MainThm,Cor::Desnity::MeasureSets}
to the list of vector fields $X^{\delta}$ (with $\nu=\LebDensity$ and $\xi=1$).
Take $\eta\in (0,1]$, depending on $\sK$ and upper bounds for $\CjNorm{X_j^{\delta}}{1}[\Omega]$, so that
$\forall x\in \sK$, $X_1^{\delta},\ldots, X_q^{\delta}$ satisfy $\sC(x,\eta,\Omega)$.  Note that $\eta$ can be chosen
independent of $x\in \sK$ and $\delta\in (0,1]$.
Take $\delta_0>0$ as in \cref{Prop::MoreAssumpt} when applied to $X_1^{\delta},\ldots, X_q^{\delta}$, with $\fM=\Omega$.
It can be seen from the proof of \cref{Prop::MoreAssumpt} that $\delta_0$ can be chosen independent of $\delta\in (0,1]$.
Finally, note that $\sL_{X_j^{\delta}} \nu = \Div(X_j^{\delta}) \nu =:f_j^{\delta} \nu$, where $\sup_{\delta\in (0,1]} \CNorm{f_j^\delta}{\Omega}<\infty$.

Using the above choices, all of the hypotheses of \cref{Thm::Results::MainThm,Thm::Density::MainThm,Cor::Desnity::MeasureSets}
hold for $x_0\in \sK$ with $X_1,\ldots, X_q$ replaced by $X_1^{\delta},\ldots, X_q^{\delta}$, uniformly for $\delta\in (0,1]$, $x_0\in \sK$.
In particular, any constant which is admissible (of any kind) in the sense of those results is $\approx 1$ in the sense
of this theorem (when working with $\nu$, we only use $1;\nu$-admissible constants--see \cref{Defn::Densities::1nuAdmissible} for the definition of $1;\nu$-admissible constants).

\Cref{Item::GenSubr::Diffeo} is contained in \cref{Thm::Results::MainThm}.

\Cref{Item::GenSubr::Image}:  \Cref{Thm::Results::MainThm} gives $\xi_2\approx 1$ ($\xi_2<1$) such that
\begin{equation*}
B_{X^{\delta}}(x,\xi_2)\subseteq \Phi_{x,\delta}(B^n(\eta_1))\subseteq B_{X^{\delta}}(x,1)=B(x,\delta).
\end{equation*}
Thus, to prove \cref{Item::GenSubr::Image}, we wish to show $\exists \xi_0\approx 1$ with
\begin{equation}\label{Eqn::GenSubr::CompareBalls}
B(x,\xi_0\delta)\subseteq B_{X^{d}}(x,\xi_2).
\end{equation}
Take $k\approx 1$ so large that $b_1^k\leq \xi_2$
and set $\xi_0=B_1^{-k}$.  Let $y\in B(x,\xi_0\delta)$, so that there exists $\gamma:[0,1]\rightarrow \Omega$,
$\gamma(0)=x$, $\gamma(1)=y$, $\gamma'(t)=\InnerProduct{a(t)}{X^{\xi_0\delta}(\gamma(t))}$,
with $\Norm{a}[L^\infty]<1$.
Then,
\begin{equation*}
\gamma'(t)=\InnerProduct{a(t)}{\xi_2 A(t) X^{\delta}(\gamma(t))} = \InnerProduct{A(t)^{\transpose}a(t)}{\xi_2 X^{\delta}(\gamma(t))},
\end{equation*}
where $A(t) = \xi_2^{-1} S_{\xi_0\delta}(\gamma(t))S_{\xi_0 B_1\delta}(\gamma(t))\cdots S_{\xi_0 B_1^{k-1}\delta}(\gamma(t))$;
note that $\Norm{A}[L^\infty([0,1];\M^{q\times q})]\leq 1$, and therefore,
$\Norm{A^{\transpose}a}[L^\infty([0,1];\R^q)]\leq \Norm{a}[L^\infty([0,1];\R^q)]<1$.  It follows that
$y=\gamma(1)\in B_{X^{\delta}}(x,\xi_2)$, completing the proof of \cref{Item::GenSubr::Image}.

We claim, for $\delta_1\leq \delta_2$,
\begin{equation}\label{Eqn::GenSubr::LambdaIncrease}
	\Lambda(x,\delta_1)\lesssim \Lambda(x,\delta_2),
\end{equation}
where the implicit constant can be chosen to depend only on $q$.  Indeed,
\begin{equation*}
\begin{split}
	\Lambda(x,\delta_1) &= \max_{j_1,\ldots, j_n\in \{1,\ldots, q\}} \left| \det \left(X_{j_1}^{\delta_1}(x)| \cdots | X_{j_n}^{\delta_1}(x)\right) \right|
\\&= \max_{j_1,\ldots, j_n\in \{1,\ldots, q\}} \left| \det \left((T_{\delta_1,\delta_2}X^{\delta_2})_{j_1}(x)| \cdots | (T_{\delta_1,\delta_2}X^{\delta_2})_{j_n}(x)\right) \right|.
\end{split}
\end{equation*}
Since $\Norm{T_{\delta_1,\delta_2}(x)}\leq 1$, the right hand side is the determinant of a matrix whose columns are linear combinations (with coefficients bounded by $1$)
of the vectors $X_1^{\delta_2}(x),\ldots, X_q^{\delta_2}(x)$.  \Cref{Eqn::GenSubr::LambdaIncrease} follows.

Next we claim, for $c>0$ fixed,
\begin{equation}\label{Eqn::GenSubr::ScaleLambda}	
	\Lambda(x,c\delta)\approx \Lambda(x,\delta),\quad \delta, c\delta\in (0,1],
\end{equation}
where the implicit constant depends on $c$.  It suffices to prove \cref{Eqn::GenSubr::ScaleLambda} for $c<1$.
By \cref{Eqn::GenSubr::LambdaIncrease}, it suffices to prove \cref{Eqn::GenSubr::ScaleLambda}  for
$c=B_2^{-k}$ for some $k$.  We have
\begin{equation}\label{Eqn::GenSubr::WriteLambdaAgain}
\begin{split}
\Lambda(x,\delta) &= \max_{j_1,\ldots, j_n\in \{1,\ldots, q\}} \left|\det \left( X^{\delta}_{j_1}(x) | \cdots | X^{\delta}_{j_n}(x) \right)\right|
\\& = \max_{j_1,\ldots, j_n\in \{1,\ldots, q\}} \left|\det \left( (AX^{c\delta})_{j_1}(x) | \cdots | (AX^{c\delta})_{j_n}(x) \right)\right|,
\end{split}
\end{equation}
where $A(x) = R_{B_2^{-1}\delta}(x) R_{B_2^{-2}\delta}(x)\cdots R_{B_2^{-k}\delta}(x)$.  Since $\sup_{x\in \Omega}\Norm{A(x)}[\M^{q\times q}]\leq b_2^{-k}\lesssim 1$,
it follows that the right hand side of \cref{Eqn::GenSubr::WriteLambdaAgain} is the determinant of a matrix whose
columns are linear combinations (with coefficients whose magnitudes  are $\lesssim 1$)
of the vectors $X_1^{c\delta}(x),\ldots, X_q^{c\delta}(x)$.  It follows that $\Lambda(x, \delta)\lesssim \Lambda(x,c\delta)$.
Combining this with \cref{Eqn::GenSubr::LambdaIncrease}, \cref{Eqn::GenSubr::ScaleLambda} follows.

\Cref{Cor::Desnity::MeasureSets} shows 
\begin{equation}\label{Eqn::GenSubr::CorVolConc}
	\LebDensity(B_{X^{\delta}}(x,\xi_2))\approx \Lambda(x,\delta),
\end{equation}
where we have used that (thinking of $\LebDensity$ as a density) $\LebDensity(V_1(x),\ldots, V_n(x)) = \left|\det(V_1(x) | \cdots| V_n(x))\right|$.
Combining this with \cref{Eqn::GenSubr::ScaleLambda} and \cref{Eqn::GenSubr::CompareBalls}, we have
\begin{equation}\label{Eqn::GenSubr::EstVol::1}
	\LebDensity(B(x,\xi_0\delta))\leq \LebDensity(B_{X^{\delta}}(x,\xi_2))\approx \Lambda(x,\delta) \approx \Lambda(x,\xi_0\delta).
\end{equation}
Conversely, using \cref{Eqn::GenSubr::CorVolConc} again we have,
\begin{equation}\label{Eqn::GenSubrEstVol::2}
\Lambda(x,\delta)\approx \LebDensity(B_{X^{\delta}}(x,\xi_2)) \leq \LebDensity(B_{X^{\delta}}(x,1)) = \LebDensity(B(x,\delta)).
\end{equation}
Combining \cref{Eqn::GenSubr::EstVol::1} and \cref{Eqn::GenSubrEstVol::2} proves \cref{Item::GenSubr::VolEst}.
\Cref{Item::GenSubr::Doubling} follows from  \cref{Item::GenSubr::VolEst} and \cref{Eqn::GenSubr::ScaleLambda}.

Since $\Phi_{x,\delta}^{*}\LebDensity = |\det d\Phi_{x,\delta}| \LebDensity$, \cref{Item::GenSubr::Jacobian} follows from
\cref{Thm::Density::MainThm} \cref{Item::Density::hconst} and \cref{Cor::Desnity::MeasureSets}.
\Cref{Item::GenSubr::Pullback} follows directly from \cref{Thm::Results::MainThm}.
\end{proof}

\begin{rmk}
One can generalize the multi-parameter geometries from \cref{Section::Scale::MultiParam} in a similar way by changing the above variable $\delta\in (0,1]$
to a \textit{vector}, $\delta\in [0,1]^{\nu}$ for some $\nu\in \N$, and proceeding in a a similar way.
\end{rmk}

\begin{rmk}
The most artificial hypothesis in this section is \cref{Item::GenSubr::Assume::C1}.  Indeed, it is not directly related to any of the hypotheses of a space of homogeneous type.
This hypothesis can be replaced with weaker hypotheses and we can still achieve the same result.  In fact, the main purposes of \cref{Item::GenSubr::Assume::C1}
are to ensure the existence of $\eta$ and $\delta_0$ (independent of $x\in \sK$, $\delta\in (0,1]$) in our application of \cref{Thm::Results::MainThm}, and to estimate
$\Lie{X_j^{\delta}} \LebDensity$.  One could just directly assume
the existence of such constants and estimates, or assume any number of other hypotheses which imply their existence, depending on the application at hand.
\end{rmk}

%% file: nonsmooth.tex
An important way in which the results in this paper are stronger than previously mentioned works is that the statements of the main thoerems are completely invariant
under $C^2$ diffeomorphisms (see \cref{Section::Series::DiffInv}).  This is true quantitatively:  all of the estimates
depend on quantities which are invariant under arbitrary $C^2$ diffeomorphisms.
In previous works like \cite{NagelSteinWaingerBallsAndMetrics,TaoWrightLpImproving,S,MontanariMorbidelliNonsmoothHormanderVectorFields} the estimates
were in terms of $C^m$ type norms of the vector fields in some fixed coordinate system.\footnote{\cite{MontanariMorbidelliNonsmoothHormanderVectorFields}
works with Lipschitz vector fields to obtain some results with less regularity than the other mentioned works.
  It is possible that the ideas from that paper could be combined with the ideas from this paper to prove
 results like the ones in this paper, but with Lipschitz vector fields instead of $C^1$ vector fields; though we do not pursue this here.}
Thus, the vector fields had to be a priori ``smooth'' and ``not large'' in some fixed coordinate system.
The concepts of ``smooth'' and ``not large'' are not invariant under $C^2$ diffeomorphisms.
Under the assumptions of \cref{Thm::Results::MainThm}, we conclude the existence
of a coordinate system in which the vector fields are smooth and not large, but we need not
assume it.
This allows us to address some settings where the vector fields are given in a coordinate
system in which they are large and/or are merely $C^1$; in particular, unlike previous works, we only use the qualitative assumption that the vector fields are $C^1$, and our estimates
do not depend on the $C^1$ norms of the coefficients in a coordinate system.

When considering only the smoothness (and not the size) aspect of this diffeomorphism invariance,
these results can be rephrased as the qualitative results in \cref{Section::Series::Qual};
the methods from previous works on this subject cannot yield such theorems, since they require the vector fields
to be smooth in the first place.  In fact, the qualitative results in this series seem to be of a new type;
though there may be some connection to Hilbert's fifth problem.

In the series of papers \cite{SteinStreetA,SteinStreetI,SteinStreetII,SteinStreetIII,SteinStreetS}, the second author and Stein
used the scaling techniques from \cite{S} to study singular Radon transforms of the form
\begin{equation*}
Tf(x) = \psi(x)\int f(\gamma(t,x))K(t)\: dt,
\end{equation*}
where $\gamma(t,x)$ is a germ of a smooth function defined near $(0,0)$, $\gamma(t,x):\R^N_0\times \R^n_0\rightarrow \R^n$
with $\gamma(0,x)\equiv x$ (we have used $\R_0^m$ to denote a small neighborhood of $0\in \R^m)$, and $K(t)$ is a ``multi-parameter singular kernel'' supported near $0\in \R^N$.  Conditions were given so that the above operator was bounded on $L^p$.  Because the theory was based on \cite{S}, it was required that $\gamma(t,x)$ be smooth
and supported very near $(0,0)$.  One could replace every application of the results from \cite{S} in these papers
with \cref{Thm::Results::MainThm} to obtain more general results where $\gamma$ is not necessarily required to be smooth or supported very close to $0$.
In fact, the results can be made completely invariant under arbitrary $C^2$ diffeomorphisms, and so the concepts of smooth and small do not have intrinsic meaning.
Similar remarks hold for many other settings where methods from  \cite{NagelSteinWaingerBallsAndMetrics,TaoWrightLpImproving,S} are used.

Large sub-Riemannian balls have been studied in some special cases before.
See, for example,
the discussion of model pseudoconvex boundaries in \cite[Section 4]{NagelSteinDifferentiableControl}
as well as \cite{PetersonCCMetricsInUnboudnedSubDomains,DlugiePetersonOnUniformLargeScale}.
The approach in this paper allows us to unify the ideas behind these large sub-Riemannian balls
with the more robust theory of small sub-Riemannian balls.


%% file: scv.tex
As described in \cref{Section::Scale::GenSubR}, the results in this paper can be used to study generalized
versions of sub-Riemannian geometries, and as elucidated by Charpentier and Dupain \cite{CharpentierDupainExtremalBases},
these geometries arise when studying $\overline{\partial}$-problems.
When applying the results from this series to such questions, a difficulty arises.
We turn to describing this issue, and how it will be addressed in a future work of the second author.

Let $M$ be a complex manifold of dimension $n$, and for each $\delta\in (0,1]$, let $L_1^{\delta},\ldots, L_q^{\delta}$
be $C^1$ complex vector fields on $M$ such that $\forall \zeta\in M$, $\Span\{L_1^{\delta}(\zeta),\ldots, L_q^{\delta}(\zeta)\} = T^{0,1}_{\zeta}M$.
Let $X_1^{\delta},\ldots, X_{2q}^{\delta}$ denote the list of real vector fields $\Real(L_1^{\delta}),\ldots, \Real(L_q^{\delta}),\Imag(L_1^{\delta}),\ldots, \Imag(L_q^{\delta})$.
We assume that the list $X_1^{\delta},\ldots, X_{2q}^{\delta}$ locally satisfies the hypotheses of \cref{Section::Scale::GenSubR}.
Then, \cref{Thm::GenSubr::MainThm} applies to show that the balls $B(x,\delta)$ defined in that section locally give $M$ the structure of a space of
homogeneous type\footnote{Since $M$ is an abstract manifold, we do not have a natural choice of density $\LebDensity$ on $M$.  However,
one may instead use any strictly positive $C^1$ density on $M$ and obtain the same results.  All such choices of density are equivalent for our purposes.}, and we obtain scaling maps $\Phi_{x,\delta}:B^{2n}(\eta_1)\rightarrow B(x,\delta)$ as in that theorem.
In particular, by \cref{Thm::GenSubr::MainThm} \cref{Item::GenSubr::Pullback}, the maps $\Phi_{x,\delta}$ ``rescale''
the vector fields $X_1^{\delta},\ldots, X_{2q}^{\delta}$ so that they are smooth and span the tangent space, uniformly for $x$ in compact sets and $\delta\in (0,1]$.

In other words, $\Phi_{x,\delta}^{*}L_1^{\delta},\ldots, \Phi_{x,\delta}^{*}L_q^{\delta},\Phi_{x,\delta}^{*}\overline{L_1^{\delta}},\ldots, \Phi_{x,\delta}^{*}\overline{L_q^{\delta}}$
are smooth and span the complexified tangent space, uniformly for $x$ in compact sets and $\delta\in (0,1]$.
The hope is to apply techniques from several complex variables at the unit scale to these rescaled vector fields, to be able to conclude results at every
scale $\delta\in (0,1]$.
However, there is one key component that is missing in the complex setting.  We identify $\R^{2n}$ with $\C^n$ via the map
$(x_1,\ldots, x_{2n})\mapsto (x_1+ix_{n+1}, \ldots, x_n+ix_{2n})$.  To be able to apply 
results from complex analysis, we would need that
$\Phi_{x,\delta}^{*}L_1^{\delta},\ldots, \Phi_{x,\delta}^{*}L_q^{\delta}$ (thought of as vector fields on the ball of radius $\eta_1$ in $\C^n$) are still $T^{0,1}$ vector fields.
It is easy to see that this is equivalent to the map $\Phi_{x,\delta}$ being holomorphic.
However, the best one can say about the maps constructed in this series is that they are $C^2$.

One therefore wishes to obtain the same results as this paper, but with a different map $\Phi$, where we can also conclude that $\Phi$ is holomorphic.
In the past, this has been achieved in special cases by
using ad hoc methods for the particular problem at hand (e.g., by using non-isotropic dilations determined by the Taylor series of some ingredients in the problem)--see, for example,
\cite[Section 3]{NagelRosaySteinWaingerEstimatesForTheBergmanAndSzegoKernelsInCt}, \cite[Section 3.3.2]{CharpentierDupainExtremalBases}, and 
\cite[Section 2.1]{CharpentierDupainEstimatesForBergmanAndSezgoLocallyDiag}.
However, using such ad hoc methods does not allow one to proceed in the generality of this paper, and can obfuscate the underlying mechanism of the problem.

In a forthcoming paper, the second author will address this issue, and obtain appropriate analogs of results in this series in the complex setting; which
can be seen as a quantitatively diffeomorphic invariant version of the classical Newlander-Nirenberg theorem \cite{NewlanderNirenbergComplexAnalyticCoordiantesInAlmostComplexManifolds}.
The results and methods of this series are the first step in addressing this complex setting.

When we move to the complex setting (and more general settings which will be discussed in a future paper), the ODE methods of this paper are no longer sufficient to obtain even non-sharp results, and one
must move to PDE methods.  In particular, Zygmund spaces are the right scale of spaces to discuss any of the results in the complex setting.

%% file: funcspacesrev.tex

In this section, we state and prove the basic results we need concerning the function spaces
introduced in \cref{Section::FuncSpace}.  We begin with several straightforward inclusions
of these spaces, which we state in the next lemma.  For the rest of this section,
we take the setting of \cref{Section::FuncSpace::Manif}.

\begin{lemma}\label{Lemma::FuncSpaces::Inclu}
\begin{enumerate}[(i)]
\item\label{Item::FuncSpace::IncludHold} For $0\leq s_1\leq s_2\leq 1$, $m\in \N$, $\HXNorm{f}{X}{m}{s_1}[M]\leq 3\HXNorm{f}{X}{m}{s_2}[M]$.
\item\label{Item::FuncSpace::IncludLip} $\HXNorm{f}{X}{m}{1}[M]\leq \CXjNorm{f}{X}{m+1}[M]$.
\item\label{Item::FuncSpace::IncludZygInHold} For $s\in (0,1]$, $m\in \N$, $\ZygXNorm{f}{X}{s+m}[M]\leq 5 \HXNorm{f}{X}{m}{s}[M]$.
\item\label{Item::FuncSpace::IncludZyg} For $0<s_1\leq s_2<\infty$, $\ZygXNorm{f}{X}{s_1}[M]\leq 15 \ZygXNorm{f}{X}{s_2}[M]$.
\item\label{Item::FuncSpace::IncludInclusion} If $U\subseteq M$ is an open set, then $\HXNorm{f}{X}{m}{s}[U]\leq \HXNorm{f}{X}{m}{s}[M]$
and $\ZygXNorm{f}{X}{s}[U]\leq \ZygXNorm{f}{X}{s}[M]$.
\end{enumerate}
\end{lemma}
\begin{proof}
For \cref{Item::FuncSpace::IncludHold}, it suffices to prove the case $m=0$.  We have,
\begin{equation*}
\begin{split}
    &\HXNorm{f}{X}{0}{s_1} = \CNorm{f}{M} + \sup_{x\ne y} \rho(x,y)^{-s_1} |f(x)-f(y)|
    \leq \CNorm{f}{M} + \sup_{x\ne y} \min\{\rho(x,y),1\}^{-s_1} |f(x)-f(y)|
    \\&\leq \CNorm{f}{M} + \sup_{x\ne y} \min\{\rho(x,y),1\}^{-s_2} |f(x)-f(y)|
    \leq 3\CNorm{f}{M} + \sup_{x\ne y} \rho(x,y)^{-s_2} |f(x)-f(y)|
    \leq 3\HXNorm{f}{X}{0}{s_2},
\end{split}
\end{equation*}
proving \cref{Item::FuncSpace::IncludHold}.

For \cref{Item::FuncSpace::IncludLip}, it suffices to prove the case $m=0$.
Let $x\ne y\in M$ with $\rho(x,y)<\infty$, fix $\epsilon>0$,
and let $\delta=\rho(x,y)+\epsilon$.  Pick
$\gamma:[0,1]\rightarrow M$ with $\gamma(0)=x$, $\gamma(1)=y$, $\gamma'(t)=\sum_{j=1}^q a_j(t) \delta X_j(\gamma(t))$, $\Norm{\sum|a_j|^2}[L^\infty([0,1])]<1$.  Then we have,
\begin{equation*}
\begin{split}
    &\rho(x,y)^{-1}|f(x)-f(y)|=\rho(x,y)^{-1} \left|\int_0^1 \sum a_j(t) \delta (X_j f)(\gamma(t))\: dt\right|
    \leq \frac{\delta}{\rho(x,y)} \BNorm{\max_{1\leq j\leq q} |a_j(t)|}[L^\infty([0,1])] \sum_{j=1}^q \CNorm{X_j f}{M}
    \\&\leq \frac{\delta}{\rho(x,y)} \sum_{j=1}^q \CNorm{X_j f}{M} = \frac{\rho(x,y)+\epsilon}{\rho(x,y)} \sum_{j=1}^q \CNorm{X_j f}{M} \xrightarrow{\epsilon\rightarrow 0} \sum_{j=1}^q \CNorm{X_j f}{M}.
\end{split}
\end{equation*}
If $\rho(x,y)=\infty$, then $\rho(x,y)^{-1}|f(x)-f(y)|=0\leq \sum_{j} \CNorm{X_j f}{M}$.
It follows that $\HXNorm{f}{X}{0}{1}[M]\leq \CXjNorm{f}{X}{1}[M]$, completing the proof of
\cref{Item::FuncSpace::IncludLip}.

For \cref{Item::FuncSpace::IncludZygInHold}, it suffices to prove the case $m=0$.
Let $\gamma\in \sP^M_{X,s/2}(h)$.  Then $\rho(\gamma(2h),\gamma(h)),\rho(\gamma(h),\gamma(0))<h$,
and so
\begin{equation*}
    h^{-s}|f(\gamma(2h))-2f(\gamma(h))+f(\gamma(0))|\leq 2\sup_{\rho(x,y)<h} h^{-s} |f(x)-f(y)|
    \leq 2\sup_{x,y\in M, x\ne y} \rho(x,y)^{-s} |f(x)-f(y)|.
\end{equation*}
Combining this with $\HXNorm{f}{X}{0}{s/2}[M]\leq 3 \HXNorm{f}{X}{0}{s}[M]$ (by \cref{Item::FuncSpace::IncludHold}), \cref{Item::FuncSpace::IncludZygInHold} follows.

For \cref{Item::FuncSpace::IncludZyg} it suffices to prove the case when $s_1\in (0,1]$.
When $s_2\in (0,1]$, as well, then it follows easily from the definitions
that $\ZygXNorm{f}{X}{s_1}[M]\leq 5 \ZygXNorm{f}{X}{s_2}[M]$.
When $s_2>1$, we use \cref{Item::FuncSpace::IncludZygInHold}, \cref{Item::FuncSpace::IncludHold}, and \cref{Item::FuncSpace::IncludLip} to see
\begin{equation*}
    \ZygXNorm{f}{X}{s_1}[M]\leq 5 \HXNorm{f}{X}{0}{s_1}[M]\leq 15 \HXNorm{f}{X}{0}{1}[M]\leq 15\CXjNorm{f}{X}{1}[M]\leq 15 \ZygXNorm{f}{X}{s_2}[M],
\end{equation*}
completing the proof of \cref{Item::FuncSpace::IncludZyg}.  \Cref{Item::FuncSpace::IncludInclusion} follows
easily from the definitions.
\end{proof}

\begin{rmk}
Given the analogy with Euclidean spaces, one expects the reverse inequality to
\cref{Lemma::FuncSpaces::Inclu} \cref{Item::FuncSpace::IncludZygInHold}, when
$s\in (0,1)$; namely $\HXNorm{f}{X}{m}{s}[M]\lesssim \ZygXNorm{f}{X}{s+m}[M]$.  Under additional hypotheses, this is true locally.
See the second paper in this series  for details.
\end{rmk}

\begin{prop}\label{Prop::FuncSpaceRev::Algebra}
The spaces $\HXSpace{X}{m}{s}[M]$, $\ZygXSpace{X}{s}[M]$, $\HSpace{m}{s}[\Omega]$, and $\ZygSpace{s}[\Omega]$
are algebras.  In fact, we have for $m\in \N$, $s\in [0,1]$,
\begin{equation*}
    \HXNorm{fg}{X}{m}{s}[M]\leq C_{m,q} \HXNorm{f}{X}{m}{s}[M]\HXNorm{g}{X}{m}{s}[M],
\end{equation*}
where $C_{m,q}$ is a constant depending only on $m$ and $q$.  And for $m\in \N$, $s\in (m,m+1]$,
\begin{equation}\label{Eqn::FuncSpacesRev::ZygAlg}
    \ZygXNorm{fg}{X}{s}[M]\leq C_{m,q} \ZygXNorm{f}{X}{s}[M]\ZygXNorm{g}{X}{s}[M].
\end{equation}
Moreover, these algebras have multiplicative inverses for functions which are bounded away from zero.  If $f\in \HXSpace{X}{m}{s}[M]$ with $\inf_{x\in M} |f(x)|\geq c_0>0$
then $f(x)^{-1}=\frac{1}{f(x)}\in \HXSpace{X}{m}{s}[M]$ with
\begin{equation*}
	\HXNorm{f(x)^{-1}}{X}{m}{s}[M]\leq C,
\end{equation*}
where $C$ can be chosen to depend only on $m$, $q$, $c_0$, and an upper bound for $\HXNorm{f}{X}{m}{s}[M]$.
And for $m\in \N$, $s\in (m,m+1]$ if $f\in \ZygXSpace{X}{s}[M]$ with $\inf_{x\in M}|f(x)|\geq c_0>0$ then
$f(x)^{-1}\in \ZygXSpace{X}{s}[M]$ with
\begin{equation}\label{Eqn::FuncSpacesRev::ZygInv}
	\ZygXNorm{f(x)^{-1}}{X}{s}[M]\leq C,
\end{equation}
where $C$ can be chosen to depend only on $m$, $q$, $c_0$, and an upper bound for $\ZygXNorm{f}{X}{s}[M]$.
The same results hold with $\HXSpace{X}{m}{s}[M]$ replaced by $\HSpace{m}{s}[\Omega]$ and $\ZygXSpace{X}{s}[M]$ replaced by $\ZygSpace{s}[\Omega]$ (with $n$ playing
the role of $q$).
\end{prop}
\begin{proof}
The proofs for $\HXSpace{X}{m}{s}[M]$ and $\HSpace{m}{s}[\Omega]$ are straightforward and standard,
so we focus on the Zygmund spaces.  We prove
\cref{Eqn::FuncSpacesRev::ZygAlg}
by induction on $m$, where $s\in (m,m+1]$.
We begin with the base case $s\in (0,1]$.
Since we already know $\HXNorm{fg}{X}{0}{s/2}[M]\lesssim \HXNorm{f}{X}{0}{s/2}[M]\HXNorm{g}{X}{0}{s/2}[M]$,
it suffices to show for $\gamma\in \sP^M_{X,s/2}(h)$,
\begin{equation*}
    h^{-s}|f(\gamma(2h))g(\gamma(2h))-2f(\gamma(h))g(\gamma(h))+f(\gamma(0))g(\gamma(0))|\leq 6 \ZygXNorm{f}{X}{s}[M]\ZygXNorm{g}{X}{s}[M].
\end{equation*}
Notice that $\rho(\gamma(h),\gamma(0))\leq h$, and therefore
$|f(\gamma(h))-f(\gamma(0))|\leq h^{s/2} \HXNorm{f}{X}{0}{s/2}[M]$.
Thus, we have
\begin{equation*}
    \begin{split}
        &h^{-s} |f(\gamma(2h))g(\gamma(2h))-2f(\gamma(h))g(\gamma(h))+f(\gamma(0))g(\gamma(0))|
        \\&\leq h^{-s}|f(\gamma(2h))-2f(\gamma(h))+f(\gamma(0))||g(\gamma(2h))|
        \\&\quad +h^{-s}|2f(\gamma(h))-f(\gamma(0))| |g(\gamma(2h))-2g(\gamma(h))+g(\gamma(0))|
        \\&\quad +h^{-s} 2|f(\gamma(h))-f(\gamma(0))||g(\gamma(h))-g(\gamma(0))|
        \\&\leq \ZygXNorm{f}{X}{s}[M] \CNorm{g}{M} + 3 \CNorm{f}{M}\ZygXNorm{g}{X}{s}[M]+2\HXNorm{f}{X}{0}{s/2}[M]\HXNorm{g}{X}{0}{s/2}[M]
        \\&\leq 6 \ZygXNorm{f}{X}{s}[M]\ZygXNorm{g}{X}{s}[M].
    \end{split}
\end{equation*}
Having proved the base case, \cref{Eqn::FuncSpacesRev::ZygAlg} follows by a straightforward induction,
which we leave to the reader.

We now turn to inverses.  We prove \cref{Eqn::FuncSpacesRev::ZygInv} by induction on $m$, where $s\in (m,m+1]$.  We begin with the base case $s\in (0,1]$.
Let $f\in \ZygXSpace{X}{s}[M]$ with $\inf_{x\in M} |f(x)|\geq c_0>0$.  We write $A\lesssim B$ for $A\leq C B$ where $C$ is as in \cref{Eqn::FuncSpacesRev::ZygInv}.
Since we already know the results for H\"older spaces, we have $\HXNorm{f(x)^{-1}}{X}{0}{s/2}[M]\lesssim 1$.
Thus, it suffices to show for $\gamma\in \sP^M_{X,s/2}(h)$,
\begin{equation*}
\left|\frac{1}{f(\gamma(2h))} -\frac{2}{f(\gamma(h))}  + \frac{1}{f(\gamma(0))} \right| = \left| \frac{f(\gamma(h))f(\gamma(0)) - 2f(\gamma(2h))f(\gamma(0)) + f(\gamma(2h))f(\gamma(h))}{f(\gamma(2h))f(\gamma(h))f(\gamma(0))}\right|\lesssim h^s.
\end{equation*}
Since we have $|f(\gamma(2h))f(\gamma(h))f(\gamma(0))|\geq c_0^3\gtrsim 1$, it suffices to show
\begin{equation*}
\left|f(\gamma(h))f(\gamma(0)) - 2f(\gamma(2h))f(\gamma(0)) + f(\gamma(2h))f(\gamma(h))\right|\lesssim h^s.
\end{equation*}
But we have
\begin{equation*}
\begin{split}
&\left|f(\gamma(h))f(\gamma(0)) - 2f(\gamma(2h))f(\gamma(0)) + f(\gamma(2h))f(\gamma(h))\right|
\\&\leq \left|\left(f(\gamma(2h)) - 2f(\gamma(h))+f(\gamma(0))\right)f(\gamma(h))\right| +2 \left|f(\gamma(h))^2 -f(\gamma(2h))f(\gamma(0))  \right|
\\&\leq h^s \ZygXNorm{f}{X}{s}[M]\CNorm{f}{M} + 2 \left|f(\gamma(h))^2 -f(\gamma(2h))f(\gamma(0))  \right| \lesssim h^s + 2 \left|f(\gamma(h))^2 -f(\gamma(2h))f(\gamma(0))  \right|.
\end{split}
\end{equation*}
Thus, it suffices to show
\begin{equation*}
\left|f(\gamma(h))^2 -f(\gamma(2h))f(\gamma(0))  \right|\lesssim h^s.
\end{equation*}
But, using that $\rho(\gamma(h),\gamma(0))\leq h$, and therefore
$|f(\gamma(h))-f(\gamma(0))|\leq h^{s/2} \HXNorm{f}{X}{0}{s/2}[M]\lesssim h^{s/2}$, we have
\begin{equation*}
\begin{split}
&\left|f(\gamma(h))^2 -f(\gamma(2h))f(\gamma(0))  \right|\leq
\left|\left(f(\gamma(2h))-2f(\gamma(h))+f(\gamma(0))\right)f(\gamma(0))\right| + \left| (f(\gamma(h))-f(\gamma(0))^2 \right|
\\& \lesssim h^s+h^s \lesssim h^s,
\end{split}
\end{equation*}
completing the proof of the base case.  Having proved the base case, the inductive step is straightforward, and we leave it to the reader.

The proofs for $\ZygSpace{s}[\Omega]$ are similar, and we leave them to
the reader.
\end{proof}

\begin{rmk}\label{Rmk::FuncSpaceRev::NonstandardNorm}
In the proof of \cref{Prop::FuncSpaceRev::Algebra}, it is used that $\HNorm{f}{0}{s/2}[\Omega]\leq \ZygNorm{f}{s}[\Omega]$, $s\in (0,1]$, which is clearly true because of our nonstandard definition of $\ZygNorm{f}{s}[\Omega]$ (see \cref{Rmk::FuncSpace::NonstandardNorm}).
Even with the more standard definition, for a bounded Lipschitz domain $\Omega$, one has
$\HNorm{f}{0}{s/2}[\Omega]\leq C\ZygNorm{f}{s}[\Omega]$, however $C$ depends on $\Omega$.
Thus, if one takes the more standard definition, the conclusions of \cref{Prop::FuncSpaceRev::Algebra}
take a more complicated form.
\end{rmk}

\begin{rmk}
\Cref{Lemma::FuncSpaces::Inclu,Prop::FuncSpaceRev::Algebra} hold (with exactly the same proofs)
if $M$ is repalced by $B_X(x_0,\xi)$, whether or not $B_X(x_0,\xi)$ is a manifold--see
\cref{Section::FuncSpace::BeyondManif}.
\end{rmk}

\begin{prop}\label{Prop::FuncSpaceRev::PushForwardNorm}
Let $N$ be another $C^2$ manifold, $Y_1,\ldots, Y_q$ be $C^1$ vector fields on $N$,
and $\Phi:N\rightarrow M$ be a $C^1$ map such that
$d\Phi(u) Y_j(u) = X_j(\Phi(u))$, $\forall u\in N$.  Then,
\begin{equation}\label{Eqn::FuncSpaceRev::PushForwardHolder}
    \HXNorm{f\circ \Phi}{Y}{m}{s}[N]\leq \HXNorm{f}{X}{m}{s}[M],\quad m\in \N, s\in [0,1],
\end{equation}
\begin{equation}\label{Eqn::FuncSpaceRev::PushForwardZyg}
    \ZygXNorm{f\circ\Phi}{Y}{s}[N]\leq \ZygXNorm{f}{X}{s}[M], \quad s>0.
\end{equation}
\end{prop}
\begin{proof}
We begin with \cref{Eqn::FuncSpaceRev::PushForwardHolder}.
Since $Y^{\alpha} (f\circ \Phi)=(X^{\alpha} f)\circ \Phi$, it suffices to prove the case
$m=0$.  We have a sub-Riemannian metric $\rho_Y$ on $N$ and another sub-Riemannian metric
$\rho_X$ on $M$, defined by \cref{Eqn::FucsSpaceM::rho}.  We claim
\begin{equation}\label{Eqn::FuncSpaceRev::Comparerho}
\rho_X(\Phi(u_1),\Phi(u_2))\leq \rho_Y(u_1,u_2).
\end{equation}
This is clear if $\rho_Y(u_1,u_2)=\infty$.  If $\rho_Y(u_1,u_2)<\infty$,
let $\delta>\rho_Y(u_1,u_2)$.  Then, there exists $\gamma:[0,1]\rightarrow N$,
$\gamma(0)=u_1$, $\gamma(1)=u_2$, $\gamma'(t)=\sum a_j(t) \delta Y_j(\gamma(t))$,
$\Norm{\sum |a_j|^2}[L^\infty([0,1])]<1$.  Set $\gammat=\Phi\circ \gamma$.  Then,
$\gammat(0)=\Phi(u_1)$, $\gammat(1)=\Phi(u_2)$, and
$\gammat'(t)=\sum a_j(t) \delta X_j(\gammat(t))$.  This proves
$\rho_X(\Phi(u_1),\Phi(u_2))<\delta$.
Taking $\delta\rightarrow \rho_Y(u_1, u_2)$ proves \cref{Eqn::FuncSpaceRev::Comparerho}.
We conclude, for $s\in [0,1]$,
\begin{equation*}
    \rho_Y(u_1,u_2)^{-s}|f\circ \Phi(u_1)-f\circ\Phi(u_2)| \leq \rho_X(\Phi(u_1),\Phi(u_2))^{-s}|f(\Phi(u_1))-f(\Phi(u_2))|.
\end{equation*}
\Cref{Eqn::FuncSpaceRev::PushForwardHolder} follows.

We turn to \cref{Eqn::FuncSpaceRev::PushForwardZyg}.  Again, since $Y^{\alpha} (f\circ \Phi)=(X^{\alpha} f)\circ \Phi$, it suffices to prove \cref{Eqn::FuncSpaceRev::PushForwardZyg} for $s\in (0,1]$.
That $\HXNorm{f\circ \Phi}{Y}{0}{s/2}[N]\leq \HXNorm{f}{X}{0}{s/2}[M]$ follows from
\cref{Eqn::FuncSpaceRev::PushForwardHolder}.  Furthermore, it follows easily from the definitions
that for $\gamma\in \sP^{N}_{Y,s/2}(h)$, we have $\Phi\circ \gamma\in \sP^M_{X,s/2}(h)$.
Using this, \cref{Eqn::FuncSpaceRev::PushForwardZyg} for $s\in (0,1]$ follows immediately.
\end{proof}

%% file: funcspacesrevcomp.tex

Fix $\eta\in (0,1]$ and let $Y_1,\ldots, Y_q$ be vector fields on $B^n(\eta)$.
When $Y_1,\ldots, Y_q$ span the tangent space at every point of $B^n(\eta)$
and are sufficiently smooth, we have $\HXSpace{Y}{m}{s}[B^n(\eta)]=\HSpace{m}{s}[B^n(\eta)]$
and $\ZygXSpace{Y}{s}[B^n(\eta)]=\ZygSpace{s}[B^n(\eta)]$.  In what follows, we state and
prove quantitative versions of these equalities.

We write $Y_j =\sum_{k=1}^n a_j^k \diff{t_k}$ and assume $\diff{t_k}=\sum_{j=1}^q b_k^j Y_j$,
where $a_j^k\in \CjSpace{1}[B^n(\eta)]$, $b_k^j\in \CSpace{B^n(\eta)}$.

\begin{defn}
In analogy with \cref{Defn::FuncSpacesM::NegativeSpaces}, for $m<0$ we define
$\HSpace{m}{s}[B^n(\eta)]:=\CSpace{B^n(\eta)}$, with equality of norms.
For $s\in (-1,0]$ we define $\ZygSpace{s}[B^n(\eta)]:=\HSpace{0}{(s+1)/2}[B^n(\eta)]$,
with equality of norms.
\end{defn}


\begin{defn}
We say $C$ is a $\ZeroE$-admissible constant\footnote{Here we are using the $\mathrm{E}$ to stand for Euclidean, and to help differentiate these admissible constants
 from the other admissible constants in this paper.} if $C$ can be chosen to depend only on
upper bounds for $q$ and $\CNorm{a_j^k}{B^n(\eta)}$, $\CNorm{b_k^j}{B^n(\eta)}$, $\forall j,k$.
\end{defn}

\begin{defn}\label{Defn::FuncSpaeRevComp::HEAdmiss}
For $m\in \Z$, $s\in [0,1]$, if we say $C$ is an $\HEad{m,s}$-admissible constant
if $a_j^k,b_k^j\in \HSpace{m}{s}[B^n(\eta)]$, $\forall j,k$, and $C$ can be chosen
to depend only on upper bounds for $q$, $m$, and $\HNorm{a_j^k}{m}{s}[B^n(\eta)]$, $\HNorm{b_k^j}{m}{s}[B^n(\eta)]$, $\forall j,k$.
\end{defn}

\begin{defn}
For $s>-1$ we say $C$ is an $\ZygEad{s}$-admissible constant if $a_j^k,b_k^j\in \ZygSpace{s}[B^n(\eta)]$,
$\forall j,k$ and $C$ can be chosen to depend only on $s$ and upper bounds for $q$, $\eta^{-1}$,
and $\ZygNorm{a_j^k}{s}[B^n(\eta)]$, $\ZygNorm{b_k^j}{s}[B^n(\eta)]$, $\forall j,k$.
\end{defn}

As before, we define $A\lesssim_{\HEad{m,s}} B$ to be $A\leq CB$ where $C$ is an $\HEad{m,s}$-admissible
constant.  We similarly define $\approx_{\HEad{m,s}}$, $\lesssim_{\ZygEad{s}}$, and $\approx_{\ZygEad{s}}$.
Recall, the vector fields $Y_1,\ldots, Y_q$ induce a metric $\rho$ on $B^n(\eta)$
via \cref{Eqn::FucsSpaceM::rho}.

\begin{lemma}\label{Lemma:FuncSpaceRev::Comp::rhoequiv}
$\rho(x,y)\approx_{\ZeroE} |x-y|$.
\end{lemma}
\begin{proof}
This follows immediately from the assumptions.
\end{proof}

\begin{prop}\label{Prop::FuncSpaceRev::CompNorms}
For $m\in \N$, $s\in [0,1]$,
\begin{equation}\label{Eqn::FuncSpaceRev::Comp::Holder}
\HNorm{f}{m}{s}[B^n(\eta)]\approx_{\HEad{m-1,s}} \HXNorm{f}{Y}{m}{s}[B^n(\eta)],
\end{equation}
and for $s>0$,
\begin{equation}\label{Eqn::FuncSpaceRev::Comp::Zyg}
\ZygNorm{f}{s}[B^n(\eta)]\approx_{\ZygEad{s-1}} \ZygXNorm{f}{Y}{s}[B^n(\eta)].
\end{equation}
\end{prop}
\begin{proof}
We use \cref{Prop::FuncSpaceRev::Algebra} freely in this proof.
In this proof, the norms $\HXNorm{f}{Y}{m}{s}$, $\HNorm{f}{m}{s}$, $\ZygXNorm{f}{Y}{s}$, and $\ZygNorm{f}{s}$
are always taken to be over the domain $B^n(\eta)$ unless otherwise mentioned.
We prove \cref{Eqn::FuncSpaceRev::Comp::Holder} by induction on $m$.
The base case, $\HNorm{f}{0}{s}[B^n(\eta)]\approx_{\ZeroE} \HXNorm{f}{Y}{0}{s}[B^n(\eta)]$,
follows immediately from \cref{Lemma:FuncSpaceRev::Comp::rhoequiv}.
We assume \cref{Eqn::FuncSpaceRev::Comp::Holder} for $m-1$ and prove it for $m$.
We have
\begin{equation*}
    \begin{split}
        &\HXNorm{f}{Y}{m}{s} = \HXNorm{f}{Y}{m-1}{s} + \sum_{j=1}^q \HXNorm{Y_j f}{Y}{m-1}{s}
        \approx_{\HEad{m-2,s}} \HNorm{f}{m-1}{s}+\sum_{j=1}^q \HNorm{Y_j f}{m-1}{s}
        \\&\leq \HNorm{f}{m-1}{s}+\sum_{j=1}^q\sum_{k=1}^n \HNorm{a_j^k \partial_{x_k} f}{m-1}{s}
        \lesssim_{\HEad{m-1,s}} \HNorm{f}{m}{s}.
    \end{split}
\end{equation*}
For the reverse inequality,
\begin{equation*}
    \begin{split}
        &\HNorm{f}{m}{s}\leq \HNorm{f}{m-1}{s} + \sum_{k=1}^n \HNorm{\partial_{x_k} f}{m-1}{s}
        \leq \HNorm{f}{m-1}{s}+\sum_{k=1}^n\sum_{j=1}^q \HNorm{b_k^j Y_j f}{m-1}{s}
        \\&\lesssim_{\HEad{m-1,s}} \HNorm{f}{m-1}{s} + \sum_{j=1}^q \HNorm{Y_j f}{m-1}{s}
        \lesssim_{\HEad{m-2,s}} \HXNorm{f}{Y}{m-1}{s}+\sum_{j=1}^q \HXNorm{Y_j f}{Y}{m-1}{s}
        =\HXNorm{f}{Y}{m}{s}.
    \end{split}
\end{equation*}
This completes the proof of \cref{Eqn::FuncSpaceRev::Comp::Holder}.

We prove \cref{Eqn::FuncSpaceRev::Comp::Zyg} by induction on $m$, where $s\in (m,m+1]$.
We begin with the base case, $m=0$, and thus $s\in (0,1]$.
First we show $\lesssim_{\ZygEad{s-1}}$.  Take $0\ne h\in \R^n$, and $x\in \Omega_h$ (where $\Omega=B^n(\eta)$).  Set $\gamma(t)=x+t\theta$, where $\theta=h/|h|$.
Note $\gamma'(t)=\sum_{k=1}^n \theta_k \diff{x_k}= \sum_{j=1}^q\sum_{k=1}^n \theta_k b_k^j(\gamma(t)) Y_j (\gamma(t))$.  Since $\HNorm{b_k^j}{0}{s/2}\lesssim_{\ZygEad{s-1}} 1$,
we have $\HNorm{b_j^k\circ \gamma}{0}{s/2}\lesssim_{\ZygEad{s-1}} 1$, and therefore
$\gamma\in \sP_{Y,s/2}^{B^n(\eta)}(C|h|)$, where $C\lesssim_{\ZygEad{s-1}} 1$.  Hence,
\begin{equation*}
    |h|^{-s} |f(x+2h)-2f(x+h)+f(x)|\lesssim_{\ZygEad{s-1}} (C|h|)^{-s} |f(\gamma(2|h|))-2f(\gamma(|h|))+f(\gamma(0))| \leq \ZygXNorm{f}{Y}{s}[B^n(\eta)].
\end{equation*}
Since we have already shown $\HNorm{f}{0}{s/2}\approx_{\ZygEad{s-1}} \HXNorm{f}{Y}{0}{s/2}$ (by \cref{Eqn::FuncSpaceRev::Comp::Holder}), the $\lesssim_{\ZygEad{s-1}}$ direction of
\cref{Eqn::FuncSpaceRev::Comp::Zyg} follows.

We turn to $\gtrsim_{\ZygEad{s-1}}$.  We already have $\HXNorm{f}{Y}{0}{s/2}\approx_{\ZygEad{s-1}} \HNorm{f}{0}{s/2}\leq \ZygNorm{f}{s}$ (by \cref{Eqn::FuncSpaceRev::Comp::Holder}).
Fix $h>0$ and $\gamma\in \sP_{Y,s/2}^{B^n(\eta)}(h)$.
Note $\gamma'(t)=\sum_{j=1}^q d_j(t) Y_j(\gamma(t))=\sum_{j=1}^q \sum_{k=1}^n d_j(t) a_j^k(\gamma(t))\diff{x_k}$, with $\sum \HNorm{d_j}{0}{s/2}[[0,2h]]^2<1$.  Since we also have  $\HNorm{a_j^k}{0}{s/2}\lesssim_{\ZygEad{s-1}} 1$,
it follows that $\HNorm{\gamma}{1}{s/2}[[0,2h]]\lesssim_{\ZygEad{s-1}} 1$.
Define $\gammat:[0,2h]\rightarrow B^n(\eta)$ by
$\gammat(t)=(t/2h)\gamma(2h)+(1-t/2h)\gamma(0)$.

We claim that
\begin{equation}\label{Eqn::FuncSpaceRevComp::GammaminusGammt}
|\gamma(t)-\gammat(t)|\lesssim_{\ZygEad{s-1}} h^{1+s/2}.
\end{equation}
Indeed,
\begin{equation*}
    |\gammat(t)-\gamma(t)|= t\left| \frac{\gamma(2h)-\gamma(0)}{2h} -\frac{\gamma(t)-\gamma(0)}{t}\right|
    =t|\gamma'(c_1)-\gamma'(c_2)|,
\end{equation*}
by the mean value theorem, where $c_1,c_2\in [0,2h]$.  Since $t\in [0,2h]$,
it follows that $|\gamma(t)-\gammat(t)|\lesssim_{\ZygEad{s-1}} h^{1+s/2}$,
by using the estimate $\HNorm{\gamma}{1}{s/2}[[0,2h]]\lesssim_{\ZygEad{s-1}} 1$.

Next we claim that
\begin{equation}\label{Eqn::FuncSpaeRevComp::ToShow::TrebBound}
\HNorm{f}{0}{s/(1+s/2)}[B^n(\eta)]\lesssim_{\ZygEad{s-1}} \ZygNorm{f}{s}[B^n(\eta)].
\end{equation}
To prove \cref{Eqn::FuncSpaeRevComp::ToShow::TrebBound} we use
\begin{equation}\label{Eqn::FundSpaceRevComp::ToShow::TrebEquiv}
\HNorm{f}{0}{s/(1+s/2)}[B^n(\eta)]\approx \ZygNorm{f}{s/(1+s/2)}[B^n(\eta)],
\end{equation}
where the implicit constants depend on $s$, $n$, and an upper bound for $\eta^{-1}$
(here we use $s/(1+s/2)\in (0,1)$; \cref{Eqn::FundSpaceRevComp::ToShow::TrebEquiv} does not hold when the exponent equals $1$).
Then, since $0<s/(1+s/2)<s\leq 1$, we have $\ZygNorm{f}{s/(1+s/2)}[B^n(\eta)]\leq 5 \ZygNorm{f}{s}[B^n(\eta)]$ (this follows immediately from the definitions) and
\cref{Eqn::FuncSpaeRevComp::ToShow::TrebBound} follows.
\Cref{Eqn::FundSpaceRevComp::ToShow::TrebEquiv} is classical;
indeed, we first consider the case when $\eta=1$.  The $\gtrsim$
part of \cref{Eqn::FundSpaceRevComp::ToShow::TrebEquiv} follows immediately
from the definitions.  For the $\lesssim$ part when $\eta=1$, see \cite[Theorem 1.118 (i)]{TriebelTheoryOfFunctionSpacesIII}--by choosing $M=1,2$ in that theorem,
the $\lesssim$ part of \cref{Eqn::FundSpaceRevComp::ToShow::TrebEquiv} follows, for $\eta=1$,
with implicit constant depending only on $s$ and $n$.
Finally, a simple scaling argument establishes \cref{Eqn::FundSpaceRevComp::ToShow::TrebEquiv} for general
$\eta>0$, which we leave to the reader.
 

Note that $\gammat(t)$ is a line with $|\gammat(2h)-\gammat(0)| \leq 2h \CjNorm{\gamma}{1}[[0,2h]]\lesssim_{\ZygEad{s-1}} h$;
and therefore $|f(\gammat(2h))-2f(\gammat(h))+f(\gammat(0))|\lesssim_{\ZygEad{s-1}} h^s \ZygNorm{f}{s}$.
We combine this with \cref{Eqn::FuncSpaeRevComp::ToShow::TrebBound,Eqn::FuncSpaceRevComp::GammaminusGammt}
 to see:
\begin{equation*}
    \begin{split}
        &|f(\gamma(2h))-2f(\gamma(h))+f(\gamma(0))|
        \\&\leq |f(\gammat(2h))-2f(\gammat(h))+f(\gammat(0))| + 2 |f(\gamma(h))-f(\gammat(h))| 
        \\&\lesssim_{\ZygEad{s-1}} h^s \ZygNorm{f}{s} + |\gamma(h)-\gammat(h)|^{s/(1+s/2)} \HNorm{f}{0}{s/(1+s/2)} 
        \\&\lesssim_{\ZygEad{s-1}} h^s \ZygNorm{f}{s}.
    \end{split}
\end{equation*}
This proves $\ZygXNorm{f}{Y}{s}\lesssim_{\ZygEad{s-1}} \ZygNorm{f}{s}$, and completes
the proof for the base case of \cref{Eqn::FuncSpaceRev::Comp::Zyg}.  From here
the inductive step follows just as in the inductive step for \cref{Eqn::FuncSpaceRev::Comp::Holder},
and we leave it to the reader.
\end{proof}

%% file: pfintro.tex

We turn to the proofs of the main results of this paper.  The heart of this paper
is the study of a certain ODE which arises in canonical coordinates; this is presented in
\cref{Section::Proofs::ODE}.  Then we present a quantitative version of a special case
of the Inverse Function Theorem in \cref{Section::Proofs::IFT}.  We then prove
the main result (\cref{Thm::Results::MainThm}) in \cref{Section::Proofs::Main}.
Next, we prove the results concerning densities from \cref{Section::Densities}
in \cref{Section::Proofs::Densities}.
Finally, we prove \cref{Prop::MoreAssumpt} in \cref{Section::Proofs::MoreAssump}.

%% file: pfode.tex

The quantitative study of canonical coordinates is closely tied to the study of the following ODE,
defined for an $n\times n$ matrix $A(u)$, depending on $u\in B^n(\eta)$ for some $\eta>0$.
Write $u=r\theta$, $r>0$, $\theta\in S^{n-1}$.  The ODE is:
\begin{equation}\label{Eqn::ODE::TheODE}
    \diff{r} rA(r\theta)= -A(r\theta)^2-C(r\theta)A(r\theta)-C(r\theta),
\end{equation}
where $C(u)\in C(B^n(\eta); \M^{n\times n})$ is a given function.
That this ODE arises in the study of cannonical coordinates
is classical (see, for example, \cite[p. 155]{ChevalleyTheoryOfLieGroups} for the derivation of a similar ODE);
however the detailed study of the ODE to prove regularity properties
in canonical coordinates was pioneered by Tao and Wright \cite{TaoWrightLpImproving}.

In \cref{Section::DerivODE} we show how this ODE arises in cannonical coordinates.
Because our vector fields $X_1,\ldots, X_q$ are merely assumed to be $C^1$, there are some slight
technicalities which we deal with in that section.
In \cref{Section::RegularODE} we prove the regularity properties of solutions to this ODE.

%% file: pfodederive.tex

Let $X_1,\ldots, X_n$ be $C^1$ vector fields on an
$n$-dimensional $C^2$ manifold $M$.
Fix $x\in M$ and $\epsilon>0$ and suppose:
\begin{itemize}
    \item $X_1,\ldots, X_n$ span the tangent space at every point of $M$.
    \item $\Phi(u):=e^{u_1 X_1+u_2X_2+\cdots +u_n X_n}x$ exists for $u\in B^n(\epsilon)$.
\end{itemize}
Write $[X_j,X_k]=\sum_{l=1}^n c_{j,k}^l X_l$.  Since $X_1,\ldots, X_n$ form a basis
for the tangent space of $M$ at every point, $c_{j,k}^l\in C(M)$ are uniquely defined.
Classical theorems show that $\Phi$ is $C^1$ (since $X_1,\ldots, X_n$ are).

Let $U\subseteq M$ and $V\subseteq B^n(\epsilon)$ be open sets such that
$\Phi|_V:V\rightarrow U$ is a $C^1$ diffeomorphism.
Let $Y_j=\Phi|_V^{*}X_j$ so that $Y_j$ is a $C^0$ vector field on $V$.
Write,
\begin{equation}\label{Eqn::ODE::Derive::Defnajk}
    Y_j=\diff{u_j}+\sum_{k=1}^n a_j^k(u)\diff{u_k},
\end{equation}
where $a_j^k\in C(V)$.  Let $A(u)$ denote the $n\times n$ matrix with $j,k$ component $a_j^k(u)$,
and let $C(u)$ denote the $n\times n$ matrix with $j,k$ component $\sum_l u_l c_{j,l}^k\circ \Phi(u)$.
We write $u$ in polar coordinates as $u=r\theta$, $r\geq 0$.

\begin{prop}\label{Prop::ODE::Derive::MainProp}
In the above setting, $A(u)$ satisfies the differential equation
\begin{equation}\label{Eqn::ODE::Derive::MainEqn}
    \diff{r}rA(r\theta)=-A(r\theta)^2-C(r\theta)A(r\theta)-C(r\theta).
\end{equation}
In particular, $\diff{r}rA(r\theta)$ exists in the classical sense.
\end{prop}

\begin{lemma}\label{Lemma::ODE::Deriv::SmoothMfld}
\Cref{Prop::ODE::Derive::MainProp} holds in the special case when
$M$ is a $C^\infty$ manifold and $X_1,\ldots, X_n$ are $C^\infty$ vector fields on $M$.
\end{lemma}
\begin{proof}
When $X_1,\ldots, X_n$ are $C^\infty$, then $\Phi$ is $C^\infty$ and
$\Phi|_V:V\rightarrow \Phi(V)$ is a $C^\infty$ diffeomorphism.
We conclude that $Y_1,\ldots, Y_n$ are $C^\infty$ vector fields.
Furthermore, $[Y_j,Y_k]=\sum_{l} \ct_{j,k}^l Y_l$.
where $\ct_{j,k}^l=c_{j,k}^l\circ \Phi$.

Note that $d\Phi(r\theta) r\diff{r} = r d\Phi(r\theta)\diff{r} = r\theta\cdot X(\Phi(r\theta))$,
since $\Phi(r\theta) = e^{r(\theta\cdot X)}x$,
and we are identifying $X$ with the vector of vector fields $(X_1,\ldots, X_n)$.
Writing this in Cartesian coordinates, we have
\begin{equation}\label{Eqn::ODE::Derive::Smooth::Basic}
    \sum_{j=1}^n u_j \diff{u_j} = \sum_{j=1}^n u_j Y_j(u).
\end{equation}
Taking the Lie bracket of \cref{Eqn::ODE::Derive::Smooth::Basic} with $Y_i$, we obtain
\begin{equation}\label{Eqn::ODE::Derive::Smooth::1}
    \begin{split}
        \sum_{j=1}^n \left( (Y_i u_j) \partial_{u_j} + u_j [Y_i, \partial_{u_j} ] \right)
        =\sum_{j=1}^n \left( (Y_i u_j) Y_j + u_j [Y_i,Y_j] \right)
        =\sum_{j=1}^n\left( (Y_i u_j)Y_j + u_j \sum_{l=1}^n \ct_{i,j}^l(u) Y_l \right).
    \end{split}
\end{equation}
We re-write \cref{Eqn::ODE::Derive::Smooth::1} as
\begin{equation}\label{Eqn::ODE::Derive::Smooth::2}
    \begin{split}
        \left(\sum_{j=1}^n u_j [\partial_{u_j}, Y_i-\partial_{u_i} ]\right) + Y_i-\partial_{u_i}
        =-\left( \sum_{j=1}^n \left((Y_i-\partial_{u_i})(u_j)\right) (Y_j-\partial_{u_j})  \right)
        -\sum_{j=1}^n \sum_{l=1}^n u_j \ct_{i,j}^l(u) Y_l.
    \end{split}
\end{equation}
Plugging \cref{Eqn::ODE::Derive::Defnajk} into \cref{Eqn::ODE::Derive::Smooth::2}, we have
\begin{equation*}
    \sum_{j=1}^n \sum_{k=1}^n u_j (\partial_{u_j} a_i^k) \partial_{u_k} + \sum_{k=1}^n a_i^k\partial_{u_k}
    =-\sum_{k=1}^n \sum_{j=1}^n a_i^j a_j^k \partial_{u_k} - \sum_{k=1}^n \sum_{j=1}^n u_j \ct_{i,j}^k \partial_{u_k} - \sum_{l=1}^n\sum_{k=1}^n \sum_{j=1}^n u_j \ct_{i,j}^l a_l^k \partial_{u_k}.
\end{equation*}
Taking the $\partial_{u_k}$ component of the above, and writing $1+\sum_{j=1}^n u_j \partial_{u_j}=\partial_r r$, we have
\begin{equation*}
    \partial_r r a_i^k = -\sum_{j=1}^n a_i^j a_j^k -\sum_{j=1}^n u_j \ct_{i,j}^k -\sum_{l=1}^n \left(\sum_{j=1}^n u_j \ct_{i,j}^l\right)a_l^k.
\end{equation*}
This is exactly \cref{Eqn::ODE::Derive::MainEqn} and completes the proof.
\end{proof}

\begin{proof}[Proof of \cref{Prop::ODE::Derive::MainProp}]
By a classical theorem of Whitney, there is a $C^\infty$ structure on $M$
compatible with its $C^2$ structure, so we may assume $M$ is a $C^\infty$ manifold.
  Pick\footnote{Recall, $\Vt\Subset V$ means that $\Vt$ is a relatively compact susbet of $V$.} 
  $\Vt\Subset V$ and $\Ut\Subset U$ open sets with $\Phi|_{\Vt}:\Vt\rightarrow \Ut$
a $C^1$ diffeomorphism.  Fix $u_0\in \Vt$.  We will prove the result with
$V$ replaced by $B^{n}(u_0,\delta_0)$ for some $\delta_0>0$, and
the result will follow as the conclusion is local.

Fix $\epsilon'\in (0,\epsilon)$ so large that $\Vt\subseteq B^{n}(\epsilon')$.
Let $X_j^{\sigma}$ be smooth vector fields on $M$ such that $X_j^{\sigma}\rightarrow X_j$ in $C^1$
as $\sigma\rightarrow 0$.  Define
\begin{equation*}
    \Phi_\sigma(u) = e^{u_1 X_1^{\sigma}+\cdots+ u_nX_n^{\sigma}}x.
\end{equation*}
Then, for $\sigma$ sufficiently small, $\Phi_\sigma(u)$ is defined for $u\in B^n(\epsilon')$,
and $X_1^{\sigma},\ldots, X_n^{\sigma}$ form a basis for the tangent space
at every point of a neighborhood of the closure of $\Phi_\sigma(B^n(\epsilon'))$.
Thus, we may write $[X_i^{\sigma},X_j^{\sigma}]=\sum_{k} c_{i,j}^{k,\sigma} X_k^{\sigma}$,
with $c_{i,j}^{k,\sigma}\rightarrow c_{i,j}^l$ in $C^0$ as $\sigma\rightarrow 0$.
Also, $\Phi_{\sigma}\rightarrow \Phi$ in $C^1(B^{n}(\epsilon'))$ as $\sigma\rightarrow 0$,
by standard theorems.

For $\sigma$ sufficiently small, $|\det d\Phi_\sigma (u_0)|\geq \frac{1}{2} |\det d\Phi(u_0)|>0$.
The Inverse Function Theorem shows that there is a $\delta_0>0$ (independent of $\sigma$)
so that for $\sigma$ small, $\Phi_{\sigma}|_{B^n(u_0,\delta_0)}$ is a diffeomorphism
onto its image.

Define $A_{\sigma}$ and $C_\sigma$ in the obvious way on $B^{n}(u_0,\delta_0)$,
by using the vector fields $X_1^{\sigma},\ldots, X_n^{\sigma}$.  We have that
$A_{\sigma}\rightarrow A$ and $C_\sigma\rightarrow C$ in $C^0(B^n(u_0,\delta_0))$.
Furthermore, by \cref{Lemma::ODE::Deriv::SmoothMfld},
$\partial_r rA_\sigma=-A_\sigma^2-C_\sigma A_\sigma-C_\sigma$.
Taking the limit as $\sigma\rightarrow 0$, we find that $\partial_r rA$ exists in the
classical sense and $\partial_r rA=-A^2-C A-C$,
completing the proof.
\end{proof}

For another proof of \cref{Prop::ODE::Derive::MainProp} in the special case where $\epsilon$ is assumed to be small, see \cite[Appendix A]{MontanariMorbidelliStepSInvolutiveFamiliesOfVectorFields}.

%% file: pfoderegularity.tex

In this section, we discuss the existence, uniqueness, and regularity of solutions to \cref{Eqn::ODE::TheODE}
satisfying $A(0)=0$.
Some of this was done in \cite{S}, however we provide a complete proof here.

To facilitate the proof, we introduce a family of function spaces on $B^n(\eta)$.
Throughout this section, for a matrix $A$, we write $|A|$ to denote the operator norm of $A$.

Fix $\eta>0$, we are interested in solutions $A(x)\in C(B^n(\eta);\M^{n\times n})$
to \cref{Eqn::ODE::TheODE} (in this section, we use the variable $x$ in place of $u$).
For $l\in \N$ set
\begin{equation*}
    \Omega_l:=\left\{(x,h)\in B^n(\eta) \times (\R^n\setminus\{0\}) : x+jh\in B^n(\eta), 0\leq j\leq l\right\}.
\end{equation*}
Note that $\Omega_0:=B^n(\eta)\times (\R^n\setminus\{0\})$.
For $h\in \R^n\setminus\{0\}$ set $\DiffOp{h}A(x)=A(x+h)-A(x)$ and $\DiffOp{h}^lA(x) = (\DiffOp{h})^lA(x)$.
Note that $\DiffOp{h}^lA(x)$ is defined precisely for $(x,h)\in\Omega_l$.
Without explicitly mentioning it, we will repeatedly use the fact that if
$(x,h)\in \Omega_l$ and $s\in (0,1]$, then $(sx,sh)\in \Omega_l$.

Let $\omega:(0,\infty)\rightarrow(0,\infty)$ be a non-decreasing function and
for $l,m\in \N$ set
\begin{equation*}
    \Norm{A}[C^{m,l,\omega}]:=\sum_{|\beta|\leq m} \sum_{j=0}^l \sup_{(x,h)\in \Omega_j} \omega(|h|)^{-j} \left|\DiffOp{h}^{j} \partial_x^\beta A(x)\right|, \quad C^{m,l,\omega}:= \left\{ A\in C^m(B^n(\eta);\M^{n\times n}) : \Norm{A}[C^{m,l,\omega}]<\infty\right\}.
\end{equation*}
Note that $C^{m,l,\omega}$ is a Banach space, and when $l=0$, $\omega$ does not play a role.

\begin{rmk}\label{Rmk::ODE::ImportantSpaces}
We are particularly interested in the following special cases
\begin{equation*}
    \CjSpace{m}[B^n(\eta)][\M^{n\times n}]= C^{m,0,\omega}, \quad
    \HSpace{m}{s}[B^n(\eta)][\M^{n\times n}]=C^{m,1,\omega_s}, \quad
    \ZygSpace{m+s}[B^n(\eta)][\M^{n\times n}]
    =C^{m,2,\omega_{s/2}},
\end{equation*}
with equality of norms, where $\omega_{s}(h)=h^{s}$.
\end{rmk}

\begin{prop}\label{Prop::ODE::ExistMainProp}
Let $C\in C(B^{n}(\eta); \M^{n\times n})$ be given with $C(0)=0$.
Suppose $|C(x)|\leq D|x|$, for $x\in B^n(\eta)$.  Then, if $\eta\leq (10 D)^{-1}$,
there exists a unique $A\in C^{0}(B^n(\eta); \M^{n\times n})$ with $A(0)=0$
satisfying \cref{Eqn::ODE::TheODE}.  This unique solution satisfies:
\begin{equation}\label{Eqn::ODE::ExistMainProp::SupBounds}
    |A(x)|\leq \frac{5}{8}D|x|\text{ and }|A(x)|\leq \frac{1}{16},\quad \forall x\in B^n(\eta).
\end{equation}
Furthermore, for this solution $A$,
\begin{equation*}
    C\in C^{m,l,\omega}\Rightarrow A\in C^{m,l,\omega},\quad \forall m,l,\omega,
\end{equation*}
and
\begin{equation*}
    \Norm{A}[C^{m,l,\omega}]\leq K_{n,m,l,\omega},
\end{equation*}
where $K_{n,m,l,\omega}$ can be chosen to depend only on $n$, $m$, $l$, and an upper bound
for $\Norm{C}[C^{m,l,\omega}]$.
\end{prop}

The rest of this section is devoted to the proof of \cref{Prop::ODE::ExistMainProp}.  We begin
with several lemmas.

\begin{lemma}\label{Lemma::ODE::BasicCmlProperties}
For $j\leq m$, $k\leq l$, $C^{m,l,\omega}\hookrightarrow C^{j,k,\omega}$ and
\begin{equation}\label{Eqn::ODE::InclusionIneq}
\Norm{A}[C^{j,k,\omega}]\leq \Norm{A}[C^{m,l,\omega}].
\end{equation}
If $A,B\in C^{m,l,\omega}$, then $AB\in C^{m,l,\omega}$ and
\begin{equation}\label{Eqn::ODE::AlgebraIneq}
    \Norm{AB}[C^{m,l,\omega}]\leq C_{m,l} \Norm{A}[C^{m,l,\omega}]\Norm{B}[C^{m,l,\omega}],
\end{equation}
where $C_{m,l}$ can be chosen to depend only on $m$ and $l$.
\end{lemma}
\begin{proof}
The inclusion and inequality \cref{Eqn::ODE::InclusionIneq} follow immediately from the definitions,
thus we prove only the algebra property and \cref{Eqn::ODE::AlgebraIneq}.

For $A,B\in C^{m,l,\omega}$ and $0\leq j\leq l, |\beta|\leq m$,
we have
    $\partial_x^{\beta} \DiffOp{h}^j(AB)$
is a constant coefficient linear combination of terms of the form
\begin{equation}\label{Eqn::ODE::ProductDefined}
    \tau_{k_1 h} \left(\DiffOp{h}^{j_1} \partial_x^{\beta_1} A \right) \tau_{k_2 h} \left( \DiffOp{h}^{j_2} \partial_x^{\beta_2} B \right),
\end{equation}
where $\tau_h A(x) = A(x+h)$, $j_1+j_2=j$, $0\leq k_1\leq j_2$, $0\leq k_2\leq j_1$, $\beta_1+\beta_2=\beta$.  Note that, since $0\leq k_1\leq j_2$, $0\leq k_2\leq j_1$, and $j_1+j_2= j\leq l$,
the expression in \cref{Eqn::ODE::ProductDefined} is defined for $(x,h)\in \Omega_l$.
Finally,
\begin{equation*}
\begin{split}
    &\left|\omega(|h|)^{-j} \tau_{k_1 h} \left(\DiffOp{h}^{j_1} \partial_x^{\beta_1} A \right) \tau_{k_2 h} \left( \DiffOp{h}^{j_2} \partial_x^{\beta_2} B \right)\right|
    =\left|\tau_{k_1 h} \left(\omega(|h|)^{-j_1}\DiffOp{h}^{j_1} \partial_x^{\beta_1} A \right) \tau_{k_2 h} \left( \omega(|h|)^{-j_2}\DiffOp{h}^{j_2} \partial_x^{\beta_2} B \right)\right|
    \\&\leq \Norm{A}[C^{|\beta_1|,j_1,\omega}]\Norm{B}[C^{|\beta_2|,j_2,\omega}]
    \leq \Norm{A}[C^{m,l,\omega}]\Norm{B}[C^{m,l,\omega}],
\end{split}
\end{equation*}
where the last inequality follows from \cref{Eqn::ODE::InclusionIneq}.  The result follows.
\end{proof}

Define $\sT:C(B^{n}(\eta);\M^{n\times n})\rightarrow C(B^{n}(\eta); \M^{n\times n})$ by
\begin{equation*}
    \sT(A)(x)=\int_0^1 -A(sx)^2-C(sx)A(sx)-C(sx)\:ds.
\end{equation*}
The relevance of $\sT$ is the following lemma.
\begin{lemma}\label{Lemma::ODE::RelvanceOfsT}
$A\in C(B^n(\eta);\M^{n\times n})$ is a solution to \cref{Eqn::ODE::TheODE}
if and only if $\sT(A)=A$.  Also, writing $x=r\theta$, we have the following formula for $\sT$
when $r>0$:
\begin{equation}\label{Eqn::ODE::PolarT}
    \sT(A)(r\theta)= \frac{1}{r}\int_0^r -A(s\theta)^2-C(s\theta)A(s\theta)-C(s\theta)\: ds.
\end{equation}
\end{lemma}
\begin{proof}
\cref{Eqn::ODE::PolarT} follows from a straightforward change of variables in the
definition of $\sT$.  That $A\in C(B^n(\eta);\M^{n\times n})$ is a solution to \cref{Eqn::ODE::TheODE} if and only
if $\sT(A)=A$ follows from \cref{Eqn::ODE::PolarT}.
\end{proof}

\begin{lemma}\label{Lemma::ODE::sTPreserves}
If $C\in C^{m,l,\omega}$, then $\sT:C^{m,l,\omega}\rightarrow C^{m,l,\omega}$.
\end{lemma}
\begin{proof}
Let $A\in C^{m,l,\omega}$.  We wish to show $\sT(A)\in C^{m,l,\omega}$.
Set $B:=-A^{2}-CA-C$.  By \cref{Lemma::ODE::BasicCmlProperties}, $B\in C^{m,l,\omega}$.
We wish to show $\int_0^1 B(sx)\:ds\in C^{m,l,\omega}$.

Let $0\leq j\leq l$, $|\beta|\leq m$.  Consider,
\begin{equation*}
    \left|\DiffOp{h}^j \partial_x^{\beta} \int_0^1 B(sx)\: ds\right|
    =\left| \int_0^1 s^{|\beta|} (\DiffOp{sh}^j \partial_x^{\beta} B )(sx) \right|
    \leq \int_0^1 s^{|\beta|} \omega(s|h|)^{j} \Norm{B}[C^{m,l,\omega}]\: ds
    \leq \omega(|h|)^j \Norm{B}[C^{m,l,\omega}] (|\beta|+1)^{-1},
\end{equation*}
where we have used that $\omega$ is non-decreasing.
The result follows.
\end{proof}

\begin{lemma}[Izzo's contraction mapping principle \cite{IzzoCrConvergenceOfPicard}]\label{Lemma::ODE::Izzo}
Suppose $(X,d)$ is a metric space and $\{\sQ_a\}_{a=0}^\infty$ is a sequence
of contractions on $X$ for which there exists $c<1$ with
\begin{equation*}
    d(\sQ_a(x),\sQ_a(y))\leq c d(x,y), \quad \forall x,y\in X, a\in \N.
\end{equation*}
Suppose $\exists x_\infty\in X$ with $\lim_{a\rightarrow \infty} \sQ_a(x_\infty)=x_\infty$.
Let $x_0\in X$ be arbitrary, and define $x_a$ recursively by
$x_{a+1}=\sQ_a(x_a)$.  Then $\lim_{a\rightarrow \infty} x_a=x_\infty$.
\end{lemma}
\begin{proof}
We include a slightly modified version of the proof in \cite{IzzoCrConvergenceOfPicard}.
For each $a\in \N$,
\begin{equation}\label{Eqn::ODE::MainIzzoIneq}
\begin{split}
    &d(x_{a+1},x_{\infty})= d(\sQ_a(x_a),x_{\infty})\leq d(\sQ_a(x_a), \sQ_a(x_\infty))+ d(\sQ_a(x_\infty),x_\infty)
    \\&\leq cd(x_a,x_\infty)+d(\sQ_a(x_\infty),x_{\infty}).
\end{split}
\end{equation}

First we claim that the sequence $d(x_a,x_\infty)$ is bounded.  Since $\sQ_a(x_\infty)\rightarrow x_\infty$, $\exists N$, $a\geq N\Rightarrow d(\sQ_a(x_\infty),x_\infty)<1-c$.
Suppose $d(x_a,x_\infty)$ is not bounded; then $\exists a\geq N$ with $\max\{d(x_a,x_\infty),1\}\leq d(x_{a+1},x_\infty)$.
Applying this to \cref{Eqn::ODE::MainIzzoIneq}, we have $d(x_{a+1},x_\infty)\leq cd(x_a,x_\infty)+d(\sQ_a(x_\infty),x_\infty)<cd(x_{a+1},x_\infty)+1-c$.
And so $d(x_{a+1},x_{\infty})<1\leq d(x_{a+1},x_\infty)$, a contradiction.  Thus
the sequence $d(x_a,x_\infty)$ is bounded.

Since $\sQ_a(x_\infty)\rightarrow x_\infty$, \cref{Eqn::ODE::MainIzzoIneq} implies
$\limsup_{a\rightarrow \infty} d(x_a,x_\infty) \leq c\limsup_{a\rightarrow \infty} d(x_a,x_\infty)$.
Since $\limsup_{a\rightarrow \infty} d(x_a,x_\infty)<\infty$, this gives
$\limsup_{a\rightarrow \infty} d(x_a,x_\infty)=0$, completing the proof.
\end{proof}

We now turn to \cref{Prop::ODE::ExistMainProp}.  We begin with uniqueness.
Suppose $A_1,A_2\in C(B^{n}(\eta);\M^{n\times n})$ are two solutions
to \cref{Eqn::ODE::TheODE} with $A_1(0)=A_2(0)=0$.
By \cref{Lemma::ODE::RelvanceOfsT} we have $\sT(A_1)=A_1$, $\sT(A_2)=A_2$.
We first claim that $|A_j(x)|=O(|x|)$ for $j=1,2$; we prove this for $A_1$
and the same is true for $A_2$ by symmetry.
Set $F(r)=\sup_{|x|\leq r} |A_1(x)|$, note that $F:[0,\eta)\rightarrow \R$
is continuous, increasing, and $F(0)=0$.
Since $\sT(A_1)=A_1$ and $|C(sx)|\leq Ds|x|$ by assumption,
we have
\begin{equation*}
    |A_1(x)|\leq \int_0^1 F(s|x|)^2+Ds|x|F(s|x|)+ Ds|x|\: ds
    \leq F(|x|)^2 + \frac{1}{2} D|x|F(|x|)+\frac{1}{2}D|x|.
\end{equation*}
And so $F(r)\leq F(r)^2+\frac{1}{2}DrF(r)+\frac{1}{2}D r$, and thus
$F(r)(1-F(r))\leq \frac{1}{2}Dr F(r)+ \frac{1}{2} Dr$.
Taking $r$ so small that $F(r)\leq \frac{1}{2}$, we have  for such $r$, $F(r)\leq \frac{3}{2} D r$.
Thus $|A_1(x)|=O(|x|)$.

Writing $x$ in polar coordinates $x=r\theta$ and using \cref{Eqn::ODE::PolarT}
we have for $r>0$,
\begin{equation*}
    | r(A_1(r\theta)-A_2(r\theta))| \leq \int_0^r |s(A_1(s\theta)-A_2(s\theta))| \left( s^{-1}|A_1(s\theta)| +s^{-1} |A_2(s\theta)| + s^{-1}|C(s\theta)| \right)\: ds.
\end{equation*}
Using that $|A_1(s\theta)|,|A_2(s\theta)|,|C(s\theta)|=O(s)$,
the integral form of Gr\"onwall's inequality shows that $A_1(r\theta)=A_2(r\theta)$
for $r>0$ and therefore $A_1=A_2$.  This completes the proof of uniqueness.

We now turn to existence for which we use the contraction mapping principle.
Let
\begin{equation*}
    \sM:=\left\{A\in C^{0}(B^n(\eta);\M^{n\times n}) \:\bigg|\: A(0)=0, \sup_{0\ne x\in B^n(\eta)} \frac{1}{|x|}|A(x)|<\infty, \sup_{x\in B^{n}(\eta)} |A(x)|\leq \frac{1}{10} \right\}.
\end{equation*}
We give $\sM$ the metric
\begin{equation*}
    d(A,B):=\sup_{0\ne x\in B^n(\eta)} \frac{1}{|x|}\left|A(x)-B(x)\right|.
\end{equation*}
With this metric, $\sM$ is a complete metric space.

\begin{lemma}\label{Lemma::ODE::Contraction}
$\sT:\sM\rightarrow \sM$ and $\forall A,B\in \sM$, $d(\sT(A),\sT(B))\leq \frac{1}{5}d(A,B)$.
Also, $d(\sT(0),0)\leq D/2$.
\end{lemma}
\begin{proof}
Let $A\in \sM$.  For $x\in B^n(\eta)$,
\begin{equation}\label{Eqn::ODE::Regularity::BasicContractBound}
    |\sT(A)(x)|\leq \int_0^1 \Norm{A}[C^0]^2 + Ds|x|\Norm{A}[C^0]+ Ds|x|\: ds
    \leq \frac{1}{100}+\frac{D}{2}\eta\frac{1}{10}+ \frac{D}{2}\eta \leq \frac{1}{100}+\frac{1}{200}+\frac{1}{20}\leq \frac{1}{10}.
\end{equation}
Also,
\begin{equation}\label{Eqn::ODE::Reguarity::Boundst0}
    \frac{1}{|x|} |\sT(0)(x)|\leq \frac{1}{|x|} \int_0^1 Ds|x|\: ds\leq \frac{1}{2}D,
\end{equation}
and so $\sT(0)\in \sM$ with $d(\sT(0),0)\leq D/2$.

Finally, for $A,B\in \sM$, $0\ne x\in B^n(\eta)$,
\begin{equation}\label{Eqn::ODE::Regularity::Contract}
\begin{split}
    &\frac{1}{|x|} |\sT(A)(x)-\sT(B)(x)|\leq \frac{1}{|x|}\int_0^1 |A(sx)-B(sx)|(|A(sx)|+|B(sx)|+|C(sx)|)\: ds
    \\&\leq \frac{1}{|x|} \int_0^1 s|x|d(A,B)\left(\frac{1}{5}+Ds|x|\right)\: ds
    \leq \int_0^1 s d(A,B)\left(\frac{1}{5}+\frac{s}{10}\right)\: ds
    \\&\leq d(A,B)\left(\frac{1}{10}+\frac{1}{30}\right)\leq \frac{1}{5}d(A,B).
\end{split}
\end{equation}
Putting $0=B$ in \cref{Eqn::ODE::Regularity::Contract} and using
\cref{Eqn::ODE::Reguarity::Boundst0} shows
$\sup_{0\ne x\in B^n(\eta)}\frac{1}{|x|}|\sT(A)(x)|<\infty$.  Combining this with
\cref{Eqn::ODE::Regularity::BasicContractBound} shows $\sT:\sM\rightarrow \sM$.
Further, \cref{Eqn::ODE::Regularity::Contract} with arbitrary $A,B\in \sM$
shows $d(\sT(A),\sT(B))\leq \frac{1}{5}d(A,B)$, and this completes the proof.
\end{proof}

By \cref{Lemma::ODE::Contraction}, $\sT:\sM\rightarrow\sM$ is a strict
contraction, and the contraction mapping principle applies
to show that if $A_0=0$, $A_a=\sT(A_{a-1})$, $a\geq 1$,
then $A_a\rightarrow A_\infty$ in $\sM$, where $\sT(A_\infty)=A_{\infty}$.
$A_{\infty}$ is the desired solution to \cref{Eqn::ODE::TheODE}.

Also, for $a\in \N\cup\{\infty\}$ we have, using \cref{Lemma::ODE::Contraction},
\begin{equation}\label{Eqn::ODE::BoundsForAa}
\frac{1}{|x|} |A_a(x)| \leq d(A_a,0) \leq \sum_{b=0}^{a-1} d(\sT^{b+1}(0), \sT^b(0))
\leq \sum_{b=0}^{a-1} 5^{-b} d(\sT(0),0)\leq \frac{5}{8} D.
\end{equation}
In particular, for $x\in B^n(\eta)$, $|A_\infty(x)|\leq \frac{5}{8} D|x|$.
Also, since $\eta\leq (10D)^{-1}$, it follows that $|A_\infty(x)|\leq \frac{1}{16}$;
this establishes \cref{Eqn::ODE::ExistMainProp::SupBounds}.

It remains to prove the regularity properties of $A_{\infty}$, in terms of the regularity
of $C$.  For the remainder of this section, $K_{n,m,l,\omega}$ is a constant
which can be chosen to depend only on $n$, $m$, $l$, and an upper bound
for $\Norm{C}[C^{m,l,\omega}]$.  This constant may change from line to line.

To complete the proof of \cref{Prop::ODE::ExistMainProp}, we will prove the following when $C\in C^{m,l,\omega}$:
\begin{itemize}
    \item $A_a\rightarrow A_\infty$ in $C^{m,l,\omega}$.
    \item $\Norm{A_\infty}[C^{m,l,\omega}]\leq K_{n,m,l,\omega}$.
\end{itemize}
We prove the above two properties by induction on $m,l$.  The base case, $m=l=0$,
was just proved above (since $C^{0,0,\omega}=C^{0}(B^n(\eta);\M^{n\times n})$).

Fix $(m,l)$.  We assume we have the above for all $(k,j)$ with $0\leq k\leq m$,
$0\leq j\leq l$, and $(k,j)\ne (m,l)$, and we assume $C\in C^{m,l,\omega}$.
Since for $0\leq k\leq m$, $0\leq j\leq l$, $C^{m,l,\omega}\hookrightarrow C^{k,j,\omega}$
(\cref{Lemma::ODE::BasicCmlProperties}), the inductive hypothesis
shows for such $(k,j)$ with $(k,j)\ne (m,l)$,
$A_a\rightarrow A_\infty$ in $C^{k,j,\omega}$ and $\Norm{A_\infty}[C^{k,j,\omega}]\leq K_{n,k,j,\omega}$.

We define a Banach space $X_{\omega,l}$ as follows:
\begin{itemize}
    \item $X_{\omega,0}=C(B^n(\eta); \M^{n\times n})$, with the usual norm.
    \item For $l>0$, $X_{\omega,l}=\{B(x,h)\in C(\Omega_l;\M^{n\times n}) : \Norm{B}[X_{\omega,l}]<\infty\}$, where $\Norm{B}[X_{\omega,l}]:=\sup_{(x,h)\in \Omega_l} \omega(|h|)^{-l} |B(x,h)|$.
\end{itemize}
Fix $|\beta|=m$.  We will show (under our inductive hypothesis)
\begin{enumerate}[(i)]
    \item\label{Item::ODE::Aagood} $A_a\in C^{m,l,\omega}$, $\forall a\in \N$.
    \item\label{Item::ODE::AaDeriv} $\DiffOp{h}^l \partial_x^{\beta} A_a(x)\in X_{\omega,l}$, $\forall a\in \N$.
    \item\label{Item::ODE::AaConverge} $\exists B_\infty\in X_{\omega,l}$ such that $\DiffOp{h}^l \partial_x^{\beta} A_a\xrightarrow{a\rightarrow \infty} B_{\infty}$ in $X_{\omega,l}$.
    \item\label{Item::ODE::BinftyBound} $\Norm{B_\infty}[X_{\omega,l}]\leq K_{n,m,l,\omega}$.
\end{enumerate}

First we see why the above completes the proof.  We already know from our inductive
hypothesis that
\begin{equation}\label{Eqn::ODE::Exist::WahtWeKnow}
    \sup_{(x,h)\in \Omega_j} \omega(|h|)^{-j} \left|\DiffOp{h}^j \partial_x^{\alpha} (A_a-A_\infty) \right|\xrightarrow{a\rightarrow \infty} 0,
\end{equation}
for $0\leq j\leq l$, $|\alpha|\leq m$ with $(j,|\alpha|)\ne (l,m)$,
and that $\Norm{A_\infty}[C^{j,k,\omega}]\leq K_{n,k,j,\omega}$ for $0\leq k\leq m$, $0\leq j\leq l$,
$(j,k)\ne (l,m)$.
Thus, that $A_a\rightarrow A_\infty$ in $C^{m,l,\omega}$ will follow
from \cref{Eqn::ODE::Exist::WahtWeKnow} for $(j,|\alpha|)=(l,m)$ and the fact that $A_a\in C^{m,l,\omega}$.
If $l=0$, \cref{Item::ODE::Aagood} implies $A_a\in C^m$ and \cref{Item::ODE::AaConverge} implies
$\partial_x^{\beta} A_a\rightarrow B_\infty$ in the supremum norm.
Since $A_a\rightarrow A_\infty$ in $C^0$, we have $\partial_x^{\beta} A_\infty=B_\infty$.
\cref{Item::ODE::BinftyBound} implies the desired bound on $\partial_x^{\beta} A_\infty$.
Since $\beta$ is arbitrary with $|\beta|=m$, we conclude $A_\infty\in C^m$, with $\Norm{A_\infty}[C^{m,0,\omega}]\leq K_{n,m,0,\omega}$, and $A_a\rightarrow A_\infty$ in $C^{m,0,\omega}$,
as desired.

If $l\geq 1$, then we already know $A_a\rightarrow A_\infty$ in $C^m(B^n(\eta);\M^{n\times n})$, by the
inductive hypothesis.  Thus
\begin{equation*}
    \DiffOp{h}^l \partial_x^{\beta} A_a(x)\rightarrow \DiffOp{h}^l\partial_x^{\beta} A_\infty, \quad \text{pointwise}.
\end{equation*}
Hence, $\DiffOp{h}^l \partial_x^{\beta} A_\infty (x) = B_{\infty}(x,h)$.
Since $\beta$ was arbirary with $|\beta|=m$, \cref{Item::ODE::AaConverge} shows
$A_a\rightarrow A_\infty$ in $C^{m,l,\omega}$ and \cref{Item::ODE::BinftyBound}
shows $\Norm{A_\infty}[C^{m,l,\omega}]\leq K_{n,m,l,\omega}$.

Having shown them to be sufficient, we turn to
proving \cref{Item::ODE::Aagood}, \cref{Item::ODE::AaDeriv}, \cref{Item::ODE::AaConverge}, and \cref{Item::ODE::BinftyBound}.
Recall, we have fixed $\beta$ with $|\beta|=m$.
Since $A_a=\sT^{a}(0)$, \cref{Item::ODE::Aagood} follows from \cref{Lemma::ODE::sTPreserves}.
\cref{Item::ODE::AaDeriv} is an immediate consequence of \cref{Item::ODE::Aagood}.
Thus, it remains only to prove \cref{Item::ODE::AaConverge,Item::ODE::BinftyBound}.
We will do this by applying \cref{Lemma::ODE::Izzo}.  To begin, we need a few preliminary lemmas.

\begin{lemma}\label{Lemma::ODE::BilinCont}
Fix $m_1,l_1,m_2,l_2,j_1,j_2\in \N$ and set $l=l_1+l_2$ and suppose $j_1+l_1, j_2+l_2\leq l$.
Let $\beta_1$ and $\beta_2$ be multi-indicies with $|\beta_1|=m_1$ and $|\beta_2|=m_2$.
Then, the bilinear map for $A_1\in C^{m_1,l_1,\omega}$, $A_2\in C^{m_2,l_2,\omega}$
given by
\begin{equation}\label{Eqn::ODE::BilinearExpression}
    (A_1,A_2)\mapsto \left(\tau_{j_1 h} \DiffOp{h}^{l_1} \partial_x^{\beta_1} A_1\right)(x) \left(\tau_{j_2 h} \DiffOp{h}^{l_2} \partial_x^{\beta_2} A_2\right)(x).
\end{equation}
is a continuous map $C^{m_1,l_1,\omega}\times C^{m_2,l_2,\omega}\rightarrow X_{\omega,l}$,
and the norm of this map is $\leq 1$.
Here, $\tau_h A(x)= A(x+h)$.
\end{lemma}
\begin{proof}
The restriction $j_1+l_1, j_2+l_2\leq l$, ensures
that the expression in \cref{Eqn::ODE::BilinearExpression} is defined for $(x,h)\in \Omega_l$.
With this in mind, the result follows immediately from the definitions.
\end{proof}

For an element $B\in X_{\omega,l}$ we often write $B(x,h)$.  When $l\geq 1$, the meaning
of this is obvious.  For $l=0$ this is to be interpreted as $B(x)$.

\begin{lemma}\label{Lemma::ODE::IntCont}
For $B(x,h)\in X_{\omega,l}$ and $d\geq 1$ the map
\begin{equation*}
    B\mapsto \int_0^1 s^d B(sx,sh)\: ds
\end{equation*}
is continous $X_{\omega,l}\rightarrow X_{\omega,l}$ and has norm $\leq 1$.
\end{lemma}
\begin{proof}
This is clear from the definitions.
\end{proof}

For $A_1,A_2\in C^{m,\omega,l}$, we have
\begin{equation*}
    \DiffOp{h}^l \partial_x^{\beta}(A_1 A_2)(x) = (\DiffOp{h}^l \partial_x^{\beta} A_1)(x) A_2(x+lh)+
    A_1(x) (\DiffOp{h}^l \partial_x^{\beta} A_2)(x) + R_{\beta,l}(A_1,A_2)(x),
\end{equation*}
where $R_{\beta,l}(A_1,A_2)(x,h)$ is a constant coefficient linear combination (depending only on $\beta$ and $l$) of terms of the form
\begin{equation*}
    \left(\tau_{j_1 h} \DiffOp{h}^{l_1} \partial_x^{\beta_1} A_1\right)(x) \left(\tau_{j_2 h} \DiffOp{h}^{l_2} \partial_x^{\beta_2} A_2\right)(x),
\end{equation*}
where $0\leq j_1\leq l_2$, $0\leq j_2\leq l_1$, $l_1+l_2=l$, $\beta_1+\beta_2=\beta$,
and $l_1+|\beta_1|, l_2+|\beta_2|>0$.

\begin{lemma}\label{Lemma::ODE::ConvergeInts}
We have the following limits in $X_{\omega,l}$:
\begin{equation}\label{Eqn::ODE::Limit1}
\int_0^1 s^{|\beta|} R_{\beta,l}(A_a,A_a)(sx,sh)\: ds\xrightarrow{a\rightarrow\infty} \int_0^1 s^{|\beta|} R_{\beta,l}(A_\infty,A_\infty)(sx,sh)\: ds.
\end{equation}
\begin{equation}\label{Eqn::ODE::Limit2}
\int_0^1 s^{|\beta|} R_{\beta,l}(C,A_a)(sx,sh)\: ds\xrightarrow{a\rightarrow\infty} \int_0^1 s^{|\beta|} R_{\beta,l}(C,A_\infty)(sx,sh)\: ds.
\end{equation}
\begin{equation}\label{Eqn::ODE::Limit5}
    \int_0^1 s^{|\beta|} \left(  \DiffOp{sh}^l \partial_x^{\beta} C \right)(sx) A_a(s(x+lh))\:ds \xrightarrow{a\rightarrow \infty} \int_0^1 s^{|\beta|} \left(\DiffOp{sh}^l \partial_x^{\beta} C\right)(sx) A_\infty(s(x+lh))\: ds.
\end{equation}
And for any $B(x,h)\in X_{\omega,l}$,
\begin{equation}\label{Eqn::ODE::Limit3}
    \int_0^1 s^{|\beta|} B(sx, sh) A_a(s(x+lh))\: ds\xrightarrow{a\rightarrow \infty} \int_0^1 s^{|\beta|} B(s,x) A_\infty(s(x+lh))\: ds.
\end{equation}
\begin{equation}\label{Eqn::ODE::Limit4}
    \int_0^1 s^{|\beta|} A_a(sx)B(sx, sh) \: ds\xrightarrow{a\rightarrow \infty} \int_0^1 s^{|\beta|} A_\infty(sx) B(sx,sh) \: ds.
\end{equation}
\end{lemma}
\begin{proof}
Recall, we are assuming $C\in C^{m,\omega,l}$ and our inductive hypothesis implies $A_a\rightarrow A_\infty$ in $C^{k,j,\omega}$ with $0\leq k\leq m$, $0\leq j\leq l$, and $(k,j)\ne (m,l)$.
Using this and \cref{Lemma::ODE::BilinCont,Lemma::ODE::IntCont},
\cref{Eqn::ODE::Limit1}, \cref{Eqn::ODE::Limit2}, and \cref{Eqn::ODE::Limit5} follow immediately.
\Cref{Eqn::ODE::Limit3,Eqn::ODE::Limit4} follow from the fact that $A_a\rightarrow A_\infty$ in $C^0(B^{n}(\eta))$ and a straightforward estimate.
\end{proof}

\begin{lemma}\label{Lemma::ODE::BoundInts}
\begin{equation*}
    \left\| \int_0^1 s^{|\beta|} R_{\beta,l}(A_\infty,A_\infty)(sx,sh)\: ds  \right\|_{X_{\omega,l}} \leq K_{n,m,l,\omega}.
\end{equation*}
\begin{equation*}
    \left\| \int_0^1 s^{|\beta|} R_{\beta,l}(C,A_\infty)(sx,sh)\: ds  \right\|_{X_{\omega,l}} \leq K_{n,m,l,\omega}.
\end{equation*}
\begin{equation*}
    \left\| \int_0^1 s^{|\beta|}  \left(\DiffOp{sh}^l \partial_x^{\beta} C\right)(sx)A_\infty(s(x+lh)) \: ds  \right\|_{X_{\omega,l}} \leq K_{n,m,l,\omega}.
\end{equation*}
\end{lemma}
\begin{proof}
This follows from the inductive hypothesis and \cref{Lemma::ODE::BilinCont,Lemma::ODE::IntCont}.
\end{proof}

For $a\in \N\cup\{\infty\}$, $B\in X_{\omega,l}$, define
\begin{equation*}
\begin{split}
    \sQ_a(B)(x,h) = \int_0^1 -s^{|\beta|}
    \bigg[&
     B(sx,sh) A_a(s(x+lh)) +A_a(sx) B(sx,sh)+C(sx)B(sx,sh)
     \\&+ \left(\DiffOp{sh}^l \partial_x^{\beta} C\right)(sx) A_a(s(x+lh))
     +\left(\DiffOp{sh}^l \partial_x^\beta C\right)(sx)
     \\&+R_{\beta,l}(A_a, A_a)(sx, sh) + R_{\beta,l}(C,A_a)(sx,sh)
    \bigg]\:ds
\end{split}
\end{equation*}

\begin{lemma}\label{Lemma::ODE::QContractions}
For $a\in \N\cup\{\infty\}$, $\sQ_a:X_{\omega,l}\rightarrow X_{\omega,l}$
and satisfies
\begin{equation}\label{Eqn::ODE::sQaContraction}
    \Norm{\sQ_a(B)-\sQ_a(B')}[X_{\omega,l}]\leq \frac{1}{8} \Norm{B-B'}[X_{\omega,l}].
\end{equation}
Furthermore, $\forall B\in X_{\omega,l}$, $\lim_{a\rightarrow \infty} \sQ_a(B)=\sQ_{\infty}(B)$.
Finally, $\Norm{\sQ_{\infty}(0)}[X_{\omega,l}]\leq K_{n,m,l,\omega}$.
\end{lemma}
\begin{proof}
That $\sQ_a:X_{\omega,l}\rightarrow X_{\omega,l}$ follows from \cref{Lemma::ODE::BilinCont,Lemma::ODE::IntCont}, the inductive hypothesis,
and the fact that $A_a\in C^0(B^n(\eta))$, $\forall a\in \N\cup\{\infty\}$.

That $\lim_{a\rightarrow \infty} \sQ_a(B)=\sQ_\infty(B)$ follows from
\cref{Lemma::ODE::ConvergeInts} and $\Norm{\sQ_{\infty}(0)}[X_{\omega,l}]\leq K_{n,m,l,\omega}$
follows from \cref{Lemma::ODE::BoundInts,Lemma::ODE::IntCont}.

Thus we need only show \cref{Eqn::ODE::sQaContraction}.
We have, using \cref{Eqn::ODE::BoundsForAa}, for $(x,h)\in \Omega_l$, $a\in \N\cup\{\infty\}$,
\begin{equation*}
\begin{split}
    &|\sQ_a(B)(x,h)-\sQ_a(B')(x,h)|
    \leq \int_0^1 s^{|\beta|} ( |A_a(s(x+lh)| + |A_a(sx)| + |C(sx)|) |B(sx,sh)-B'(sx,sh)|\: ds
    \\& \leq \int_0^1 \left(\frac{5}{8} Ds |x+lh| + \frac{5}{8} D s|x|+ Ds|x|\right) \omega(s|h|)^{l} \Norm{B-B'}[X_{\omega,l}]\: ds
    \leq \Norm{B-B'}[X_{\omega,l}] \int_0^1 \frac{9}{4} D s \eta \omega(|h|)^l\: ds
    \\&\leq \frac{1}{8} \omega(|h|)^l \Norm{B-B'}[X_{\omega,l}],
\end{split}
\end{equation*}
completing the proof of \cref{Eqn::ODE::sQaContraction}, and therefore the proof of the lemma.
\end{proof}

For $a\in \N$, define $B_a(x,h):=\DiffOp{h}^l \partial_x^{\beta}A_a(x)$; note that
$B_a\in X_{\omega,l}$ since $A_a\in C^{m,\omega,l}$.
Also,
$B_{a+1}(x,h) = \DiffOp{h}^l\partial_x^{\beta} \sT(A_a)(x)= \sQ_a(B_a)(x,h)$.

Since $\sQ_\infty$ is a strict contraction (\cref{Lemma::ODE::QContractions}), there exists a unique fixed point
$B_\infty\in X_{\omega,l}$.  Since $\sQ_a(B_\infty)\rightarrow \sQ_\infty(B_\infty)=B_\infty$,
by \cref{Lemma::ODE::QContractions}, \cref{Lemma::ODE::Izzo} shows
$B_a\rightarrow B_\infty$ in $X_{\omega,l}$.  Since $B_a(x,h)=\DiffOp{h}^l \partial_x^{\beta}A_a(x)$,
this proves \cref{Item::ODE::AaConverge}.

Finally, to prove \cref{Item::ODE::BinftyBound} note that $B_\infty$ is the fixed point
of the strict contraction $\sQ_\infty$.  Thus, $\sQ_\infty^a(0)\rightarrow B_{\infty}$.
Hence,
\begin{equation*}
    \Norm{B_{\infty}}[X_{\omega,l}] \leq \sum_{a=0}^\infty \Norm{\sQ_\infty^{a+1}(0)-\sQ_\infty^{a}(0)}[X_{\omega,l}]
    \leq \sum_{a=0}^\infty 8^{-a} \Norm{\sQ_\infty(0)-0}[X_{\omega,l}]
    =\frac{8}{7} \Norm{\sQ_{\infty}(0)}[X_{\omega,l}]\leq K_{n,m,l,\omega},
\end{equation*}
where the last inequality follows from \cref{Lemma::ODE::QContractions}.
This completes the proof.

%% file: pfift.tex

We require a quantitative version of a special case of the Inverse Function Theorem that does
not follow from the standard statement of the theorem, though we will be able to achieve it
by keeping track of some constants in a standard proof.  We present it here.

Fix $\eta>0$ and let $Y_1,\ldots, Y_n\in C^1(B^n(\eta);\R^n)$
be vector fields on $B^{n}(\eta)$ and suppose they satisfy
$$\inf_{u\in B^n(\eta)} \left|\det \left(Y_1(u)|\cdots|Y_n(u)\right)\right|\geq c_0>0.$$
Take $C_0>0$ so that $\Norm{Y_j}[C^1(B^n(\eta);\R^n)]\leq C_0$, $\forall j$.
Define
\begin{equation*}
\Psi_u(v):= e^{v_1 Y_1+\dots + v_n Y_n} u.
\end{equation*}

\begin{prop}\label{Prop::AppIFT::MainProp}
There exist $\kappa=\kappa(C_0,c_0,n)>0$ and $\Delta_0=\Delta_0(C_0,c_0,n,\eta)>0$
such that $\forall \delta\in (0,\Delta_0]$, $u\in B^{n}(\kappa\delta)$,
$v\mapsto \Psi_u(v)$ is defined and injective on $v\in B^n(\delta)$.  Furthermore,
$B^n(\kappa\delta)\subseteq \Psi_u(B^n(\delta))$.
\end{prop}

The rest of this section is devoted to the proof of \cref{Prop::AppIFT::MainProp}; for a closely related result
see \cite[Theorem 4.5]{MontanariMorbidelliStepSInvolutiveFamiliesOfVectorFields}.

\begin{lemma}\label{Lemma::AppIFT::FLemma}
Let $\delta_0>0$, $F\in C^1(B^n(\delta_0); \R^n)$, and suppose $dF(0)$ is nonsingular
and $\sup_{x\in B^n(\delta_0)}\Norm{dF(0)^{-1}dF(x)-I}[\M^{n\times n}] \leq \frac{1}{2}$.
Then $F(B^n(\delta_0))\subseteq \R^n$ is open and $F:B^n(\delta_0)\rightarrow F(B^n(\delta_0))$
is a $C^1$ diffeomorphism.
Furthermore,
$F(B^n(\delta_0))\supseteq B^n(F(0), \kappa\delta_0)$ where
\begin{equation}\label{Eqn::AppIFT::Boundkappa}
    \kappa:= \CjNorm{d (F^{-1})}{0}[F(B^n(\delta_0))][\M^{n\times n}]^{-1}\geq c_n |\det dF(0)| \CjNorm{F}{1}[B^n(\delta_0)][\R^n]^{-(n-1)},
\end{equation}
and $c_n>0$ can be chosen to depend only on $n$.
\end{lemma}
\begin{proof}
We first show $F$ is injective.  Fix $y\in \R^n$
and set $\phi(x)=x+dF(0)^{-1}(y-F(x))$.  Note that $F(x)=y \Leftrightarrow \phi(x)=x$.
Also, $\forall x\in B^n(\delta_0)$,  $\Norm{d\phi(x)}[\M^{n\times n}]\leq \Norm{ I-dF(0)^{-1} dF(x)}[\M^{n\times n}]\leq \frac{1}{2}$.
Hence $|\phi(x_1)-\phi(x_2)|\leq \frac{1}{2}|x_1-x_2|$.  Hence, there is at most one solution
of $\phi(x)=x$, and therefore at most one solution of $F(x)=y$, proving that $F$ is injective.

Since $\Norm{dF(0)^{-1}dF(x)-I}[\M^{n\times n}] \leq \frac{1}{2}$, $\forall x\in B^n(\delta_0)$, it follows that $dF(x)$ is invertible $\forall x\in B^n(\delta_0)$.
Combining this with the fact that $F$ is injective,
the Inverse Function Theorem shows $F(B^n(\delta_0))$ is open and $F:B^n(\delta_0)\rightarrow F(B^n(\delta_0))$ is a $C^1$ diffeomorphism.

Next we prove the bound for $\kappa$ given in \cref{Eqn::AppIFT::Boundkappa}.
In what follows, we use $A\lesssim B$ to denote $A\leq C_n B$, where $C_n$ can be chosen to depend only on $n$.
Since $\|dF(0)^{-1} dF(x)- I\|\leq \frac{1}{2}$, by assumption,
\begin{equation}\label{Eqn::AppIFT::BoundF1}
    \inf_{x\in B^n(\eta)} |\det dF(x)| \gtrsim |\det dF(0)|.
\end{equation}
Also, $\forall x\in B^n(\delta_0)$,
\begin{equation*}
    \Norm{ (dF(x))^{-1}}[\M^{n\times n}] \lesssim |\det dF(x)|^{-1} \CjNorm{dF}{0}[B^n(\delta_0)][\M^{n\times n}]^{n-1},
\end{equation*}
as can be seen via the cofactor representation $dF(x)^{-1}$.
Hence,
$$\sup_{x\in B^n(\delta_0)}\Norm{ (dF(x))^{-1}}[\M^{n\times n}] \lesssim \left(\inf_{y\in B^n(\delta_0)} |\det dF(y)|\right)^{-1} \CjNorm{dF}{0}[B^n(\delta_0)][\M^{n\times n}]^{n-1},$$
and therefore
\begin{equation}\label{Eqn::AppIFT::BoundF2}
\begin{split}
&\CjNorm{d(F^{-1})}{0}[F(B^n(\delta_0))][\M^{n\times n}] \lesssim \left( \inf_{x\in B^n(\delta_0)} |\det dF(x)|\right)^{-1} \CjNorm{dF}{0}[B^n(\delta_0)][\M^{n\times n}]^{n-1}
\\& \lesssim \left( \inf_{x\in B^n(\delta_0)} |\det dF(x)|\right)^{-1} \CjNorm{F}{1}[B^n(\delta_0)][\R^n]^{n-1}.
\end{split}
\end{equation}
Combining \cref{Eqn::AppIFT::BoundF1,Eqn::AppIFT::BoundF2} yields \cref{Eqn::AppIFT::Boundkappa}.

Finally, we prove $F(B^n(\delta_0))\supseteq B(F(0),\kappa\delta_0)$.
Take $\epsilon>0$ to be the largest $\epsilon$ so that $B^n(F(0),\epsilon)\subseteq F(B^n(\delta_0))$
(note that $\epsilon>0$ by the Inverse Function Theorem).
The proof will be complete once we show $\epsilon\geq \delta_0\kappa$.  Suppose, for contradiction, $\epsilon<\delta_0\kappa$.
We have, by the Mean Value Theorem,
\begin{equation*}
    F^{-1}(B(F(0),\epsilon))\subseteq B(0, \epsilon\CjNorm{dF^{-1}}{0}[F(B^n(\delta_0))][\M^{n\times n}] ).
\end{equation*}
Thus, if $\epsilon<\kappa\delta_0$, $F^{-1}(B(F(0),\epsilon))\Subset B(0,\delta_0)$,
which contradicts the choice of $\epsilon$ and completes the proof.
\end{proof}

\begin{lemma}\label{Lemma::AppIFT::LipDerivs}
Let $Y_j$, $C_0$, $n$, $\eta$, and $\Psi$ be as in \cref{Prop::AppIFT::MainProp}.
There exists $\delta_1=\delta_1(C_0,n,\eta)>0$ such that
$\forall u\in B^n(\eta/2)$, $\Psi_u$ is defined on $B^n(\delta_1)$
and satisfies
\begin{equation}\label{Eqn::AppIFT::C1Bound::Psi}
\Norm{\Psi_u}[C^1(B^n(\delta_1);\R^n)]\leq C(C_0,n)
\end{equation}
and $\forall u\in B^n(\eta/2), v\in B^n(\delta_1)$,
\begin{equation}\label{Eqn::AppIFT::PsiLip}
    \BNorm{ d_v\Psi_u(v) - d_v\Psi_u(0)}[\M^{n\times n}]\leq C(C_0,n)|v|,
\end{equation}
where $C(C_0,n)$ can be chosen to depend only on $C_0$ and $n$.
\end{lemma}
\begin{proof}
The existence of $\delta_1>0$ so that $\forall u\in B^n(\eta/2)$,  $\Psi_u(v)$ is defined and
\cref{Eqn::AppIFT::C1Bound::Psi} holds are classical theorems from ODEs.
Thus, we prove only \cref{Eqn::AppIFT::PsiLip}.  We write $A\lesssim B$ for $A\leq C B$
where $C$ can be chosen to depend only on $C_0$ and $n$.
We use the equation $\partial_r \Psi_u(rv)= v\cdot Y(\Psi_u(rv))$, and so
\begin{equation*}
    \Psi_u(v) = \int_0^1 v\cdot Y(\Psi_u(sv))\: ds.
\end{equation*}
Since $d_v \Psi_u(0)= (Y_1(u)|\cdots|Y_n(u))$, we have $\forall u\in B^n(\eta/2), v\in B^n(\delta_1)$
\begin{equation*}
    \Psi_u(v)- (d_v\Psi_u(0))v = \int_0^1 v\cdot \left( Y(\Psi_u(sv))- Y(\Psi_u(0)) \right)\: ds.
\end{equation*}
Applying $d_v$ to the above equation and using the chain rule, we have $\forall u\in B^n(\eta/2)$, $v\in B^n(\delta_1)$,
\begin{equation*}
\begin{split}
    &\BNorm{ d_v\Psi_u(v) - d_v\Psi_u(0) }[\M^{n\times n}]
    = 
    \BNorm{
        \int_0^1 (Y(\Psi_u)(sv) -Y(\Psi_u(0))) + sv^{\transpose} dY(\Psi_u(sv)) (d_v\Psi_u)(sv)\: ds
    }[\M^{n\times n}]
    \\&\lesssim |v| \CjNorm{Y\circ \Psi_u}{1}[B^n(\delta_0)][\M^{n\times n}] + |v| \CjNorm{Y}{1}[B^n(\eta)][\M^{n\times n}]\CjNorm{\Psi_u}{1}[B^n(\delta_1)][\R^n]
    \\&\lesssim |v|\CjNorm{Y}{1}[B^n(\eta)][\M^{n\times n}]\CjNorm{\Psi_u}{1}[B^n(\delta_1)][\R^n] \lesssim |v|,
\end{split}
\end{equation*}
where we have written $Y(u)$ for the matrix valued function $(Y_1(u)|\cdots|Y_n(u))$
and used \cref{Eqn::AppIFT::C1Bound::Psi}.
This completes the proof.
\end{proof}

\begin{proof}[Proof of \cref{Prop::AppIFT::MainProp}]
In what follows we write $A\lesssim B$ for $A\leq C B$, where $C$ can be chosen to depend only on
$n$, $C_0$, and $c_0$, and write $A\lesssim_\eta B$ if $C$ can also depend on $\eta$.
By taking $\delta_1\gtrsim_\eta 1$
as in \cref{Lemma::AppIFT::LipDerivs}, for all $u\in B^n(\eta/2)$, $v\in B^n(\delta_1)$,
$\Psi_u(v)$ is defined.  For such $u$, since $|\det d\Psi_u(0)|=|\det (Y_1(u)|\cdots|Y_n(u))|\gtrsim 1$
and using \cref{Eqn::AppIFT::C1Bound::Psi}, we have
$\Norm{ d_v\Psi_u(0)^{-1}}[\M^{n\times n}]\lesssim 1$.
Hence, using \cref{Eqn::AppIFT::PsiLip}, for $u\in B^n(\eta/2)$, $v\in B^n(\delta_1)$,
\begin{equation*}
    \Norm{d_v\Psi_u(0)^{-1} d_v\Psi_u(v) -I}[\M^{n\times n}] \lesssim \Norm{d_v\Psi_u(v) - d_v\Psi_u(0)}[\M^{n\times n}]\lesssim |v|.
\end{equation*}
Thus, if $\delta_2\gtrsim_\eta 1$ is sufficiently small, for all $u\in B^{n}(\eta/2)$,
$v\in B^{n}(\delta_2)$,
\begin{equation*}
    \Norm{d_v\Psi_u(0)^{-1} d_v\Psi_u(v) -I}[\M^{n\times n}]\leq \frac{1}{2}.
\end{equation*}
By \cref{Lemma::AppIFT::FLemma}, for $|u|\leq \eta/2$, $\Psi_u:B^{n}(\delta_2)\rightarrow \Psi_u(B^n(\delta_2))$ is a $C^1$ diffeomorphism,
and if we set $\kappa:=\frac{1}{2} \inf_{|u|<\eta/2}\Norm{d(\Psi_u^{-1})}[\CjSpace{0}[\Psi_u(B^n(\delta_2))][ \M^{n\times n} ]]^{-1}$,
we have $\kappa\gtrsim 1$ (also by \cref{Lemma::AppIFT::FLemma}).  Notice the extra factor of $1/2$ in the defintion
of $\kappa$ as compared to \cref{Lemma::AppIFT::FLemma}.

Take $\Delta_0<\delta_2$, $\Delta_0\gtrsim_\eta 1$ sufficiently small so that $\kappa\Delta_0<\eta/2$.
Then for $\delta\in (0,\Delta_0]$ and $|u|<\kappa\delta$, \cref{Lemma::AppIFT::FLemma}
shows
\begin{equation*}
    \Psi_u(B^n(\delta))\supseteq B^n(\Psi_u(0), 2\kappa\delta) = B^n(u, 2\kappa\delta)\supseteq B^n(0,\kappa\delta),
\end{equation*}
which completes the proof.
\end{proof}

%% file: pfmain.tex

We turn to the proof of \cref{Thm::Results::MainThm}.  We separate the proof into two parts:  when $X_1(x_0),\ldots, X_q(x_0)$ are linearly independent (i.e., when $n=q$), and more generally when $X_1(x_0),\ldots, X_q(x_0)$
may be linearly dependent (i.e., when $q\geq n$).

%% file: pfli.tex

In this section, we prove \cref{Thm::Results::MainThm} in the special case $n=q$.
We take the same setting as \cref{Thm::Results::MainThm} with the same notions of admissible constants, and with the additional assumption that $n=q$.
Note that, in this case, $X_{J_0}=X$, so we may replace $X_{J_0}$ with $X$
throughout the statement of \cref{Thm::Results::MainThm}.
Also, because $n=q$, in  $\Had{m_1,m_2,s}$-admissible constants, $m_2$ does not play a role (since in all of our results $m_1\geq m_2$ when $\Had{m_1,m_2,s}$ admissible constants are used),
so we instead use $\Hmad{m_1}{s}$-admissible constants throughout this section.
Similarly, we use $\Zygmad{s}$-admissible constants throughout this section.

\Cref{Prop::ResQual::Mmanif} implies that $B_X(x_0,\xi)$ is an $n$-dimensional
manifold and that $X_1,\ldots, X_n$ span the tangent space to every point of $B_X(x_0,\xi)$.
Thus, $X_1(y),\ldots, X_n(y)$ are linearly independent $\forall y\in B_X(x_0,\xi)$,
and \cref{Thm::Results::MainThm} \cref{Item::Results::WedgeNonzero} follows with $\chi=\xi$.
\cref{Item::Results::BigWedge,Item::Results::Open} are both obvious when $n=q$ (and $\chi=\xi$).
With \cref{Item::Results::WedgeNonzero}, \cref{Item::Results::BigWedge}, and \cref{Item::Results::Open} proved, we henceforth assume
$c_{j,k}^l\in \CXjSpace{X_{J_0}}{1}[B_{X_{J_0}}(x_0,\xi)]=\CXjSpace{X}{1}[B_X(x_0,\xi)]$, $1\leq j,k,l\leq n$.

Consider the map $\Phi:B^n(\eta_0)\rightarrow B_X(x_0,\xi)$ defined in \cref{Eqn::Results::DefinePhi}; which we a priori know to be $C^1$.
Clearly $d\Phi(0)\diff{t_j}=X_j(x_0)$.  Since $X_1(x_0),\ldots, X_n(x_0)$ form a basis
of the tangent space $T_{x_0}B_X(x_0,\xi)$, the Inverse Function Theorem shows
that there exists a (non-admissible) $\delta>0$ such that $\Phi:B^n(\delta)\rightarrow \Phi(B^n(\delta))$
is a $C^1$ diffeomorphism.  Let $\Yh_j:=\Phi\big|_{B^n(\delta)}^{*} X_j$,
so that $\Yh_j$ is a $C^0$ vector field on $B^n(\delta)$.
Write $\Yh_j=\diff{t_j}+\sum_k \ah_j^k(t) \diff{t_k}$.  Let $\Ah(t)\in \CSpace{B^n(\delta)}[\M^{n\times n}]$
denote the $n\times n$ matrix with $(j,k)$ component $\ah_j^k(t)$ and
 let $C(t)\in \CSpace{B^n(\eta_0)}[\M^{n\times n}]$ denote the $n\times n$ matrix
 with $(j,k)$ component equal to $\sum_{l=1}^{n} t_l c_{j,l}^k\circ \Phi(t)$.

\begin{prop}\label{Prop:PfLI::MainODEProp}
Write $t$ in polar coordinates, $t=r\theta$, and consider the differential equation
\begin{equation}\label{Eqn::PfLI::MainODE}
    \diff{r} rA(r\theta) = -A(r\theta)^2-C(r\theta)A(r\theta)-C(r\theta),
\end{equation}
defined for $A:B^n(\eta_0)\rightarrow \M^{n\times n}$.  There exists a $0$-admissible constant $\eta'>0$,
which also depends on a lower bound for $\eta>0$, such that there exists a unique continuous solution
$A\in \CSpace{B^n(\eta')}[\M^{n\times n}]$ to \cref{Eqn::PfLI::MainODE}
with $A(0)=0$.  Moreover, 
this solution 
lies in $\CjSpace{1}[B^n(\eta')][\M^{n\times n}]$ 
and satisfies
\begin{equation*}
    \Norm{A(t)}[\M^{n\times n}]\lesssim_0 |t|\text{ and }\Norm{A(t)}[\M^{n\times n}]\leq \frac{1}{2}, \quad \forall t\in B^n(\eta').
\end{equation*}
For $m\in \N$ and $s\in [0,1]$, if $c_{i,j}^k\circ \Phi\in \HSpace{m}{s}[B^n(\eta')]$
with $\HNorm{c_{i,j}^k\circ \Phi}{m}{s}[B^n(\eta')]\leq D_{m,s}$, $\forall i,j,k$,
then $A\in \HSpace{m}{s}[B^n(\eta')][\M^{n\times n}]$ and there exists a constant $C_{m,s}$,
which depends only on $n$, $m$, and $D_{m,s}$, such that
$\HNorm{A}{m}{s}[B^n(\eta')][\M^{n\times n}]\leq C_{m,s}$.  Similarly, for $s\in (0,\infty)$,
if $c_{i,j}^k\circ \Phi\in \ZygSpace{s}[B^n(\eta')]$ with
$\ZygNorm{c_{i,j}^k\circ \Phi}{s}[B^n(\eta')]\leq D_s$, then there exists a constant $C_s$
which depends only on $n$, $s$, and $D_s$ such that
$\ZygNorm{A}{s}[B^n(\eta')][\M^{n\times n}]\leq C_s$.
Finally, 
$\Ah\big|_{B^n(\min\{\eta',\delta\} )} = A\big|_{B^n(\min\{\eta',\delta\})}$.
%
\end{prop}
\begin{proof}
Note that, by the definition of $C(t)$ we have $\Norm{C(t)}[\M^{n\times n}]\lesssim_0 |t|$.
Also, $\Ah$ satisfies \cref{Eqn::PfLI::MainODE} on $B^n(\delta)$ by \cref{Prop::ODE::Derive::MainProp}.
Since $d\Phi(0)\diff{t_j}=X_j(x_0)$, we have $\Ah(0)=0$.
With these remarks in hand, the proposition (except for the claim 
$A\in \CjSpace{1}[B^n(\eta')][\M^{n\times n}]$)
follows directly from \cref{Prop::ODE::ExistMainProp}
(see also \cref{Rmk::ODE::ImportantSpaces}).

The claim that 
$A\in \CjSpace{1}[B^n(\eta')][\M^{n\times n}]$
can be seen as follows.
First note that we may assume $\eta'<\eta_0$ as if $\eta'=\eta_0$, we may replace $\eta'$ with $\eta_0/2$.
Since $c_{j,k}^l\in \CXjSpace{X_{J_0}}{1}[B_{X_{J_0}}(x_0,\xi)]=\CXjSpace{X}{1}[B_X(x_0,\xi)]$, 
$X_1,\ldots, X_n$ span the tangent space
at every point of $B_X(x_0,\xi)$, and the vector fields $X_1,\ldots, X_n$ are $C^1$, it follows that $c_{j,k}^l$ are $C^1$ on $B_X(x_0,\xi)$.
Since $\Phi:B^n(\eta_0)\rightarrow B_X(x_0,\xi)$ is a priori known to be $C^1$, we have $c_{j,k}^l\circ \Phi$ is $C^1$ on $B^n(\eta_0)$.  Thus, 
$C\in \CjSpace{1}[B^n(\eta')][\M^{n\times n}]$,
and it follows from \cref{Prop::ODE::ExistMainProp} that $A\in \CjSpace{1}[B^n(\eta')][\M^{n\times n}]$.
\end{proof}

We fix $\eta'>0$ and $A$ as in \cref{Prop:PfLI::MainODEProp}.  Write $a_j^k(t)$ for the $(j,k)$
component of $A(t)$ and set $Y_j:=\diff{t_j}+\sum_{k=1}^n a_j^k \diff{t_k}$.  Note that $Y_1,\ldots, Y_n$ are $C^1$ vector fields on $B^n(\eta')$.
By \cref{Prop:PfLI::MainODEProp}, $Y_j\big|_{B^n(\min\{\eta',\delta\} )}=\Yh_j\big|_{B^n(\min\{\eta',\delta\} )}$.  Since $\delta$ is not admissible, we think of $\delta$ as being much smaller than
$\eta'$, and so $Y_j$ should be thought of as extending $\Yh_j$.

\begin{prop}\label{Prop::PfLI::dPhiYj}
$\forall t\in B^n(\eta')$, $d\Phi(t)Y_j(t)=X_j(\Phi(t))$, $1\leq j\leq n$.
\end{prop}
\begin{proof}
Fix $\theta\in S^{n-1}$ and set
\begin{equation*}
    r_1:=\sup \{r\geq 0 : d\Phi(r'\theta) Y_j(r'\theta) =X_j(\Phi(r'\theta)), 0\leq r'\leq r, 1\leq j\leq n\}.
\end{equation*}
We wish to show $r_1=\eta'$, and this will complete the proof since $\theta\in S^{n-1}$ was
arbitrary.  Suppose, for contradiction, $r_1<\eta'$.
Since $Y_j\big|_{B^n(\min\{\eta',\delta\} )}=\Yh_j\big|_{B^n(\min\{\eta',\delta\} )}$
and $d\Phi(u)\Yh_j(u)=X_j(\Phi(u))$, we know $r_1>0$.
By continuity, we have
\begin{equation*}
    d\Phi(r_1\theta) Y_j(r_1\theta) = X_j(\Phi(r_1\theta)).
\end{equation*}
By \cref{Prop::ResQual::Mmanif}, $X_1(\Phi(r_1\theta)),\ldots, X_n(\Phi(r_1\theta))$
span $T_{\Phi(r_1\theta)}B_X(x_0,\xi)$, and therefore the Inverse Function
Theorem applies to $\Phi$ at the point $r_1\theta$.
Thus, there exists a neighborhood $V$ of $r_1\theta$ such that $\Phi:V\rightarrow \Phi(V)$ is a
$C^1$ diffeomorphism.
Pick $0<r_2<r_3<r_1<r_4<\eta'$ such that $\{r'\theta : r_2\leq r'\leq r_4\}\subset V$.

Let $\Yt_j:=\Phi\big|_V^{*} X_j$.  By the choice of $r_1$, for $r_2\leq r'\leq r_3$
we have $\Yt_j(r'\theta)=Y_j(r'\theta)$.  Write $\Yt_j=\diff{u_j}+\sum_{k=1}^n\at_j^k\diff{u_k}$
and let $\At$ denote the matrix with $(j,k)$ component $\at_j^k$.
We therefore have $\At(r'\theta)=A(r'\theta)$ for $r_2\leq r'\leq r_3$.
$\At$ satisfies \cref{Eqn::PfLI::MainODE} by \cref{Prop::ODE::Derive::MainProp}.
Away from $r=0$, \cref{Eqn::PfLI::MainODE} is a standard ODE
that both $A$ and $\At$ satisfy.  Thus, by standard uniqueness theorems (using, for example,
Gr\"onwall's inequality) we have $\At(r'\theta)=A(r'\theta)$ for $r_2\leq r'\leq r_4$.
Thus, $Y_j(r'\theta)=\Yt_j(r'\theta)$, $r_2\leq r'\leq r_4$.  Since
$d\Phi(r'\theta) \Yt_j(r'\theta)=X_j(\Phi(r'\theta))$ we conclude $r_1\geq r_4$.
This is a contradiction, completing the proof.
\end{proof}

\begin{lemma}\label{Lemma::PfLI::PhiC2}
$\Phi:B^n(\eta')\rightarrow B_X(x_0,\xi)$ is $C^2$.
\end{lemma}
\begin{proof}
Since we already know that $\Phi:B^n(\eta')\rightarrow B_X(x_0,\xi)$ is $C^1$,
it suffices to show the
map $u\mapsto d\Phi(u)$, $u\in B^n(\eta')$ is $C^1$.
We have already remarked that $Y_1,\ldots, Y_n$ are $C^1$.
Since $Y=(I+A)\grad$, with $\Norm{A(t)}[\M^{n\times n}]\leq \frac{1}{2}$, $\forall t$,
we conclude $Y_1,\ldots, Y_n$ are a basis for the tangent space at every point of $B^n(\eta')$.
Also, $d\Phi(u)Y_j(u)=X_j(\Phi(u))\in C^1$ since $X_j\in C^1$, $\Phi\in C^1$.
Since $Y_1,\ldots, Y_n$ are $C^1$ and a basis for the tangent space at every point, we conclude
$u\mapsto d\Phi(u)$ is $C^1$, and therefore $\Phi$ is $C^2$,
completing the proof.
\end{proof}

\begin{prop}\label{Prop::PfLI::EquivNorms}
For $m\in \N$, $s\in [0,1]$, $\eta''\in (0,\eta']$ we have (for any function $f$),
\begin{equation}\label{EqnPfLI::EquivNorm::HolderNorm}
    \HNorm{f}{m}{s}[B^n(\eta'')] \approx_{\Hmad{m-1}{s}} \HXNorm{f}{Y}{m}{s}[B^n(\eta'')],
\end{equation}
and
\begin{equation}\label{EqnPfLI::EquivNorm::HolderVF}
    \HNorm{Y_j}{m}{s}[B^n(\eta')][\R^n]\lesssim_{\Hmad{m}{s}} 1.
\end{equation}
Similarly, for $s\in (0,\infty)$,
\begin{equation}\label{EqnPfLI::EquivNorm::ZygNorm}
    \ZygNorm{f}{s}[B^n(\eta'')]\approx_{\Zygmad{s-1},\eta''} \ZygXNorm{f}{Y}{s}[B^n(\eta'')]
\end{equation}
and
\begin{equation}\label{EqnPfLI::EquivNorm::ZygVF}
    \ZygNorm{Y_j}{s}[B^n(\eta')][\R^n]\lesssim_{\Zygmad{s}} 1.
\end{equation}
In \cref{EqnPfLI::EquivNorm::ZygNorm} we have written $\approx_{\Zygmad{s-1},\eta''}$ to denote
that the implicit constants are also allowed to depend on the choice of $\eta''$.

Furthermore, for $m\in \N$, $s\in [0,1]$, and $1\leq i,j,k\leq n$, we have
\begin{equation}\label{EqnPfLi::EquivNorm::Pullbackcijk::Holder}
\HNorm{c_{i,j}^k\circ \Phi}{m}{s}[B^n(\eta')]\lesssim_{\Hmad{m}{s}} 1,
\end{equation}
and for $s\in (0,\infty)$,
\begin{equation}\label{EqnPfLi::EquivNorm::Pullbackcijk::Zyg}
\ZygNorm{c_{i,j}^k\circ \Phi}{s}[B^n(\eta')]\lesssim_{\Zygmad{s}} 1.
\end{equation}
\end{prop}
\begin{proof}
Since $\sup_{t\in B^n(\eta')}\Norm{A(t)}[\M^{n\times n}]\leq \frac{1}{2}$,
and $Y=(I+A) \grad$, we also have $\grad = (I+A)^{-1}Y$.  Thus, once we prove a certain
regularity on $A$, we can compare norms as in \cref{EqnPfLI::EquivNorm::HolderNorm,EqnPfLI::EquivNorm::ZygNorm}
by applying \cref{Prop::FuncSpaceRev::CompNorms}.
For example, once we show $\HNorm{A}{m}{s}[B^n(\eta')][\M^{n\times n}]\lesssim_{\Hmad{m}{s}} 1$, we will also have
$\HNorm{(I+A)^{-1}}{m}{s}[B^n(\eta')][\M^{n\times n}]\lesssim_{\Hmad{m}{s}} 1$.
It will then follow that constants which are $\HEad{m,s}$-admissible in the sense
of 
\cref{Defn::FuncSpaeRevComp::HEAdmiss}
(when applied to the vector fields $Y_1,\ldots, Y_n$) are $\Hmad{m}{s}$-admissible in the sense
of \cref{Defn::Results::HXAdmiss}.
 From here, \cref{Prop::FuncSpaceRev::CompNorms} implies \cref{EqnPfLI::EquivNorm::HolderNorm}.  Similar comments hold for Zygmund spaces;
 however, we are applying  \cref{Prop::FuncSpaceRev::CompNorms} with
 $\eta$ replaced by $\eta''$, and therefore $\ZygEad{s}$-admissible constants
 will also depend on an upper bound for $(\eta'')^{-1}$.  This is where the dependance
 on $\eta''$ enters in \cref{EqnPfLI::EquivNorm::ZygNorm}.

We first prove \cref{EqnPfLI::EquivNorm::HolderNorm,EqnPfLI::EquivNorm::HolderVF}.
We claim (for any function $f$),
\begin{equation}\label{EqnPfLI::EquivNorm::HolderNorm::Again}
    \HNorm{f}{m}{s}[B^n(\eta'')] \approx_{\Hmad{m-1}{s}} \HXNorm{f}{Y}{m}{s}[B^n(\eta'')],
\end{equation}
\begin{equation}\label{EqnPfLI::EquivNorm::HolderVF::Again}
    \HNorm{A}{m}{s}[B^n(\eta')][\M^{n\times n}]\lesssim_{\Hmad{m}{s}} 1,
\end{equation}
which are clearly equivalent to \cref{EqnPfLI::EquivNorm::HolderNorm,EqnPfLI::EquivNorm::HolderVF}.
We proceed by induction on $m$.  Using that $\CjNorm{A}{0}[B^n(\eta')][\M^{n\times n}]\leq \frac{1}{2}\lesssim_{\Hmad{-1}{s}} 1$,
the base case of \cref{EqnPfLI::EquivNorm::HolderNorm::Again} follows from
\cref{Prop::FuncSpaceRev::CompNorms}.  Using this and \cref{Prop::FuncSpaceRev::PushForwardNorm,Prop::PfLI::dPhiYj} we have
\begin{equation}\label{Eqn::PfLI::Basecijk::Holder}
    \HNorm{c_{i,j}^k\circ \Phi}{0}{s}[B^n(\eta')]\approx_{\Hmad{-1}{s}} \HXNorm{c_{i,j}^k\circ \Phi}{Y}{0}{s}[B^n(\eta')]
    \leq \HXNorm{c_{i,j}^k}{X}{0}{s}[B_X(x_0,\xi)]\lesssim_{\Hmad{0}{s}} 1.
\end{equation}
In light of \cref{Eqn::PfLI::Basecijk::Holder}, \cref{Prop:PfLI::MainODEProp} implies
$\HNorm{A}{0}{s}[B^n(\eta')][\M^{n\times n}]\lesssim_{\Hmad{0}{s}} 1$, completing the proof
of the base case $m=0$.

We assume \cref{EqnPfLI::EquivNorm::HolderNorm::Again,EqnPfLI::EquivNorm::HolderVF::Again}
for $m-1$ and prove them for $m$.  Because $\HNorm{A}{m-1}{s}[B^n(\eta')][\M^{n\times n}]\lesssim_{\Hmad{m-1}{s}} 1$, \cref{Prop::FuncSpaceRev::CompNorms} implies \cref{EqnPfLI::EquivNorm::HolderNorm::Again} for $m$.
Thus we need to show \cref{EqnPfLI::EquivNorm::HolderVF::Again}.

Using \cref{EqnPfLI::EquivNorm::HolderNorm::Again} and \cref{Prop::FuncSpaceRev::PushForwardNorm,Prop::PfLI::dPhiYj} we have
\begin{equation}\label{Eqn::PfLI::Inductcijk::Holder}
    \HNorm{c_{i,j}^k\circ \Phi}{m}{s}[B^n(\eta')]\approx_{\Hmad{m-1}{s}} \HXNorm{c_{i,j}^k\circ \Phi}{Y}{m}{s}[B^n(\eta')]\leq \HXNorm{c_{i,j}^k}{X}{m}{s}[B_X(x_0,\xi)]\lesssim_{\Hmad{m}{s}} 1.
\end{equation}
In light of \cref{Eqn::PfLI::Inductcijk::Holder}, \cref{Prop:PfLI::MainODEProp} implies
$\HNorm{A}{m}{s}[B^n(\eta')][\M^{n\times n}]\lesssim_{\Hmad{m}{s}} 1$, completing the proof
of \cref{EqnPfLI::EquivNorm::HolderVF::Again}, and therefore completing the proof of
\cref{EqnPfLI::EquivNorm::HolderNorm,EqnPfLI::EquivNorm::HolderVF}.

We turn to proving \cref{EqnPfLI::EquivNorm::ZygNorm,EqnPfLI::EquivNorm::ZygVF}.
We prove (for any function $f$)
\begin{equation}\label{EqnPfLI::EquivNorm::ZygNorm::Again}
    \ZygNorm{f}{s}[B^n(\eta'')]\approx_{\Zygmad{s-1},\eta''} \ZygXNorm{f}{Y}{s}[B^n(\eta'')],
\end{equation}
\begin{equation}\label{EqnPfLI::EquivNorm::ZygVF::Again}
    \ZygNorm{A}{s}[B^n(\eta')][\M^{n\times n}] \lesssim_{\Zygmad{s}} 1,
\end{equation}
which are clearly equivalent to \cref{EqnPfLI::EquivNorm::ZygNorm,EqnPfLI::EquivNorm::ZygVF}.

We first prove \cref{EqnPfLI::EquivNorm::ZygNorm::Again,EqnPfLI::EquivNorm::ZygVF::Again}
for $s\in (0,1]$.
\cref{EqnPfLI::EquivNorm::HolderVF::Again} shows
\begin{equation*}
    \HNorm{A}{0}{s/2}[B^n(\eta')][\M^{n\times n}]\lesssim_{\Hmad{-1}{s/2}} 1,
\end{equation*}
and therefore
\begin{equation*}
    \HNorm{A}{0}{s/2}[B^n(\eta')][\M^{n\times n}]\lesssim_{\Zygmad{s-1}} 1.
\end{equation*}
Using this, \cref{Prop::FuncSpaceRev::CompNorms} implies \cref{EqnPfLI::EquivNorm::ZygNorm::Again}.
In particular, since $\eta'$ is a $\Zygmad{-1}$-admissible constant (since it is a $0$-admissible constant),
and using \cref{EqnPfLI::EquivNorm::ZygNorm::Again} and \cref{Prop::FuncSpaceRev::PushForwardNorm,Prop::PfLI::dPhiYj},
\begin{equation}\label{Eqn::PfLI::Basecijk::Zyg}
    \ZygNorm{c_{i,j}^k\circ \Phi}{s}[B^n(\eta')]\approx_{\Zygmad{s-1}} \ZygXNorm{c_{i,j}^k\circ \Phi}{Y}{s}[B^n(\eta')] \leq \ZygXNorm{c_{i,j}^k}{X}{s}[B_X(x_0,\xi)]\lesssim_{\Zygmad{s}} 1.
\end{equation}
In light of \cref{Eqn::PfLI::Basecijk::Zyg}, \cref{Prop:PfLI::MainODEProp} implies \cref{EqnPfLI::EquivNorm::ZygVF::Again}.

We now assume \cref{EqnPfLI::EquivNorm::ZygNorm::Again,EqnPfLI::EquivNorm::ZygVF::Again}
for $s\in (0,k]$ and prove them for $s\in (k,k+1]$.  Fix $s\in (k,k+1]$.
By the inductive hypothesis, we know $\ZygNorm{A}{s-1}[B^n(\eta')][\M^{n\times n}]\lesssim_{\Zygmad{s-1}} 1$.
Using this, \cref{Prop::FuncSpaceRev::CompNorms} implies \cref{EqnPfLI::EquivNorm::ZygNorm::Again}
for $s$.  In particular, since $\eta'$ is a $\Zygmad{-1}$-admissible constant (since it is a $0$-admissible constant),
and using \cref{EqnPfLI::EquivNorm::ZygNorm::Again} and  \cref{Prop::FuncSpaceRev::PushForwardNorm,Prop::PfLI::dPhiYj},
\begin{equation}\label{Eqn::PfLI::Inductcijk::Zyg}
    \ZygNorm{c_{i,j}^k\circ \Phi}{s}[B^n(\eta')]\approx_{\Zygmad{s-1}} \ZygXNorm{c_{i,j}^k\circ \Phi}{Y}{s}[B^n(\eta')] \leq \ZygXNorm{c_{i,j}^k}{X}{s}[B_X(x_0,\xi)]\lesssim_{\Zygmad{s}} 1.
\end{equation}
In light of \cref{Eqn::PfLI::Inductcijk::Zyg}, \cref{Prop:PfLI::MainODEProp} implies \cref{EqnPfLI::EquivNorm::ZygVF::Again}.  

Finally, \cref{EqnPfLi::EquivNorm::Pullbackcijk::Holder} was
established in \cref{Eqn::PfLI::Basecijk::Holder,Eqn::PfLI::Inductcijk::Holder}
while \cref{EqnPfLi::EquivNorm::Pullbackcijk::Zyg} was established in
\cref{Eqn::PfLI::Basecijk::Zyg,Eqn::PfLI::Inductcijk::Zyg}.
\end{proof}

\begin{prop}
There exists a $1$-admissible constant $\eta_1\in (0,\eta']$ such that $\Phi\big|_{B^n(\eta_1)}$
is injective.  Furthermore, $\Phi(B^n(\eta_1))\subseteq B_X(x_0,\xi)$ is open and
$\Phi:B^n(\eta_1)\rightarrow \Phi(B^n(\eta_1))$ is a $C^2$-diffeomorphism.
\end{prop}
\begin{proof}
Consider the maps, defined for $u,v\in \R^n$ sufficiently small, given by
\begin{equation*}
    \Psi_{u}(v)=e^{v_1 Y_1+\cdots +v_n Y_n}u.
\end{equation*}
Since $Y=(I+A)\grad$ and $\Norm{A(t)}[\M^{n\times n}]\leq \frac{1}{2}$, $\forall t\in B^n(\eta')$,
we have $|\det(Y_1(t)|\cdots|Y_n(t))|\geq c_n>0$, $\forall t\in B^n(\eta')$, where $c_n>0$
can be chosen to depend only on $n$.
Furthermore, by \cref{Prop::PfLI::EquivNorms} (taking $m=1$, $s=0$ in \cref{EqnPfLI::EquivNorm::HolderVF}),
we have
\begin{equation}\label{Eqn::PfLI::Injective::YjC1}
\CjNorm{Y_j}{1}[B^n(\eta')][\R^n]\lesssim_{\Hmad{1}{0}} 1.
\end{equation}
Thus, by the definition of $1$-admissible constants, we have $\CjNorm{Y_j}{1}[B^n(\eta')][\R^n]\lesssim_1 1$.

Take $\Delta_0,\kappa>0$ as in \cref{Prop::AppIFT::MainProp} (with $\eta'$ playing the role of $\eta$
in that proposition).  In light of the above remarks, $\Delta_0$ and $\kappa$
can be taken to be $1$-admissible constants.
Set $\delta_1:=\min\{\Delta_0, \delta_0,1\}$ so that $\delta_1>0$ is a $1$-admissible constant;
see \cref{Section::Series::Quant} for the definition of $\delta_0$.  Let $\eta_1:=\min\{\delta_1\kappa,\eta'\}>0$
so that $\eta_1$ is a $1$-admissible constant.

We claim $\Phi\big|_{B^n(\eta_1)}$ is injective.  Let $u_1,u_2\in B^n(\eta_1)$ be such that
$\Phi(u_1)=\Phi(u_2)$; we wish to show $u_1=u_2$.
By \cref{Prop::AppIFT::MainProp} there exists $v\in B^n(\delta_1)$ such that
$u_2=\Psi_{u_1}(v)$, i.e., $u_2=e^{v\cdot Y} u_1$.  Since $d\Phi(u)Y_j(u)=X_j(\Phi(u))$ (\cref{Prop::PfLI::dPhiYj}), it follows that
\begin{equation*}
    \Phi(u_1)=\Phi(u_2)=\Phi(e^{v\cdot Y} u_1) = e^{v\cdot X} \Phi(u_1).
\end{equation*}
Also, we know $X_1(\Phi(u)),\ldots, X_n(\Phi(u))$ are linearly independent (as a consequence
of \cref{Prop::ResQual::Mmanif}).  Finally, $X$ satisfies $\sC(\Phi(u_1),\delta_1, B_X(x_0,\xi))$
because $Y$ satisfies $\sC(u_1,\delta_1,B^n(\eta'))$ (by \cref{Prop::AppIFT::MainProp}).
Hence, by the definition of $\delta_0$, we have $v=0$.  We conclude
$u_2=e^{v\cdot Y} u_1 = u_1$, and therefore $\Phi$ is injective.

Combining the fact that $d\Phi(u)Y_j(u)=X_j(\Phi(u))$ and $X_1,\ldots, X_n$ span the tangent
space at every point of $B_X(x_0,\xi)$, the Inverse Function Theorem implies
$\Phi:B^n(\eta')\rightarrow B_X(x_0,\xi)$ is an open map and is locally
a $C^1$ diffeomorphism.  In particular,
$\Phi(B^n(\eta_1))$ is open.
Hence, since $\Phi$ is injective, locally a $C^1$ diffeomorphism, and $\Phi$ is $C^2$ (\cref{Lemma::PfLI::PhiC2}), we conclude $\Phi:B^n(\eta_1)\rightarrow \Phi(B^n(\eta_1))$ is a $C^2$-diffeomorphism.
\end{proof}

\begin{lemma}
There exists a $1$-admissible constant $\xi_1>0$ such that $B_X(x_0,\xi_1)\subseteq \Phi(B^n(\eta_1))$.
\end{lemma}
\begin{proof}
Fix $\xi_1\in (0,\xi]$ to be chosen later, and suppose $y\in B_X(x_0,\xi_1)$.  Thus,
there exists $\gamma:[0,1]\rightarrow B_X(x_0,\xi)$ with $\gamma(0)=x_0$,
$\gamma(1)=y$, $\gamma'(t)=\sum_{j=1}^n b_j(t) \xi_1 X_j(\gamma(t))$,
$\Norm{\sum |b_j(t)|^2}[L^\infty([0,1])]<1$.
Define
$$t_0:=\sup \{t\in [0,1] : \gamma(t')\in \Phi(B^n(\eta_1/2)), \forall 0\leq t'\leq t\}.$$
We want to show that by taking $\xi_1>0$ to be a sufficiently small $1$-admissible
constant, we have $t_0=1$ and $\gamma(1)\in \Phi(B^n(\eta_1/2))$.
Note that $t_0\geq 0$, since $\gamma(0)=x_0=\Phi(0)$.

Suppose not.  Then $|\Phi^{-1}(\gamma(t_0))|=\frac{\eta_1}{2}$.  And, using that
$\CNorm{Y_j}{B^n(\eta_1);\R^n}\lesssim_0 1$ and $\Phi(0)=x_0$,
\begin{equation*}
    \eta_1/2 = |\Phi^{-1}(\gamma(t_0))| = \left|\int_0^{t_0} \frac{d}{dt} \Phi^{-1}\circ \gamma (t)\: dt\right|
    =\left|\int_0^{t_0} \sum_{j=1}^n b_j(t) \xi_1 Y_j(\Phi^{-1}\circ\gamma(t))\: dt\right|\lesssim_0 \xi_1.
\end{equation*}
This is a contradiction if $\xi_1$ is a sufficiently small $1$-admissible constant,
completing the proof.
\end{proof}

\begin{lemma}
$[Y_i,Y_j]=\sum_{k=1}^n \ct_{i,j}^k Y_k$ on $B^n(\eta_1)$, where
$\HNorm{\ct_{i,j}^k}{m}{s}[B^n(\eta_1)]\lesssim_{\Hmad{m}{s}} 1$
and $\ZygNorm{\ct_{i,j}^k}{s}[B^n(\eta_1)]\lesssim_{\Zygmad{s}} 1$.
\end{lemma}
\begin{proof}
Because $\Phi:B^n(\eta_1)\rightarrow \Phi(B^n(\eta_1))$ is a diffeomorphism,
we have
\begin{equation*}
    [Y_i,Y_j] = [\Phi^{*} X_i, \Phi^{*} X_j] = \Phi^{*} [X_i,X_j] =\Phi^{*}\sum_{k} c_{i,j}^k X_k =\sum_{k} \ct_{i,j}^k Y_k,
\end{equation*}
with $\ct_{i,j}^k=c_{i,j}^k\circ \Phi$.
From here the result follows from \cref{Prop::PfLI::EquivNorms}, since $\eta_1\leq \eta'$.
\end{proof}

\begin{proof}[Proof of \cref{Thm::Results::MainThm} when $n=q$]
As mentioned above, we take $\chi:=\xi$.  We also take $\xi_2:=\xi_1$.
Note that \cref{Item::Results::bklReg} is vacuuous when $n=q$.
Also, since $n=q$, $X=X_{J_0}$ and $Y=Y_{J_0}$.
With these remarks, all of the parts of \cref{Thm::Results::MainThm} except for \cref{Item::Results::AbstractNorm}
were proved above.
We clarify one point in \cref{Item::Results::EquivNorms}.
In \cref{Prop::PfLI::EquivNorms}, \cref{Item::Results::EquivNorms} was proved
on $B^n(\eta'')$ for any $\eta''\in (0,\eta']$.  Here, we are taking $\eta''=\eta_1$.
However, in the case of Zygmund spaces
the implicit constant in \cref{EqnPfLI::EquivNorm::ZygNorm}
also depended on the choice of $\eta''$.  Since $\eta_1$ is a $1$-admissible constant,
if $s>2$, it is a $\Zygmad{s-1}$-admissible constant.
This is why \cref{Item::Results::EquivNorms} is only stated for $s>2$ in the case of Zygmund spaces--in
the case $s\leq 1$, the implicit constants also depend on $\eta_1$, and are therefore $1$-admissible
constants.\footnote{It is classical that $\HSpace{0}{s}[B^n(\eta_1)]$ and
$\ZygSpace{s}[B^n(\eta_1)]$ have comparable norms for $s\in (0,1)$.  However, the constants
involved in the comparability of these norms depend on $\eta_1$, and are therefore
$1$-admissible.}

We close the proof by proving \cref{Item::Results::AbstractNorm}.  We prove the result
for Zygmund spaces, the same proof works for H\"older spaces.
Let $f\in \CSpace{B_{X_{J_0}}(x_0,\chi)}$.
We use \cref{Prop::PfLI::EquivNorms} in the case $\eta''=\eta'$, and that $\eta'$ is a
$\Zygmad{-1}$-admissible constant.  We also use \cref{Prop::FuncSpaceRev::PushForwardNorm}.
We have, for $s\in (0,\infty)$,
\begin{equation*}
    \ZygNorm{f\circ \Phi}{s}[B^n(\eta_1)]\leq \ZygNorm{f\circ\Phi}{s}[B^n(\eta')]
    \approx_{\Zygmad{s-1}} \ZygXNorm{f\circ\Phi}{Y}{s}[B^n(\eta')]
    \leq \ZygXNorm{f}{X}{s}[B_X(x_0,\chi)],
\end{equation*}
completing the proof.
\end{proof}

In the third paper of this series, it will be be convenient to use a slight modification of \cref{Thm::Results::MainThm} in the case $n=q$, where
we replace $1$-admissible constants with a slightly different definition.  We present this here.

\begin{defn}
In the case $n=q$, if we say $C$ is a $1'$-admissible constant, it means that we assume $c_{j,k}^l\circ \Phi\in \CjSpace{1}[B^n(\eta_0)]$, for $1\leq j,k,l\leq n$.
$C$ is then allowed to depend only on upper bounds for $n$, $\xi^{-1}$, $\eta^{-1}$, $\delta_0^{-1}$, and $\CjNorm{c_{j,k}^l\circ \Phi}{1}[B^n(\eta_0)]$
and $\CNorm{c_{j,k}^l}{B_{X_{J_0}}(x_0,\xi)}$
 ($1\leq j,k,l\leq n$).
\end{defn}

\begin{prop}
In the case $n=q$, \cref{Thm::Results::MainThm} (except for \cref{Item::Results::EquivNorms}) holds with the following modifications.
The assumption $c_{j,k}^l\in \CXjSpace{X_{J_0}}{1}[B_{X_{J_0}}(x_0,\xi)]$ is replaced by $c_{j,k}^l\circ \Phi\in \CjSpace{1}[B^n(\eta_0)]$
and $1$-admissible constants are replaced with $1'$-admissible constants throughout.
\end{prop}
\begin{proof}[Comments on the proof]
The only place the estimates on $\CXjNorm{c_{j,k}^l}{X_{J_0}}{1}[B_{X_{J_0}}(x_0,\xi)]$ from $1$-admissible constants arose in the proof
was to conclude 
$\CjNorm{Y_j}{1}[B^n(\eta')][\R^n]\lesssim_{1} 1$; i.e.,
to conclude $\CjNorm{A}{1}[B^n(\eta')][\M^{n\times n}]\lesssim_1 1$.
However, one obtains $\CjNorm{A}{1}[B^n(\eta')][\M^{n\times n}]\lesssim_{1'} 1$ directly from \cref{Prop:PfLI::MainODEProp}.
Using this, the proof goes through unchanged.
\end{proof}

%% file: pfld.tex

In this section, we prove \cref{Thm::Results::MainThm} in the general case $q\geq n$.
Thus, we take the same setting and notation as in \cref{Thm::Results::MainThm}.

\begin{lemma}\label{Lemma::PfLD::LieWedge}
For $J\in \sI(n,q)$, $1\leq j\leq n$,
\begin{equation*}
    \Lie{X_j} \bigwedge X_J = \sum_{K\in \sI_0(n,q)} g_{j,J}^K \bigwedge X_K, \text{ on }B_{X_{J_0}}(x_0,\xi),
\end{equation*}
where
\begin{equation*}
    \CNorm{g_{j,J}^K}{B_{X_{J_0}}(x_0,\xi)}\lesssim_0 1,
\end{equation*}
for $m\in \N$ and $s\in [0,1]$,
\begin{equation*}
    \HXNorm{g_{j,J}^K}{X}{m}{s}[B_{X_{J_0}}(x_0,\xi)]\lesssim_{\Had{m,m,s}} 1,
\end{equation*}
and for $s\in (0,\infty)$,
\begin{equation*}
    \ZygXNorm{g_{j,J}^K}{X}{s}[B_{X_{J_0}}(x_0,\xi)]\lesssim_{\Zygad{s,s}} 1.
\end{equation*}
\end{lemma}
\begin{proof}
Let $J=(j_1,\ldots, j_n)$.  We have,
\begin{equation*}
\begin{split}
    \Lie{X_j} \bigwedge X_J &=\Lie{X_j} \left(X_{j_1} \wedge X_{j_2} \wedge \cdots \wedge X_{j_n}\right)
    =\sum_{l=1}^n X_{j_1}\wedge X_{j_2}\wedge \cdots\wedge X_{j_{l-1}} \wedge [X_j,X_{j_l}] \wedge X_{j_{l+1}}\wedge \cdots \wedge X_{j_n}
    \\&=\sum_{l=1}^n \sum_{k=1}^q c_{j,j_l}^k X_{j_1}\wedge X_{j_2}\wedge \cdots \wedge X_{j_{l-1}}\wedge X_k\wedge X_{j_{l+1}}\wedge \cdots \wedge X_{j_n}.
\end{split}
\end{equation*}
The result follows from the anti-commutativity of $\wedge$ and the assumptions on $c_{i,j}^k$.
\end{proof}

\begin{lemma}\label{Lemma::PfLD::LieDerivOfQuotient}
Let $\chi'\in (0,\xi]$.  Suppose for all $y\in B_{X_{J_0}}(x_0,\chi')$,
$\bigwedge X_{J_0}(y)\ne 0$.  Then, for $J\in \sI(n,q)$, $1\leq j\leq n$,
\begin{equation*}
    X_j \frac{\bigwedge X_{J}}{\bigwedge X_{J_0}} = \sum_{K\in \sI_0(n,q)} g_{j,J}^K\frac{\bigwedge X_K}{\bigwedge X_{J_0}} -\sum_{K\in \sI_0(n,q)} g_{j,J_0}^K \frac{\bigwedge X_J}{\bigwedge X_{J_0}} \frac{\bigwedge X_K}{\bigwedge X_{J_0}} \text{ on }B_{X_{J_0}}(x_0,\chi'),
\end{equation*}
where $g_{j,J}^K$ are the functions from \cref{Lemma::PfLD::LieWedge}.
\end{lemma}
\begin{proof}
This follows by combining \cref{Lemma::PfLD::LieWedge,Lemma::Wedge::DerivFrac}.
\end{proof}

\begin{lemma}\label{Lemma::PfLD::ODEIneq}
Let $C>0$ and $u_0>0$.  Let $u_{u_0,C}(t)$ be the unique solution to
\begin{equation*}
    \frac{d}{dt} u_{u_0,C}(t) = C(u_{u_0,C}(t)+u_{u_0,C}(t)^2), \quad u_{u_0,C}(0)=u_0,
\end{equation*}
defined on some maximum interval $[0, R_{u_0,C})$.
Let $F(t)$ be a non-negative function defined on $[0,R')$ with $R'\leq R_{u_0,C}$
satisfying
\begin{equation*}
    \frac{d}{dt} F(t)\leq C( F(t)+F(t)^2),\quad F(0)\leq u_0.
\end{equation*}
   Then, for $t\in [0,R')$, $F(t)\leq u_{u_0,C}(t)$.
\end{lemma}
\begin{proof}
This is standard and is easy to see directly.  It is also a special case of the Bihari-LaSalle inequality.
\end{proof}

\begin{lemma}\label{Lemma::PfLD::BasicBootstrapchi::V2}
There exists a $0$-admissible constant $\chi\in (0,\xi]$ such that the following holds.
Suppose $\gamma:[0,\chi]\rightarrow B_{X_{J_0}}(x_0,\xi)$ satisfies $\gamma(0)=x_0$,
$\gamma'(t)=\sum_{j=1}^n a_j(t) X_j(\gamma(t))$, and $\Norm{\sum|a_j(t)|^2}[L^\infty([0,\chi])]<1$.
Suppose further that for some $\chi'\in (0,\chi]$, $\bigwedge X_{J_0}(\gamma(t))\ne 0$ for $t\in (0,\chi']$.
Then,
\begin{equation}\label{Eqn::PfLD::J0BiggestAgain}
    \sup_{\substack{J\in \sI(n,q)\\ t\in [0,\chi']}}\left|\frac{\bigwedge X_J(\gamma(t))}{\bigwedge X_{J_0}(\gamma(t))}\right|\lesssim_0 1.
\end{equation}
Here, the implicit constant depends on neither $\chi'$ nor $\gamma$.
\end{lemma}
\begin{proof}
Let $\chi\in (0,\xi]$ be a $0$-admissible constant to be chosen later.
Let $\gamma$ and $\chi'$ be as in the statement of the lemma.
We wish to show that if $\chi$ is chosen to be a sufficiently small $0$-admissible constant (which
 forces $\chi'$ to be small),
then \cref{Eqn::PfLD::J0BiggestAgain} holds.

Set
\begin{equation*}
    F(t):=\sum_{J\in \sI_0(n,q)}\left|\frac{X_J(\gamma(t))}{X_{J_0}(\gamma(t))}\right|^2.
\end{equation*}
We wish to show that if $\chi$ is a sufficiently small $0$-admissible constant,
then $F(t)\lesssim_0 1$, $\forall t\in [0,\chi']$, and this will complete the proof.\footnote{Here we are using
$\forall K\in \sI(n,q)$, either $\bigwedge X_K\equiv 0$ or $\exists J\in \sI_0(n,q)$ with $\bigwedge X_K = \pm \bigwedge X_J$,
by the basic properties of wedge products.}

Using \cref{Lemma::PfLD::LieDerivOfQuotient}, we have,
\begin{equation*}
\begin{split}
    &\frac{d}{dt} F(t) =\sum_{J\in \sI_0(n,q)} 2 \frac{\bigwedge X_J(\gamma(t))}{\bigwedge X_{J_0}(\gamma(t))} \sum_{j=1}^n a_j(t) \left(X_j \frac{\bigwedge X_J}{\bigwedge X_{J_0}}\right)(\gamma(t))
    \\&=\sum_{J\in \sI_0(n,q)} \sum_{K\in \sI_0(n,q)} \sum_{j=1}^n 2 a_j(t) \frac{\bigwedge X_J(\gamma(t))}{\bigwedge X_{J_0}(\gamma(t))} \left( g_{j,J}^K(\gamma(t)) \frac{\bigwedge X_K(\gamma(t)) }{\bigwedge X_{J_0}(\gamma(t))} - g_{j,J_0}^K(\gamma(t)) \frac{\bigwedge X_J(\gamma(t))}{\bigwedge X_{J_0}(\gamma(t))} \frac{\bigwedge X_K(\gamma(t))}{\bigwedge X_{J_0}(\gamma(t))}  \right)
    \\&\lesssim_0  F(t)+F(t)^{3/2}\lesssim_0 F(t)+F(t)^2.
\end{split}
\end{equation*}
Also, we have
\begin{equation*}
    F(0)= \sum_{J\in \sI_0(n,q)}\left|\frac{X_J(x_0)}{X_{J_0}(x_0)}\right|^2 \lesssim_0 1.
\end{equation*}
Thus, there exist  $0$-admissible constants $C$ and $u_0>0$ such that
\begin{equation*}
    \frac{d}{dt} F(t) \leq C \left(F(t)+F(t)^2\right),\quad F(0)\leq u_0.
\end{equation*}
Standard theorems from ODEs show that if $\chi=\chi(C,u_0)>0$ is chosen sufficiently small, then the unique solution
$u(t)$ to
\begin{equation*}
    \frac{d}{dt} u(t) =  C\left(u(t)+u(t)^2\right), \quad u(0)=u_0,
\end{equation*}
exists for $t\in [0,\chi]$ and satisfies $u(t)\leq 2u_0$, $\forall t\in [0,\chi]$.
For this choice of $\chi$ (which is $0$-admissible, since $C$ and $u_0$ are),
\cref{Lemma::PfLD::ODEIneq} shows $F(t)\leq 2u_0\lesssim_0 1$, $\forall t\in [0,\chi']$, completing the proof.
\end{proof}

\begin{prop}\label{Prop::PfLD::Makechi}
There exists a $0$-admissible constant $\chi\in (0,\xi]$ such that
$\forall y\in B_{X_{J_0}}(x_0,\chi)$, $\bigwedge X_{J_0}(y)\ne 0$
and
\begin{equation}\label{Eqn::PfLd::ToShow::QuotConst}
    \sup_{\substack{J\in \sI(n,q) \\ y\in B_{X_{J_0}}(x_0,\chi)}}\left|\frac{\bigwedge X_J(y)}{\bigwedge X_{J_0}(y)}\right|\lesssim_0 1.
\end{equation}
\end{prop}
\begin{proof}
Take $\chi$ as in \cref{Lemma::PfLD::BasicBootstrapchi::V2}.
First we claim $\forall y\in B_{X_{J_0}}(x_0,\chi)$, $\bigwedge X_{J_0}(y)\ne 0$.
Fix $y\in B_{X_{J_0}}(x_0,\chi)$, so that there exists $\gamma:[0,\chi]\rightarrow B_{X_{J_0}}(x_0,\xi)$,
$\gamma(0)=x_0$, $\gamma(\chi)=y$, $\gamma'(t)=\sum_{j=1}^n a_j(t) X_j(\gamma(y))$,
$\Norm{\sum |a_j(t)|^2}[L^\infty([0,1])]<1$.
We will show that $\forall t\in [0,\chi]$, $\bigwedge X_{J_0}(\gamma(t))\ne 0$, and then
it will follow that $\bigwedge X_{J_0}(y)=\bigwedge X_{J_0}(\gamma(\chi))\ne 0$.

Suppose not, so that $\bigwedge X_{J_0}(\gamma(t))= 0$ for some $t\in [0,\chi]$.
Let $t_0=\inf \{ t\in [0,\chi] : \bigwedge X_{J_0}(\gamma(t))=0\}$, so that
$\bigwedge X_{J_0}(\gamma(t_0))=0$ but $\bigwedge X_{J_0}(\gamma(t))\ne 0$, $\forall t\in [0,t_0)$.
Note that $t_0>0$ since $\bigwedge X_{J_0}(x_0)\ne 0$.

Let $\nu$ be a $C^1$ $n$-form, defined on a neighborhood of $\gamma(t_0)$ and which is nonzero at $\gamma(t_0)$.
We have
$$\lim_{t\uparrow t_0} \nu\left(\bigwedge X_{J_0}\right)(\gamma(t))=0,\quad
\lim_{t\uparrow t_0} \max_{J\in \sI(n,q)} \left|\nu(X_J)(\gamma(t))\right|>0,
$$
by continuity, the fact that $X_1,\ldots, X_q$ span the tangent space at $\gamma(t_0)$,
and that $\nu$ is nonzero at $\gamma(t_0)$.
We conclude,
\begin{equation}\label{Eqn::PfLD::BlowupContracidction}
    \lim_{t\uparrow t_0} \sup_{J\in \sI(n,q)}\left|\frac{\bigwedge X_J(\gamma(t))}{\bigwedge X_{J_0}(\gamma(t))}\right|=\lim_{t\uparrow t_0} \sup_{J\in \sI(n,q)}\left|\frac{\nu\left(\bigwedge X_J\right)(\gamma(t))}{\nu\left(\bigwedge X_{J_0}\right)(\gamma(t))}\right|=\infty.
\end{equation}

Take any $\chi'\in (0,t_0)$.  We know $\forall t\in [0,\chi']$, $\bigwedge X_{J_0}(\gamma(t))\ne 0$.
\cref{Lemma::PfLD::BasicBootstrapchi::V2} implies
\begin{equation*}
    \sup_{\substack{J\in \sI(n,q)\\ t\in [0,\chi']}}\left|\frac{\bigwedge X_J(\gamma(t))}{\bigwedge X_{J_0}(\gamma(t))}\right|\lesssim_0 1.
\end{equation*}
Since $\chi'\in (0,t_0)$ was arbitrary, we have
\begin{equation*}
    \sup_{\substack{J\in \sI(n,q)\\ t\in [0,t_0)}}\left|\frac{\bigwedge X_J(\gamma(t))}{\bigwedge X_{J_0}(\gamma(t))}\right|\lesssim_0 1.
\end{equation*}
This contradicts \cref{Eqn::PfLD::BlowupContracidction} and completes the proof that
$\bigwedge X_{J_0}(y)\ne 0$, $\forall y\in B_{X_{J_0}}(x_0,\chi)$.

To prove \cref{Eqn::PfLd::ToShow::QuotConst} take $y\in B_{X_{J_0}}(x_0,\chi)$.  Then, there exists $\gamma:[0,\chi]\rightarrow B_{X_{J_0}}(x_0,\xi)$,
$\gamma(0)=x_0$, $\gamma(\chi)=y$, $\gamma'(t)=\sum_{j=1}^n a_j(t) X_j(\gamma(y))$,
$\Norm{\sum |a_j(t)|^2}[L^\infty([0,\chi])]<1$.   We have already shown $\bigwedge X_{J_0}(\gamma(t))\ne 0$,
$\forall t\in (0,\chi]$.  \Cref{Lemma::PfLD::BasicBootstrapchi::V2} implies
$\sup_{\substack{J\in \sI(n,q) }}\left|\frac{\bigwedge X_J(y)}{\bigwedge X_{J_0}(y)}\right|=\sup_{\substack{J\in \sI(n,q) }}\left|\frac{\bigwedge X_J(\gamma(\chi))}{\bigwedge X_{J_0}(\gamma(\chi))}\right| \lesssim_0 1$.
Since $y\in B_{X_{J_0}}(x_0,\chi)$ was arbitrary, \cref{Eqn::PfLd::ToShow::QuotConst} follows.

\end{proof}

For the remainder of the section, fix $\chi\in (0,\xi]$ as in \cref{Prop::PfLD::Makechi}.

\begin{lemma}\label{Lemma::PfLD::NormsOfWedges}
For $m\in \N$, $s\in [0,1]$, $J\in \sI(n,q)$,
\begin{equation}\label{Eqn::PfLD::HolderNormWedge}
    \BHXNorm{\frac{\bigwedge X_J}{\bigwedge X_{J_0}}}{X_{J_0}}{m}{s}[B_{X_{J_0}}(x_0,\chi)]\lesssim_{\Had{m-1,m-1,s}} 1,
\end{equation}
and for $s\in (0,\infty)$,
\begin{equation}\label{Eqn::PfLD::ZygNormWedge}
    \BZygXNorm{\frac{\bigwedge X_J}{\bigwedge X_{J_0}}}{X_{J_0}}{s}[B_{X_{J_0}}(x_0,\chi)]\lesssim_{\Zygad{s-1,s-1}} 1.
\end{equation}
\end{lemma}
\begin{proof}
In this proof, we freely use the estimates on the functions $g_{j,J}^K$ as described
in \cref{Lemma::PfLD::LieWedge,Lemma::PfLD::LieDerivOfQuotient}.
We begin with \cref{Eqn::PfLD::HolderNormWedge}.
\Cref{Prop::PfLD::Makechi} shows
\begin{equation}\label{Eqn::PfLD::HolderNormWedge::C0}
    \BCNorm{\frac{\bigwedge X_J}{\bigwedge X_{J_0}}}{B_{X_{J_0}}(x_0,\chi)}\lesssim_0 1.
\end{equation}
We claim,
\begin{equation}\label{Eqn::PfLD::HolderNormWedge::C1}
\BCXjNorm{\frac{\bigwedge X_J}{\bigwedge X_{J_0}}}{X_{J_0}}{1}[B_{X_{J_0}}(x_0,\chi)]\lesssim_0 1.
\end{equation}
Indeed, for $1\leq j\leq n$, using \cref{Lemma::PfLD::LieDerivOfQuotient},
\begin{equation*}
\begin{split}
    &\BCNorm{X_j \frac{\bigwedge X_J}{\bigwedge X_{J_0}}}{B_{X_{J_0}}(x_0,\chi)}
    =\BCNorm{\sum_{K\in \sI_0(n,q)} g_{j,J}^K\frac{\bigwedge X_K}{\bigwedge X_{J_0}} -\sum_{K\in \sI_0(n,q)} g_{j,J_0}^K \frac{\bigwedge X_J}{\bigwedge X_{J_0}} \frac{\bigwedge X_K}{\bigwedge X_{J_0}}}{B_{X_{J_0}}(x_0,\chi)}
    \\&\lesssim_0 \sum_{K\in \sI_0(n,q)} \BCNorm{\frac{\bigwedge X_K}{\bigwedge X_{J_0}}}{B_{X_{J_0}}(x_0,\chi)} + \BCNorm{ \frac{\bigwedge X_J}{\bigwedge X_{J_0}}}{B_{X_{J_0}}(x_0,\chi)} \BCNorm{\frac{\bigwedge X_K}{\bigwedge X_{J_0}}}{B_{X_{J_0}}(x_0,\chi)}\lesssim_0 1,
    \end{split}
\end{equation*}
where the last inequality follows from \cref{Eqn::PfLD::HolderNormWedge::C0}.  \Cref{Eqn::PfLD::HolderNormWedge::C1} follows.

Using \cref{Lemma::FuncSpaces::Inclu} \cref{Item::FuncSpace::IncludHold} and \cref{Item::FuncSpace::IncludLip},
we have for $s\in [0,1]$,
\begin{equation*}
    \BHXNorm{\frac{\bigwedge X_J}{\bigwedge X_{J_0}}}{X_{J_0}}{0}{s}[B_{X_{J_0}}(x_0,\chi)]\leq
    3\BHXNorm{\frac{\bigwedge X_J}{\bigwedge X_{J_0}}}{X_{J_0}}{0}{1}[B_{X_{J_0}}(x_0,\chi)]
    \leq 3\BCXjNorm{\frac{\bigwedge X_J}{\bigwedge X_{J_0}}}{X_{J_0}}{1}[B_{X_{J_0}}(x_0,\chi)]\lesssim_0 1,
\end{equation*}
where the last inequality used \cref{Eqn::PfLD::HolderNormWedge::C1}.
This proves \cref{Eqn::PfLD::HolderNormWedge} in the case $m=0$.

We prove \cref{Eqn::PfLD::HolderNormWedge} by induction on $m$, the base case ($m=0$) having just been
proved.  We assume \cref{Eqn::PfLD::HolderNormWedge} for $m-1$ and prove it for $m$.
We use \cref{Prop::FuncSpaceRev::Algebra} freely in what follows.
We have
\begin{equation*}
    \BHXNorm{\frac{\bigwedge X_J}{\bigwedge X_{J_0}}}{X_{J_0}}{m}{s}[B_{X_{J_0}}(x_0,\chi)]
    = \BHXNorm{\frac{\bigwedge X_J}{\bigwedge X_{J_0}}}{X_{J_0}}{m-1}{s}[B_{X_{J_0}}(x_0,\chi)]
    +\sum_{j=1}^n\BHXNorm{X_j \frac{\bigwedge X_J}{\bigwedge X_{J_0}}}{X_{J_0}}{m-1}{s}[B_{X_{J_0}}(x_0,\chi)].
\end{equation*}
The first term is $\lesssim_{\Had{m-2,m-2,s}} 1$, by the inductive hypothesis, so we focus only
on the second term.  We have, using \cref{Lemma::PfLD::LieDerivOfQuotient}, and
letting $C_m$ be a constant which depends only on $m$,
\begin{equation*}
    \begin{split}
        &\BHXNorm{X_j \frac{\bigwedge X_J}{\bigwedge X_{J_0}}}{X_{J_0}}{m-1}{s}[B_{X_{J_0}}(x_0,\chi)]
        \leq C_m \sum_{K\in \sI_0(n,q)} \BHXNorm{g_{j,J}^K}{X_{J_0}}{m-1}{s}[B_{X_{J_0}}(x_0,\xi)]
        \BHXNorm{\frac{\bigwedge X_K}{\bigwedge X_{J_0}}}{X_{J_0}}{m-1}{s}[B_{X_{J_0}}(x_0,\chi)]
        \\&\quad \quad +C_m\sum_{K\in \sI_0(n,q)}
        \BHXNorm{g_{j,J_0}^K}{X_{J_0}}{m-1}{s}[B_{X_{J_0}}(x_0,\xi)]
        \BHXNorm{\frac{\bigwedge X_J}{\bigwedge X_{J_0}}}{X_{J_0}}{m-1}{s}[B_{X_{J_0}}(x_0,\chi)]
        \BHXNorm{\frac{\bigwedge X_K}{\bigwedge X_{J_0}}}{X_{J_0}}{m-1}{s}[B_{X_{J_0}}(x_0,\chi)]
        \\&\lesssim_{\Had{m-1,m-1,s}} 1,
    \end{split}
\end{equation*}
where the last inequality follows from the bounds described in \cref{Lemma::PfLD::LieWedge}
and the inductive hypothesis.
This completes the proof of \cref{Eqn::PfLD::HolderNormWedge}.

We turn to \cref{Eqn::PfLD::ZygNormWedge}, and proceed by induction on $m$, where $s\in (m,m+1]$.
We begin with the base case, $m=0$, so that $s\in (0,1]$.
Using \cref{Lemma::FuncSpaces::Inclu} \cref{Item::FuncSpace::IncludZygInHold}, we have
\begin{equation*}
    \BZygXNorm{\frac{\bigwedge X_J}{\bigwedge X_{J_0}}}{X_{J_0}}{s}[B_{X_{J_0}}(x_0,\chi)] \leq 5 \BHXNorm{\frac{\bigwedge X_J}{\bigwedge X_{J_0}}}{X_{J_0}}{0}{s}[B_{X_{J_0}}(x_0,\chi)]\lesssim_{0} 1,
\end{equation*}
where the last inequality follows from \cref{Eqn::PfLD::HolderNormWedge}.
This implies \cref{Eqn::PfLD::ZygNormWedge} for the base case $s\in (0,1]$.
From here, the inductive step follows just as in \cref{Eqn::PfLD::HolderNormWedge}
and we leave the remaining details to the reader.
\end{proof}

\begin{lemma}\label{Lemma::PfLD::Boundbt}
For $1\leq k\leq q$, $1\leq l\leq n$, there exists $\bt_k^l\in \CSpace{B_{X_{J_0}}(x_0,\chi)}$
such that
\begin{equation}\label{Eqn::PfLD::btFormula}
    X_k =\sum_{l=1}^n \bt_k^l X_l,
\end{equation}
where for $m\in \N$, $s\in [0,1]$,
\begin{equation*}
    \BHXNorm{\bt_k^l}{X_{J_0}}{m}{s}[B_{X_{J_0}}(x_0,\chi)]\lesssim_{\Had{m-1,m-1,s}} 1,
\end{equation*}
and for $s\in (0,\infty)$,
\begin{equation*}
    \BZygXNorm{\bt_k^l}{X_{J_0}}{s}[B_{X_{J_0}}(x_0,\chi)]\lesssim_{\Zygad{s-1,s-1}} 1.
\end{equation*}
\end{lemma}
\begin{proof}
For $1\leq k\leq n$ this is trivial (merely take $\bt_k^l=1$ if $k=l$ and $\bt_k^l=0$ if $k\ne l$),
however the proof that follows deals with all $1\leq k\leq q$ simultaneously.

For $1\leq k\leq q$, $1\leq l\leq n$, let $J(l,k)=(1,2,\ldots, l-1,k,l+1,\ldots, n)\in \sI(n,q)$.
We have, by Cramer's rule \cref{Eqn::Wedge::Cramer},
\begin{equation*}
    X_k = \sum_{l=1}^n \frac{\bigwedge X_{J(l,k)} }{\bigwedge X_{J_0}} X_l.
\end{equation*}
From here, the result follows from \cref{Lemma::PfLD::NormsOfWedges}.
\end{proof}

\begin{lemma}\label{Lemma::PfLD::LIApplies}
For $1\leq i,j,l\leq n$, $\exists \ch_{i,j}^l\in \CSpace{B_{X_{J_0}}(x_0,\chi)}$ such that
$[X_i,X_j]=\sum_{l=1}^n \ch_{i,j}^l X_l$, where for $m\in \N$, $s\in [0,1]$,
\begin{equation*}
    \HXNorm{\ch_{i,j}^l}{X_{J_0}}{m}{s}[B_{X_{J_0}}(x_0,\chi)] \lesssim_{\Had{m,m-1,s}} 1,
\end{equation*}
and for $s\in (0,\infty)$,
\begin{equation*}
    \ZygXNorm{\ch_{i,j}^l}{X_{J_0}}{s}[B_{X_{J_0}}(x_0,\chi)] \lesssim_{\Zygad{s,s-1}} 1.
\end{equation*}
\end{lemma}
\begin{proof}
For $1\leq i,j\leq n$ and using \cref{Lemma::PfLD::Boundbt}, we have
\begin{equation*}
    [X_i,X_j] = \sum_{k=1}^q c_{i,j}^k X_k =\sum_{l=1}^n \left(\sum_{k=1}^q  c_{i,j}^k \bt_k^l\right) X_l.
\end{equation*}
Setting $\ch_{i,j}^l = \sum_{k=1}^q  c_{i,j}^k \bt_k^l$, the result follows from
the definition of admissible constants, \cref{Lemma::PfLD::Boundbt}, and \cref{Prop::FuncSpaceRev::Algebra}.
\end{proof}

\Cref{Lemma::PfLD::LIApplies} shows that the case $n=q$ of \cref{Thm::Results::MainThm}
(which was proved in \cref{Section::Proofs::LI}) applies to $X_1,\ldots, X_n$,
with $\xi$ replaced by $\chi$.\footnote{When we proved  \cref{Thm::Results::MainThm}  for $n=q$, in \cref{Section::Proofs::LI},
we took $\chi=\xi$.}
In light of \cref{Lemma::PfLD::LIApplies} any constants
which are $\Had{m,m-1,s}$, $\Zygad{s,s-1}$, $0$, or $1$-admissible in the sense of this
application of the case $n=q$ of \cref{Thm::Results::MainThm},
are $\Had{m,m-1,s}$, $\Zygad{s,s-1}$, $0$, or $1$-admissible (respectively) in the sense of this section.
%
Thus, from the case $n=q$, we obtain $1$-admissible constants $\xi_1,\eta_1>0$
and a map $\Phi:B^n(\eta_1)\rightarrow B_{X_{J_0}}(x_0,\chi)$ as in \cref{Thm::Results::MainThm}.
Most of the case $q\geq n$ of \cref{Thm::Results::MainThm} immediately follows from this application
of the case $n=q$.
All that remain to show are:
\cref{Item::Results::BigWedge}, \cref{Item::Results::Open}, there exists $\xi_2$ as in \cref{Item::Results::xi2},
\cref{Item::Results::YReg} for $n+1\leq j\leq q$, \cref{Item::Results::bklReg}, and \cref{Item::Results::EquivNorms}.

\begin{proof}[Proof of \cref{Item::Results::BigWedge}]
That $\forall y\in B_{X_{J_0}}(x_0,\chi)$
$$\sup_{J\in \sI(n,q)} \left|\frac{\bigwedge X_J(y)}{\bigwedge X_{J_0}(y)}\right|\geq 1$$
is clear (by taking $J=J_0$).
That
$$\sup_{J\in \sI(n,q)} \left|\frac{\bigwedge X_J(y)}{\bigwedge X_{J_0}(y)}\right|\lesssim_0 1,$$
$\forall y\in B_{X_{J_0}}(x_0,\chi)$,
is \cref{Prop::PfLD::Makechi}.
\end{proof}

\begin{proof}[Proof of \cref{Item::Results::Open}]
Let $\chi'\in (0,\chi]$ and fix $x\in B_{X_{J_0}}(x_0,\chi)$.  \Cref{Item::Results::WedgeNonzero} shows
$X_1(x),\ldots, X_n(x)$ are linearly independent.  Define
$\Psi(t):= e^{t_1 X_1+\cdots +t_nX_n}x$, so that
$d\Psi(0) = (X_1(x)|\cdots|X_n(x))$ and is therefore invertible.
It is clear that for $\delta$ sufficienty small
$\Psi(B^n(\delta))\subseteq B_{X_{J_0}}(x_0,\chi')$ and the Inverse Function
Theorem shows that for $\delta$ sufficiently small $\Psi(B^n(\delta))\subseteq B_X(x_0,\xi)$ is
open.  Hence, $\Psi(B^n(\delta))$ is an open neighborhood of $x$ in $B_{X_{J_0}}(x_0,\chi')$.
Since $x\in B_{X_{J_0}}(x_0,\chi')$ was arbitrary, we conclude
$B_{X_{J_0}}(x_0,\chi')\subseteq B_X(x_0,\xi)$ is open.
\end{proof}

That there exists a $1$-admissible constant $\xi_2>0$ such that
\cref{Item::Results::xi2} holds follows by applying the next lemma with $\zeta_1=\xi_1$.
\begin{lemma}
Fix $\zeta_1\in (0,\chi]$.  Then, there is a $0$-admissible constant $\zeta_2>0$ (which also depends on $\zeta_1$)
such that $B_X(x_0,\zeta_2)\subseteq B_{X_{J_0}}(x_0,\zeta_1)$.
\end{lemma}
\begin{proof}
Let $\zeta_2\in (0,\zeta_1]$, we will pick $\zeta_2$ at the end of the proof.
Suppose $y\in B_{X}(x_0,\zeta_2)$, so that $\exists \gamma:[0,1]\rightarrow B_X(x_0,\zeta_2)$
with $\gamma(0)=x_0$, $\gamma(1)=y$, $\gamma'(t)=\sum_{j=1}^q a_j(t)\zeta_2 X_j(\gamma(t))$,
$\Norm{\sum|a_j(t)|^2}[L^\infty([0,1])]<1$.
Let
\begin{equation*}
    t_0=\sup\{t\in [0,1] : \gamma(t')\in B_{X_{J_0}}(x_0,\zeta_1/2), \forall t'\in [0,t]\}.
\end{equation*}
We wish to show that if $\zeta_2=\zeta_2(\zeta_1)>0$ is taken to be a sufficiently small $0$-admissible
constant, then we have $t_0=1$ and $y=\gamma(1)\in B_{X_{J_0}}(x_0,\zeta_1)$.

In fact, we will prove $\gamma(t_0)\in B_{X_{J_0}}(x_0,\zeta_1/2)$.  The result will then follow
as if $t_0<1$, the fact that $B_{X_{J_0}}(x_0,\zeta_1/2)$ is open (see \cref{Item::Results::Open})
and $\gamma$ is continuous show that $\gamma(t')\in B_{X_{J_0}}(x_0,\zeta_1/2)$ for $t'\in [0,t_0+\epsilon)$
for some $\epsilon>0$, which contradicts the choice of $t_0$.

We turn to proving $\gamma(t_0)\in B_{X_{J_0}}(x_0,\zeta_1/2)$.  We have
\begin{equation*}
    \gamma'(t) = \sum_{k=1}^q a_k(t) \zeta_2 X_k(\gamma(t)) = \sum_{l=1}^n \left(\sum_{k=1}^q a_k(t) \zeta_2 \bt_k^l(\gamma(t))\right) X_l(\gamma(t))=:\sum_{l=1}^n \at_l(t)\frac{\zeta_1}{2} X_l(\gamma(t)),
\end{equation*}
where $\BNorm{\sum |\at_l(t)|^2}[L^\infty([0,t_0])]\lesssim_0 \frac{\zeta_2}{\zeta_1}$ (see \cref{Lemma::PfLD::Boundbt}).
Thus, by taking $\zeta_2=\zeta_2(\zeta_1)>0$ to be a sufficiently small $0$-admissible constant,
we have $\BNorm{\sum |\at_l(t)|^2}[L^\infty([0,t_0])]<1$.  It follows that $\gamma(t_0)\in B_{X_{J_0}}(x_0,\zeta_1/2)$,
which completes the proof.
\end{proof}

\begin{proof}[Proof of \cref{Item::Results::bklReg}]
For $n+1\leq k\leq q$, $1\leq l \leq n$, set $b_k^l:=\bt_k^l\circ \Phi$.
Pulling back \cref{Eqn::PfLD::btFormula} via $\Phi$ shows $Y_k=\sum_{l=1}^n b_k^l Y_l$.
The regularity of $b_k^l$ now follows by combining \cref{Item::Results::AbstractNorm} and
the bounds in \cref{Lemma::PfLD::Boundbt}.
\end{proof}

\begin{proof}[Proof of \cref{Item::Results::YReg} for $n+1\leq j\leq q$]
This follows by combining \cref{Item::Results::YReg} for $1\leq j\leq n$ and \cref{Item::Results::bklReg}.
\end{proof}

\begin{proof}[Proof of \cref{Item::Results::EquivNorms}]
We prove the result for Zygmund spaces; the proof for H\"older spaces is similar, and we leave it to the
reader.
Let $s>2$.
The case $n=q$ of \cref{Thm::Results::MainThm} gives
$\ZygNorm{f}{s}[B^n(\eta_1)]\approx_{\Zygad{s-1,s-2}} \ZygXNorm{f}{Y_{J_0}}{s}[B^n(\eta_1)]$.
Also, $\ZygXNorm{f}{Y}{s}[B^n(\eta_1)]\approx_{\Zygad{s-1,s-2}} \ZygNorm{f}{s}[B^n(\eta_1)]$
follows from \cref{Prop::FuncSpaceRev::CompNorms}, \cref{Item::Results::ASize}, \cref{Item::Results::YReg},
and the fact that $\eta_1$ is a $\Zygad{s-1,s-2}$-admissible constant, for $s>2$.
Here we are using $\grad=(I+A)^{-1} Y_{J_0}$ and $\ZygNorm{(I+A)^{-1}}{s}[B^n(\eta_1)][\M^{n\times n}]\lesssim_{\Zygad{s,s-1}} 1$,
for $s>0$
(which follows from \cref{Item::Results::ASize} and \cref{Item::Results::YReg}).
\end{proof}

%% file: pfdensities.tex

In this section, we prove the results from \cref{Section::Densities}.
We recall the density $\nu_0$ from \cref{Eqn::Density::Defnnu0}, defined
on $B_{X_{J_0}}(x_0,\chi)$:
\begin{equation*}
    \nu_0(x)(Z_1(x),\ldots, Z_n(x))  := \left| \frac{Z_1(x)\wedge Z_2(x)\wedge \cdots \wedge Z_n(x)}{X_1(x)\wedge X_2(x)\wedge \cdots \wedge X_n(x) }\right|.
\end{equation*}

\begin{lemma}\label{Lemma::PfDense::Evalnu0}
$\nu_0(X_1,\ldots, X_n)\equiv 1$, and for $j_1,\ldots, j_n\in \{1,\ldots, q\}$,
$\nu_0(X_{j_1},\ldots, X_{j_n})\lesssim_0 1$.
\end{lemma}
\begin{proof}
That $\nu_0(X_1,\ldots, X_n)\equiv 1$ follows directly from the definition.
That $\nu_0(X_{j_1},\ldots, X_{j_n})\lesssim_0 1$ follows from \cref{Thm::Results::MainThm} \cref{Item::Results::BigWedge}.
\end{proof}

\begin{lemma}\label{Lemma::PfDense::DivedeWedgeMatrix}
Let $V$ and $W$ be $n$-dimensional real vector spaces, and let $A:W\rightarrow V$ be an invertible linear
transformation.  Let $v_1,\ldots, v_n$ be a basis for $V$ and let $w_1,\ldots, w_n\in W$.  Then,
\begin{equation*}
    \frac{Aw_1\wedge Aw_2\wedge \cdots \wedge Aw_n}{v_1\wedge v_2\wedge \cdots \wedge v_n} = \frac{w_1\wedge w_2\wedge \cdots \wedge w_n}{A^{-1}v_1\wedge A^{-1}v_2\wedge \cdots \wedge A^{-1}v_n}.
\end{equation*}
\end{lemma}
\begin{proof}
Let $Z_1,Z_2$ be one dimensional real vector spaces and let $B:Z_1\rightarrow Z_2$ be an invertible linear transformation.
Let $z_1\in Z_1$ and $0\ne z_2\in Z_2$.  We claim
\begin{equation}\label{Eqn::PfDense::ToShowOneDim}
\frac{Bz_1}{z_2}=\frac{z_1}{B^{-1} z_2}.
\end{equation}
Indeed, let $\lambda_2:Z_2\rightarrow \R$ be any nonzero linear functional, and set $\lambda_1:=\lambda_2\circ B:Z_1\rightarrow \R$
so that $\lambda_1$ is also a nonzero linear functional.  We have
\begin{equation*}
    \frac{B z_1}{z_2} = \frac{\lambda_2(Bz_1)}{\lambda_2(z_2)} =\frac{\lambda_1(z_1)}{\lambda_1(B^{-1} z_2)} = \frac{z_1}{B^{-1} z_2}.
\end{equation*}
Applying \cref{Eqn::PfDense::ToShowOneDim} in the case
$Z_1=\bigwedge^n W$, $Z_2=\bigwedge^n V$, and $B:Z_1\rightarrow Z_2$ given by
$B (w_1\wedge w_2\wedge \cdots \wedge w_n) = (Aw_1)\wedge (A w_2)\wedge \cdots \wedge (Aw_n)$
completes the proof.


\end{proof}

\begin{lemma}
For $1\leq j\leq n$, $\Lie{X_j} \nu_0 = f_j^0 \nu_0$, where $f_j^0\in \CSpace{B_{X_{J_0}}(x_0,\chi)}$.
Furthermore, for $m\in \N$, $s\in [0,1]$,
\begin{equation}\label{Eqn::PfDense::Holderfj0}
    \HXNorm{f_j^0}{X_{J_0}}{m}{s}[B_{X_{J_0}}(x_0,\chi)]\lesssim_{\Had{m,m,s}} 1,
\end{equation}
and for $s\in (0,\infty)$,
\begin{equation}\label{Eqn::PfDense::Zygfj0}
    \ZygXNorm{f_j^0}{X_{J_0}}{s}[B_{X_{J_0}}(x_0,\chi)]\lesssim_{\Zygad{s,s}} 1.
\end{equation}
\end{lemma}
\begin{proof}
Set $\phi_t(x)=e^{tX_j}x$ so that $\Lie{X_j} \nu_0 = \diff{t}\big|_{t=0} \phi_t^{*}\nu_0$.
We write $d\phi_t(x)$ to denote the differential of $\phi_t$ in the $x$ variable.
We have, using \cref{Lemma::PfDense::DivedeWedgeMatrix},
\begin{equation}\label{Eqn::PfDense::PullBacknu0}
\begin{split}
    &(\phi_t^{*} \nu_0)(x)(Z_1,\ldots, Z_n) = \nu_0(\phi_t(x))(d\phi_t(x) Z_1(x),\ldots, d\phi_t(x) Z_n(x))
    \\&=\left|\frac{d\phi_t(x) Z_1(x) \wedge d\phi_t(x) Z_2(x)\wedge \cdots \wedge d\phi_t(x) Z_n(x)}{X_1(\phi_t(x))\wedge X_2(\phi_t(x))\wedge \cdots \wedge X_n(\phi_t(x))}\right|
    \\&=\left|\frac{Z_1(x)\wedge Z_2(x)\wedge \cdots \wedge Z_n(x)}{ d\phi_t(x)^{-1} X_1(\phi_t(x))\wedge d\phi_t(x)^{-1} X_2(\phi_t(x))\wedge \cdots \wedge d\phi_t(x)^{-1} X_n(\phi_t(x)) }\right|
    \\&=\left|\frac{Z_1(x)\wedge Z_2(x)\wedge \cdots \wedge Z_n(x)}{\phi_t^{*} X_1(x)\wedge \phi_t^{*} X_2(x)\wedge \cdots \wedge\phi_t^{*} X_n(x) }\right|
\end{split}
\end{equation}

Fix $x\in B_{X_{J_0}}(x_0,\chi)$.
We claim that the sign of $$\frac{Z_1(x)\wedge Z_2(x)\wedge \cdots \wedge Z_n(x)}{\phi_t^{*} X_1(x)\wedge \phi_t^{*} X_2(x)\wedge \cdots\wedge \phi_t^{*} X_n(x) }$$ does not change for $t$ small.
To this end, let $\theta$ be a $C^1$ $n$-form which is nonzero near $x$.
Since
$X_1(x)\wedge X_2(x)\wedge\cdots \wedge X_n(x)\ne 0$ (\cref{Thm::Results::MainThm} \cref{Item::Results::WedgeNonzero}),
$\theta(x)(X_1(x)\wedge X_2(x)\wedge \cdots \wedge X_n(x))\ne 0$, and so by continuity, for $t$ small,
$\theta(x)(\phi_t^{*} X_1(x)\wedge \phi_t^{*} X_2(x)\wedge \cdots \wedge \phi_t^{*} X_n(x))\ne 0$.
We conclude that for $t$ sufficiently small,
\begin{equation*}
    \frac{Z_1(x)\wedge Z_2(x)\wedge \cdots \wedge Z_n(x)}{\phi_t^{*} X_1(x)\wedge \phi_t^{*} X_2(x)\wedge \cdots \wedge\phi_t^{*} X_n(x) } = \frac{\theta(x)(Z_1(x)\wedge Z_2(x)\wedge \cdots \wedge Z_n(x))}{\theta(x)(\phi_t^{*} X_1(x)\wedge \phi_t^{*} X_2(x)\wedge \cdots \wedge \phi_t^{*} X_n(x)) }
\end{equation*}
does not change sign, and is either never zero or always zero for small $t$.

Set, for $t$ small,
 $$\epsilon:=\mathrm{sgn}\frac{Z_1(x)\wedge Z_2(x)\wedge \cdots \wedge Z_n(x)}{\phi_t^{*}X_1(x)\wedge \phi_t^{*} X_2(x)\wedge \cdots \wedge \phi_t^{*} X_n(x) },$$
and in the case the quantity inside $\mathrm{sgn}$ equals zero, the choice of $\epsilon$ does not matter.
By the above discussion, $\epsilon$ does not depend on $t$ (for $t$ small).
We have, using the functions $g_{j,J}^K$ from \cref{Lemma::PfLD::LieWedge,Lemma::PfLD::LieDerivOfQuotient},
\begin{equation*}
    \begin{split}
        &\diff{t}\bigg|_{t=0} (\phi_t^{*} \nu_0)(x)(Z_1(x),\ldots, Z_n(x))
        =\diff{t}\bigg|_{t=0}\left|\frac{Z_1(x)\wedge Z_2(x)\wedge \cdots \wedge Z_n(x)}{\phi_t^{*} X_1(x)\wedge \phi_t^{*} X_2(x)\wedge \cdots \wedge\phi_t^{*} X_n(x) }\right|
        \\&\diff{t}\bigg|_{t=0}\epsilon\frac{Z_1(x)\wedge Z_2(x)\wedge \cdots \wedge Z_n(x)}{\phi_t^{*} X_1(x)\wedge \phi_t^{*} X_2(x)\wedge \cdots\wedge \phi_t^{*} X_n(x) }
        =\diff{t}\bigg|_{t=0}\epsilon\frac{\theta(x)(Z_1(x)\wedge Z_2(x)\wedge \cdots \wedge Z_n(x))}{\theta(x)(\phi_t^{*} X_1(x)\wedge \phi_t^{*} X_2(x)\wedge \cdots \wedge\phi_t^{*} X_n(x) )}
        \\& = -\epsilon\frac{\theta(x)(Z_1(x)\wedge Z_2(x)\wedge \cdots \wedge Z_n(x))}{\theta(x)(\phi_t^{*} X_1(x)\wedge \phi_t^{*} X_2(x)\wedge \cdots\wedge \phi_t^{*} X_n(x) )^2} \diff{t} \theta(x)(\phi_t^{*} X_1(x)\wedge \phi_t^{*} X_2(x)\wedge \cdots \wedge\phi_t^{*} X_n(x) )\bigg|_{t=0}
        \\&=-\epsilon\frac{\theta(x)(Z_1(x)\wedge Z_2(x)\wedge \cdots \wedge Z_n(x))}{\theta(x)( X_1(x)\wedge  X_2(x)\wedge \cdots \wedge X_n(x) )}\frac{\theta(x)(\Lie{X_j}( X_1\wedge  X_2\wedge \cdots \wedge X_n )(x)) }{\theta(x)( X_1(x)\wedge  X_2(x)\wedge \cdots \wedge X_n(x) )}
        \\&=-\epsilon \frac{Z_1(x)\wedge Z_2(x)\wedge \cdots \wedge Z_n(x)}{X_1(x)\wedge  X_2(x)\wedge \cdots \wedge X_n(x)} \frac{\sL_{X_j}( X_1\wedge  X_2\wedge \cdots \wedge X_n)(x)}{X_1(x)\wedge  X_2(x)\wedge \cdots \wedge X_n(x)}
        \\&=-\frac{\sL_{X_j} (X_1\wedge  X_2\wedge \cdots \wedge X_n)(x)}{X_1(x)\wedge  X_2(x)\wedge \cdots \wedge X_n(x)}\nu_0(x)(Z_1(x),\ldots, Z_n(x))
        \\&=-\sum_{K\in \sI_0(n,q)} g_{j,J_0}^K(x) \frac{\bigwedge X_K(x)}{\bigwedge X_{J_0}(x)} \nu_0(x)(Z_1(x),\ldots, Z_n(x)).
    \end{split}
\end{equation*}
We conclude that
\begin{equation*}
    f_j^0 =  -\sum_{K\in \sI_0(n,q)} g_{j,J_0}^K \frac{\bigwedge X_K}{\bigwedge X_{J_0}}.
\end{equation*}
\Cref{Eqn::PfDense::Holderfj0,Eqn::PfDense::Zygfj0} follow from \cref{Lemma::PfLD::LieWedge,Lemma::PfLD::NormsOfWedges,Prop::FuncSpaceRev::Algebra}.
\end{proof}

Let $\sigma_0:=\Phi^{*}\nu_0$, so that $\sigma_0$ is a density on $B^n(\eta_1)$.
Define $h_0$ by $\sigma_0=h_0\LebDensity$, so that $h_0\in \CSpace{B^n(\eta_1)}$.

\begin{lemma}\label{Lemma::Density::Estimateh0}
$h_0(t)=\det(I+A(t))^{-1}$, where $A$ is the matrix from \cref{Thm::Results::MainThm}.
In particular, $h_0(t)\approx_0 1$, $\forall t\in B^n(\eta_1)$.  For $m\in \N$, $s\in [0,1]$,
\begin{equation}\label{Eqn::PfDense::h0Holder}
    \HNorm{h_0}{m}{s}[B^n(\eta_1)]\lesssim_{\Had{m,m-1,s}} 1,
\end{equation}
and for $s\in (0,\infty)$,
\begin{equation}\label{Eqn::PfDense::h0Zyg}
    \ZygNorm{h_0}{s}[B^n(\eta_1)]\lesssim_{\Zygad{s,s-1}} 1.
\end{equation}
\end{lemma}
\begin{proof}
Because $\Norm{A(t)}[\M^{n\times n}]\leq \frac{1}{2}$, $\forall t\in B^n(\eta_1)$ (\cref{Thm::Results::MainThm} \cref{Item::Results::ASize}), we have $|\det(I+A(t))^{-1}|=\det(I+A(t))^{-1}$, $\forall t\in B^n(\eta_1)$.
We have,
\begin{equation*}
\begin{split}
    &h_0(t) = \sigma_0(t)\left(\diff{t_1},\diff{t_2},\ldots, \diff{t_n}\right) = \sigma_0(t)((I+A(t))^{-1}Y_1(t),\ldots, (I+A(t))^{-1}Y_n(t))
    \\&=|\det (I+A(t))^{-1}| \sigma_0(t)(Y_1(t),\ldots, Y_n(t))
    =\det (I+A(t))^{-1} \nu_0(\Phi(t))\left(X_1(\Phi(t)),\ldots, X_n(\Phi(t))\right)
    \\&= \det (I+A(t))^{-1}.
\end{split}
\end{equation*}
That $h_0(t)\approx_0 1$, $\forall t\in B^n(\eta_1)$ follows from the fact that
$\Norm{A(t)}[\M^{n\times n}]\leq \frac{1}{2}$, $\forall t\in B^n(\eta_1)$ (\cref{Thm::Results::MainThm} \cref{Item::Results::ASize}).
Using \cref{Prop::FuncSpaceRev::Algebra} (applied to the cofactor representation of $(I+A)^{-1}$), \cref{Eqn::PfDense::h0Holder,Eqn::PfDense::h0Zyg}  follow from the corresponding regularity for $A$
as described in \cref{Thm::Results::MainThm} \cref{Item::Results::YReg}--here we are using
that the regularity for $A$ and
the regularity for $Y_1,\ldots, Y_n$ are the same, by the definition of $A$.
\end{proof}

We now turn to studying the density $\nu$ from \cref{Section::Densities}; thus we use the
functions $f_j$ from \cref{Eqn::Desnity::Defnfj}.
Because $\nu_0$ is a nonzero density on $B_{X_{J_0}}(x_0,\chi)$, there is a unique $g\in \CSpace{B_{X_{J_0}}(x_0,\chi)}$ such that $\nu=g\nu_0$.
\begin{lemma}\label{Lemma::PfDesnity::Derivg}
For $1\leq j\leq n$, $X_j g = (f_j-f_j^{0}) g$.
\end{lemma}
\begin{proof}
We have,
\begin{equation*}
    f_j g \nu_0 = f_j \nu = \Lie{X_j} \nu = \Lie{X_j} (g\nu_0) = (X_j g) \nu_0 + g\Lie{X_j} \nu_0 = (X_j g)\nu_0 + g f_j^0\nu_0.
\end{equation*}
The result follows.
\end{proof}

\begin{lemma}\label{Lemma::Density::gconst}
\Cref{Thm::Density::MainThm} \cref{Item::Desnity::gconst} holds.  Namely,
$g(x)\approx_{0;\nu} g(x_0)=\nu(x_0)(X_1(x_0),\ldots, X_n(x_0))$, $\forall x\in B_{X_{J_0}}(x_0,\chi)$.
\end{lemma}
\begin{proof}
Note $g(x_0)=g(x_0)\nu_0(x_0)(X_1(x_0),\ldots, X_n(x_0)) =  \nu(x_0)(X_1(x_0),\ldots, X_n(x_0))$, by definition.
So it suffices to show $g(x)\approx_{0;\nu} g(x_0)$ for $x\in B_{X_{J_0}}(x_0,\chi)$.

Let $\gamma:[0,1]\rightarrow B_{X_{J_0}}(x_0,\chi)$ be such that
$\gamma(0)=x_0$, $\gamma(1)=x$, $\gamma'(t)=\sum_{j=1}^n a_j(t) \chi X_j(\gamma(t))$, $\Norm{\sum |a_j(t)|^2}[L^\infty([0,1])]<1$.  We have, using \cref{Lemma::PfDesnity::Derivg},
\begin{equation*}
    \frac{d}{dt} g(\gamma(t)) = \sum_{j=1}^n a_j(t) \chi (X_j g)(\gamma(t))
    =\sum_{j=1}^n a_j(t) \chi (f_j(\gamma(t))-f_j^0(\gamma(t)))g(\gamma(t)).
\end{equation*}
Hence, $g(\gamma(t))$ satisfies an ODE.  Solving this ODE we have
\begin{equation*}
    g(x)=g(\gamma(1))=e^{ \int_0^1 \sum_{j=1}^n a_j(s) \chi (f_j(\gamma(s))-f_j^0(\gamma(s)))\: ds  }g(x_0).
\end{equation*}
We know $\CNorm{f_j^0}{B_{X_{J_0}}(x_0,\chi)}\lesssim_0 1$ (by the case $m=0$, $s=0$ of \cref{Eqn::PfDense::Holderfj0}).  Using this and the definition of $0;\nu$-admissible constants,
$g(x)\approx_{0;\nu} g(x_0)$ follows immediately, completing the proof.
\end{proof}

\begin{lemma}
\Cref{Thm::Density::MainThm} \cref{Item::Desnity::greg} holds.  Namely,
for $m\in \N$, $s\in [0,1]$,
\begin{equation}\label{Eqn::Density::ToShow::Holderg}
\HXNorm{g}{X_{J_0}}{m}{s}[B_{X_{J_0}}(x_0,\chi)]\lesssim_{\Had{m-1,m-1,s;\nu}} |\nu(X_1,\ldots, X_n)(x_0)|,
\end{equation}
and for $s\in (0,\infty)$,
\begin{equation}\label{Eqn::Density::ToShow::Zygg}
\ZygXNorm{g}{X_{J_0}}{s}[B_{X_{J_0}}(x_0,\chi)]\lesssim_{\Zygad{s-1,s-1;\nu}} |\nu(X_1,\ldots, X_n)(x_0)|.
\end{equation}
\end{lemma}
\begin{proof}
We begin with \cref{Eqn::Density::ToShow::Holderg}.  First note that
\begin{equation}\label{Eqn::Desnity::gBoundedInC}
    \CNorm{g}{B_{X_{J_0}}(x_0,\chi)}\lesssim_{0;\nu} |\nu(X_1,\ldots, X_n)(x_0)|,
\end{equation}
which follows immediately from \cref{Lemma::Density::gconst}.
We claim that
\begin{equation}\label{Eqn::Desnity::C1g}
    \CXjNorm{g}{X_{J_0}}{1}[B_{X_{J_0}}(x_0,\chi)]\lesssim_{0;\nu} |\nu(X_1,\ldots, X_n)(x_0)|.
\end{equation}
Indeed, using \cref{Lemma::PfDesnity::Derivg}, for each $1\leq j\leq n$,
\begin{equation}\label{Eqn::Desnity::gBoundedXj}
    \CNorm{X_j g}{B_{X_{J_0}}(x_0,\chi)} = \CNorm{(f_j -f_j^0) g}{B_{X_{J_0}}(x_0,\chi)} \lesssim_{0;\nu} \CNorm{g}{B_{X_{J_0}}(x_0,\chi)} \lesssim_{0;\nu}|\nu(X_1,\ldots, X_n)(x_0)|,
\end{equation}
where in the last inequality we have used \cref{Eqn::Desnity::gBoundedInC}
and in the second to last inequality we have used $\CNorm{f_j}{B_{X_{J_0}}(x_0,\chi)}\lesssim_{0;\nu} 1$ (which
follows from the definition of $0;\nu$-admissible constants) and
$\CNorm{f_j^0}{B_{X_{J_0}}(x_0,\chi)}\lesssim_{0} 1$ (which follows from the case $m=0$, $s=0$
of \cref{Eqn::PfDense::Holderfj0}).
Combining \cref{Eqn::Desnity::gBoundedInC,Eqn::Desnity::gBoundedXj} proves \cref{Eqn::Desnity::C1g}.

We prove \cref{Eqn::Density::ToShow::Holderg} by induction on $m$.  For the base case, $m=0$, we have
using \cref{Lemma::FuncSpaces::Inclu} \cref{Item::FuncSpace::IncludHold} and \cref{Item::FuncSpace::IncludLip},
and \cref{Eqn::Desnity::C1g},
\begin{equation*}
    \HXNorm{g}{X_{J_0}}{0}{s}[B_{X_{J_0}}(x_0,\chi)] \leq 3 \HXNorm{g}{X_{J_0}}{0}{1}[B_{X_{J_0}}(x_0,\chi)]
    \leq 3\CXjNorm{g}{X_{J_0}}{1}[B_{X_{J_0}}(x_0,\chi)]\lesssim_{0;\nu} |\nu(X_1,\ldots,X_n)(x_0)|.
\end{equation*}
This proves the case $m=0$ of \cref{Eqn::Density::ToShow::Holderg}.

We now assume \cref{Eqn::Density::ToShow::Holderg} for $m-1$ and prove it for $m$.
We have
\begin{equation*}
    \HXNorm{g}{X_{J_0}}{m}{s}[B_{X_{J_0}}(x_0,\chi)] = \HXNorm{g}{X_{J_0}}{m-1}{s}[B_{X_{J_0}}(x_0,\chi)] + \sum_{j=1}^n\HXNorm{X_j g}{X_{J_0}}{m-1}{s}[B_{X_{J_0}}(x_0,\chi)].
\end{equation*}
The first term is $\lesssim_{\Had{m-2,m-2,s;\nu}} |\nu(X_1,\ldots, X_n)(x_0)|$ by the inductive hypothesis, so we focus only
on the second term.  We have, using \cref{Lemma::PfDesnity::Derivg,Prop::FuncSpaceRev::Algebra},
for a constant $C_m$ depending only on $m$, for $1\leq j\leq n$,
\begin{equation*}
\begin{split}
    &\HXNorm{X_j g}{X_{J_0}}{m-1}{s}[B_{X_{J_0}}(x_0,\chi)] = \HXNorm{(f_j-f_j^{0}) g}{X_{J_0}}{m-1}{s}[B_{X_{J_0}}(x_0,\chi)]
    \\&\leq C_m \HXNorm{f_j-f_j^{0}}{X_{J_0}}{m-1}{s}[B_{X_{J_0}}(x_0,\chi)] \HXNorm{ g}{X_{J_0}}{m-1}{s}[B_{X_{J_0}}(x_0,\chi)]
    \lesssim_{\Had{m-1,m-1,s;\nu}}|\nu(X_1,\ldots, X_n)(x_0)|,
\end{split}
\end{equation*}
where the last inequality follows from the inductive hypothesis, \cref{Eqn::PfDense::Holderfj0},
and the definition of $\Had{m-1,m-1,s;\nu}$-admissible constants.  \Cref{Eqn::Density::ToShow::Holderg}
follows.

We turn to \cref{Eqn::Density::ToShow::Zygg}, which we prove by induction on $m$,
where $s\in (m,m+1]$.  We begin with the base case, $m=0$, so that $s\in (0,1]$.
Using \cref{Lemma::FuncSpaces::Inclu} \cref{Item::FuncSpace::IncludZygInHold} and \cref{Eqn::Density::ToShow::Holderg}
we have
\begin{equation*}
    \ZygXNorm{g}{X_{J_0}}{s}[B_{X_{J_0}}(x_0,\chi)]\leq 5 \HXNorm{g}{X_{J_0}}{0}{s}[B_{X_{J_0}}(x_0,\chi)]\lesssim_{0;\nu} |\nu(X_1,\ldots, X_n)(x_0)|.
\end{equation*}
\Cref{Eqn::Density::ToShow::Zygg} follows for $s\in (0,1]$.
From here the inductive step follows just as in the inductive step for \cref{Eqn::Density::ToShow::Holderg},
and we leave the details to the reader.
\end{proof}

\begin{lemma}\label{Lemma::Density::Computeh}
Let $h(t)$ be as in \cref{Thm::Density::MainThm}.  Then
$h(t)=h_0(t)g\circ\Phi(t)$.
\end{lemma}
\begin{proof}
We have
\begin{equation*}
    \Phi^{*} \nu = \Phi^{*} g\nu_0 = (g\circ \Phi) \Phi^{*} \nu_0 = (g\circ \Phi) h_0 \LebDensity,
\end{equation*}
completing the proof.
\end{proof}

\begin{proof}[Proof of \cref{Thm::Density::MainThm} \cref{Item::Density::hconst}]
This follows from \cref{Lemma::Density::Computeh,Lemma::Density::gconst,Lemma::Density::Estimateh0}.
\end{proof}

\begin{proof}[Proof of \cref{Thm::Density::MainThm} \cref{Item::Density::hReg}]
We prove the result for Zygmund spaces; the same proof works for H\"older spaces, and we leave
the details to the reader.
Using \cref{Thm::Results::MainThm} \cref{Item::Results::AbstractNorm} we have
\begin{equation}\label{Eqn::Density::Zygg::Pullback}
    \ZygNorm{g\circ \Phi}{s}[B^{n}(\eta_1)]\lesssim_{\Zygad{s-1,s-2}} \ZygXNorm{g}{X_{J_0}}{s}[B_{X_{J_0}}(x_0,\chi)]
    \lesssim_{\Zygad{s-1,s-1;\nu}} |\nu(X_1,\ldots, X_n)(x_0)|,
\end{equation}
where the last inequality uses \cref{Eqn::Density::ToShow::Zygg}.
Since $h(t)=h_0(t)g\circ \Phi(t)$ (\cref{Lemma::Density::Computeh}),
combining \cref{Eqn::Density::Zygg::Pullback} and \cref{Eqn::PfDense::h0Zyg},
and using \cref{Prop::FuncSpaceRev::Algebra} completes the proof.
\end{proof}

Having completed the proof of \cref{Thm::Density::MainThm}, we turn to \cref{Cor::Desnity::MeasureSets}.
To facilitate this, we introduce a corollary of \cref{Thm::Results::MainThm}.

\begin{cor}\label{Cor::MainThm::MoreBalls}
Let $\eta_1,\xi_1,\xi_2$ be as in \cref{Thm::Results::MainThm}.  Then, there exist $1$-admissible constants
$0<\eta_2\leq \eta_1$, $0<\xi_4\leq \xi_3\leq \xi_2$ such that
\begin{equation*}
    \begin{split}
        &B_X(x_0,\xi_4)\subseteq B_{X_{J_0}}(x_0,\xi_3)\subseteq \Phi(B^n(\eta_2))\subseteq B_{X_{J_0}}(x_0,\xi_2)\subseteq B_X(x_0,\xi_2)
        \\&\subseteq B_{X_{J_0}}(x_0,\xi_1)\subseteq \Phi(B^n(\eta_1))\subseteq B_{X_{J_0}}(x_0,\chi)
        \subseteq B_{X}(x_0,\xi).
    \end{split}
\end{equation*}
\end{cor}
\begin{proof}
After obtaining $\eta_1$, $\xi_1$, $\xi_2$ from \cref{Thm::Results::MainThm}, apply
\cref{Thm::Results::MainThm} again with $\xi$ replaced by $\xi_2$ to obtain $\eta_2$, $\xi_3$,
and $\xi_4$ as in the statement of the corollary.
\end{proof}

\begin{proof}[Proof of \cref{Cor::Desnity::MeasureSets}]
We have
\begin{equation}\label{Eqn::Density::Cor::EstMeasure1}
\begin{split}
    &\nu(B_{X_{J_0}}(x_0,\xi_2)) = \int_{B_{X_{J_0}}(x_0,\xi_2)} \nu =\int_{\Phi^{-1}(B_{X_{J_0}}(x_0,\xi_2))}\Phi^{*}\nu
    \\&=\int_{\Phi^{-1}(B_{X_{J_0}}(x_0,\xi_2))} h(t)\: dt
    \approx_{0;\nu}\Vol{\Phi^{-1}(B_{X_{J_0}}(x_0,\xi_2))} \nu(X_1,\ldots, X_n)(x_0),
\end{split}
\end{equation}
where $\Vol{\cdot}$ denotes Lebesgue measure, and we have used \cref{Thm::Density::MainThm} \cref{Item::Density::hconst}.  By \cref{Cor::MainThm::MoreBalls}, and the fact that $\eta_1,\eta_2>0$
are $1$-admissible constants, we have
\begin{equation}\label{Eqn::Density::Cor::EstMeasure2}
    1\approx_{1} \Vol{B^n(\eta_2)} \leq \Vol{\Phi^{-1}(B_{X_{J_0}}(x_0,\xi_2))}\leq \Vol{B^n(\eta_1)}\approx_1 1.
\end{equation}
Combining \cref{Eqn::Density::Cor::EstMeasure1,Eqn::Density::Cor::EstMeasure2} proves
$\nu(B_{X_{J_0}}(x_0,\xi_2))\approx_{1;\nu} \nu(X_1,\ldots, X_n)(x_0)$.  The same proof works
with $B_{X_{J_0}}(x_0,\xi_2)$ replaced by $B_X(x_0,\xi_2)$, which completes the proof of \cref{Cor::Density::ToShowMeasure1}.

All that remains to prove \cref{Cor::Density::ToShowMeasure2} is to show
\begin{equation*}
    |\nu(X_1,\ldots, X_n)(x_0)|
 \approx_{0} \max_{(j_1,\ldots, j_n)\in \sI(n,q)} |\nu(X_{j_1},\ldots, X_{j_n})(x_0)|.
\end{equation*}
We have, using \cref{Lemma::PfDense::Evalnu0},
\begin{equation*}
\begin{split}
    &|\nu(X_1,\ldots, X_n)(x_0)| = |g(x_0) \nu_0(X_1,\ldots, X_n)(x_0)| = |g(x_0)| \approx_0 |g(x_0)|\max_{(j_1,\ldots, j_n)\in \sI(n,q)} |\nu_0(X_{j_1},\ldots, X_{j_n})(x_0)|
    \\&= \max_{(j_1,\ldots, j_n)\in \sI(n,q)}| g(x_0) \nu_0(X_{j_1},\ldots, X_{j_n})(x_0)| = \max_{(j_1,\ldots, j_n)\in \sI(n,q)} |\nu(X_{j_1},\ldots, X_{j_n})(x_0)|,
\end{split}
\end{equation*}
completing the proof.
\end{proof}

%% file: pfmoreassump.tex
In this section we prove \cref{Prop::MoreAssumpt}.  The existence of $\eta>0$ as in \cref{Prop::MoreAssumpt} follows immediately from the Picard--Lindel\"of Theorem, so we focus
on the existence of $\delta_0>0$.  The key is the next lemma.

\begin{lemma}\label{Lemma::PfMoreAssump::Classic}
Suppose $Z$ is a $C^1$ vector field on an open set $V\subseteq \R^n$.  Then, there exists $\delta>0$, depending only on $n$, such that if
$\CjNorm{Z}{1}[V][\R^n]\leq \delta$, then there does not exist $x\in V$ with:
\begin{itemize}
\item $e^{tZ}x\in V$, $\forall t\in [0,1]$.
\item $e^{Z} x=x$.
\item $Z(x)\ne 0$.
\end{itemize}
\end{lemma}
\begin{proof}For a proof of this classical result, see \cite[Lemma 3.19]{S}.\end{proof}

To prove the existence of $\delta_0$ as in \cref{Prop::MoreAssumpt}, since $K$ is compact, it suffices to prove the next lemma.

\begin{lemma}
Let $X_1,\ldots, X_q$ be $C^1$ vector fields on a $C^2$ manifold $\fM$.
For all $x\in \fM$, there exists an open set $N\subseteq \fM$ with $x\in N$, and $\delta_0>0$ such that $\forall \theta\in S^{q-1}$ if $y\in N$ is such that
$\theta_1 X_q(y)+\cdots+\theta_q X_q(y)\ne 0$, then $\forall r\in (0,\delta_0]$,
\begin{equation*}
e^{r\theta_1 X_1+\cdots+r\theta_q X_q}y\ne y.
\end{equation*}
\end{lemma}
\begin{proof}
Since this result is local, it suffices to prove the lemma in the case when $\fM=B^n(1)$ and $x=0\in \R^n$.
We set $N:=B^n(1/2)$.  Take $\delta=\delta(n)>0$ as in \cref{Lemma::PfMoreAssump::Classic}.  Take $\delta_1>0$
so small that $\forall y\in B^n(1/2)$, $t\in B^q(\delta_1)$, we have $e^{t_1 X_1+\cdots+t_q X_q}y\in B^n(3/4)$.
Set $C:=\max_{1\leq j\leq q} \CjNorm{X_j}{1}[B^n(3/4)][\R^n]$, and let $\delta_0=\min\{\delta_1, \delta/qC\}$.
From here, the result follows from \cref{Lemma::PfMoreAssump::Classic}.
\end{proof}

%% file: proofimmers.tex

The ideas behind \cref{Prop::ResQual::Mmanif} are well-known to experts;
however, we could not find an exact statement of \cref{Prop::ResQual::Mmanif} in the literature,
so we include the proof here for completeness, with the understanding that the methods used are
known to experts.  It seems closely related to the theory of orbits
of Sussman \cite{SussmanOrbitsOfFamiliesOfVectorFieldsAndIntegrabilityOfDistributions}
and Stefan \cite{StefanAccsibleSetsOrbitsAndFoliations}, though does not follow directly from these theories.
Similar methods have been used to prove the Frobenius theorem for Lipschitz vector fields; see
\cite{MontanariMorbidelliAFrobeniusTheorem} and references therein.

We begin with
the existence of the $C^2$ structure; we take all the same
notation as in the statement of \cref{Prop::ResQual::Mmanif}.
Set $D:=\dim \fM$, and let $(\phi_\alpha,U_\alpha)_{\alpha\in \sA}$ be a $C^2$
atlas for $\fM$ with $\{U_\alpha:\alpha\in \sA\}$ an open cover for $\fM$
and $\phi_\alpha:U_\alpha\rightarrow B^D(1)$ a $C^2$ diffeomorphism.

Let $\Xa_j=(\phi_{\alpha})_{*} X_j$ so that $\Xa_j$ is a $C^1$ vector field on $B^D(1)$.
We may pick the above atlas so that $\Norm{\Xa_j}[C^1(B^D(1);\R^n)]<\infty$.

\begin{lemma}\label{Lemma::ProofImmerse::FinerTop}
Let $Z$ be as in the beginning of \cref{Section::Series::Qual}.
The topology on $Z$ (induced by the metric $\rho$) is finer than the
topology 
as
a subspace of $\fM$.
\end{lemma}
\begin{proof}
Let $U\subseteq \fM$ be an open set and let $x\in U\cap Z$.  We wish to show that
there is a $\delta>0$ with $B_X(x,\delta)\subseteq U$.  Since $x\in U_\alpha$ for
some $\alpha\in \sA$, we may replace $U$ with $U\cap U_\alpha$, and therefore
assume $U\subseteq U_\alpha$ for some $\alpha\in \sA$.

By the Picard-Lindel\"of Theorem, there exists $\delta>0$ so small
such that given $a_1,\ldots, a_q\in L^\infty([0,1])$ with
$\Norm{\sum |a_j|^2}[L^\infty([0,1])]<1$, there exists
a unique $\gammat:[0,1]\rightarrow \phi_{\alpha}(U)$ with
\begin{equation}\label{Eqn::ProofImmers::gammat::defn}
\gammat(0)=\phi_\alpha(x)\text{ and }\gammat'(t)=\sum_{j=1}^q a_j(t) \delta \Xa(\gammat(t)).
\end{equation}
We claim $B_X(x,\delta)\subseteq U$.  Indeed, fix $y\in B_X(x,\delta)$.
By the definition of $B_X(x,\delta)$, $\exists \gamma:[0,1]\rightarrow B_X(x,\delta)$, $\gamma(0)=x$, $\gamma(1)=y$,
$\gamma'(t) = \sum_{j=1}^q a_j(t) \delta X_j(\gamma(t))$.
Let $\gammat:[0,1]\rightarrow \phi_{\alpha}(U)$ be the unique solution to \cref{Eqn::ProofImmers::gammat::defn} with this choice of $a_1,\ldots, a_q$, and set $\gammah:=\phi_{\alpha}^{-1}\circ \gammat$.
Then, $\gammah(0)=x=\gamma(0)$, $\gammah'(t) = \sum_{j=1}^q a_j(t)\delta X_j(\gamma(t))=\gamma'(t)$.  Standard uniqueness theorems for ODEs
show $\gamma=\gammah$, and therefore $y=\gamma(1)=\gammah(1)=\phi_\alpha^{-1}(\gammat(1))$.
Since $\gammat(1)\in \phi_\alpha(U)$, it follows $y\in U$, which completes the proof.
%
%
%
\end{proof}

Recall, $M$ is a connected open subset of $Z$ which is given the topology as subspace of $Z$; i.e., $M$ is given the topology
induced by the metric $\rho$.

Set $M_\alpha:=\phi_\alpha(U_\alpha\cap M)$; we give $M_\alpha$ the topology
so that $\phi_\alpha: M\cap U_\alpha\rightarrow M_\alpha$ is a homeomorphism
(with $M\cap U_\alpha\subseteq M$ given the topology as a subspace of $M$).
Let $\Xa(u)$ denote the $D\times q$ matrix $\Xa(u)=(\Xa_1(u)|\cdots | \Xa_q(u))$.
For $K=(k_1,\ldots, k_l)\in \sI(l,q)$ let $\Xa_K$ denote the list of vector fields
$\Xa_{k_1},\ldots, \Xa_{k_l}$ and for $J=(j_1,\ldots, j_l)\in \sI(l,D)$
let $\Xa_{J,K}$ denote the $l\times l$ submatrix of $\Xa(u)$ given by taking
the rows listed in $J$ and the columns listed in $K$.

\begin{lemma}\label{Lemma::Appendix::XaDeriv}
For $u\in M_\alpha$, $1\leq k\leq q$, $1\leq l\leq \min\{q,D\}$, $K\in \sI(l,q)$, $J\in \sI(l,D)$
\begin{equation*}
    \Xa_k \det \Xa_{J,K}(u) = \sum_{\substack{K'\in \sI(l,q) \\ J'\in \sI(l,D) }} f_{k,J,K}^{J',K'} \det \Xa_{J',K'}(u),
\end{equation*}
where $f_{k,J,K}^{J',K'}: M_\alpha\rightarrow \R$ are locally bounded.
\end{lemma}
\begin{proof}
Let $J=(j_1,\ldots, j_l)$, $K=(k_1,\ldots, k_l)$.
Then,
$\det \Xa_{J,K} = \nu_J (\Xa_{k_1},\ldots, \Xa_{k_l})$,
where $\nu_J$ is the $l$-form $du_{j_1}\wedge du_{j_2}\wedge \cdots \wedge du_{j_l}$.
Hence, using
\cite[Proposition 18.9]{LeeIntroToSmoothManifolds}
we have
\begin{equation}\label{Eqn::AppendImmerse::LieDeriv}
\begin{split}
    &\Xa_k \det \Xa_{J,K} = \sL_{\Xa_k} \left(\nu_J(\Xa_{k_1},\ldots, \Xa_{k_l})\right)
    \\&=\left(\sL_{\Xa_k} \nu_J\right)(\Xa_{k_1},\ldots, \Xa_{k_l})
    +\nu_J([\Xa_k, \Xa_{k_1}], \Xa_{k_2},\ldots, \Xa_{k_l})
    \\&\quad+\nu_J(\Xa_{k_1}, [\Xa_{k}, \Xa_{k_2}],\Xa_{k_3},\ldots, \Xa_{k_l})
    +\cdots
    +\nu_J(\Xa_{k_1},\ldots, \Xa_{k_{l-1}}, [\Xa_{k}, \Xa_{k_l}])
\end{split}
\end{equation}
We begin with the first term on the right hand side of \cref{Eqn::AppendImmerse::LieDeriv}.
Since $\Xa_k$ is a $C^1$ vector field, $\sL_{\Xa_k}\nu_J$ is a $C^0$ $l$-form on $B^D(1)$
and we have
\begin{equation*}
    \sL_{\Xa_k}\nu_J = \sum_{J'\in \sI(l,D)} f_{k,J}^{J'} \nu_{J'},
\end{equation*}
where $\Norm{f_{k,J}^{J'}}[\CjSpace{0}[B^D(1)]]<\infty$.
Hence
\begin{equation*}
    \left(\sL_{\Xa_k} \nu_J\right)(\Xa_{k_1},\ldots, \Xa_{k_l}) =\sum_{J'\in \sL(l,D)} f_{k,J}^{J'} \det \Xa_{J',K},
\end{equation*}
as desired.

We now turn to the rest of the terms on the right hand side of \cref{Eqn::AppendImmerse::LieDeriv}.
These terms are all similar, so we only discuss the first.  We have
\begin{equation*}
    \nu_J([\Xa_k, \Xa_{k_1}], \Xa_{k_2},\ldots, \Xa_{k_l})
    =\sum_{r} (c_{m,k_1}^r\circ \phi_\alpha) \nu_J(\Xa_r, \Xa_{k_2},\ldots, \Xa_{k_l})
    =\sum_r \left(c_{m,k_1}^r\circ \phi_{\alpha} \right)\det \Xa_{J, K_r},
\end{equation*}
where $K_r=(r,k_2,\ldots, k_l)\in \sI(l,q)$.  The result follows.
\end{proof}

For $1\leq l\leq \min\{D,q\}$ let $\det_{l\times l} \Xa(u)$ denote the vector
whose components are $\det \Xa_{J,K}(u)$, where $J\in \sI(l,D)$, $K\in \sI(l,q)$.

\begin{lemma}\label{Lemma::Appendix::BoundXaDet}
For $u\in M_\alpha$, $1\leq j\leq q$, $1\leq l\leq \min\{D,q\}$, $J\in \sI(l,D)$, $K\in \sI(l,q)$,
\begin{equation*}
    \left| \Xa_j \det \Xa_{J,K}(u) \right|\leq g_{j,J,K}(u) \left|\det_{l\times l} \Xa(u)\right|
\end{equation*}
where $g_{j,J,K}:M_\alpha\rightarrow [0,\infty)$ is locally bounded.
\end{lemma}
\begin{proof}This follows immediately from \cref{Lemma::Appendix::XaDeriv}.\end{proof}

\begin{lemma}\label{Lemma::Appendix::Constdimongamma}
Let $\gamma:[0,1]\rightarrow M_{\alpha}$ be such that
$\gamma'(t)=\sum_{j=1}^q a_j(t) \Xa_j(\gamma(t))$, where $a_j\in L^\infty([0,1])$.
Then,
$\dim \Span\{ \Xa_1(\gamma(0)),\ldots, \Xa_q(\gamma(0))\} = \dim \Span\{ \Xa_1(\gamma(1)),\ldots, \Xa_q(\gamma(1))\} $.
\end{lemma}
\begin{proof}
We will show
\begin{equation}\label{Eqn::Appendix::ToShowDet0}
    \left| \det_{l\times l} \Xa(\gamma(0))\right| = 0 \Rightarrow \left|\det_{l\times l} \Xa(\gamma(1))\right|=0.
\end{equation}
To see why \cref{Eqn::Appendix::ToShowDet0} implies the result note that by reversing $\gamma$,
we have
\begin{equation*}
    \left| \det_{l\times l} \Xa(\gamma(0))\right| = 0 \Leftrightarrow \left|\det_{l\times l} \Xa(\gamma(1))\right|=0,
\end{equation*}
and by noting that $\dim\Span\{\Xa_1(u),\ldots, \Xa_q(u)\}\geq l\Leftrightarrow  \left| \det_{l\times l} \Xa(u)\right| \ne 0$, the result follows.
We turn to proving \cref{Eqn::Appendix::ToShowDet0}.  We have, using \cref{Lemma::Appendix::BoundXaDet},
\begin{equation*}
\begin{split}
    &\frac{d}{dt} \left|\det_{l\times l} \Xa(\gamma(t))\right|^2
    =2\sum_{\substack{J\in \sI(l,D) \\ K\in \sI(l,q)}} \det \Xa_{J,K}(\gamma(t)) \frac{d}{dt} \det \Xa_{J,K}(\gamma(t))
    \\&=2\sum_{\substack{J\in \sI(l,D) \\ K\in \sI(l,q)}} \det \Xa_{J,K}(\gamma(t)) \left( \sum_{j=1}^q a_j(t) \left(\Xa_j \det \Xa_{J,K}\right)(\gamma(t))\right)
    \\&\leq 2 \sum_{\substack{J\in \sI(l,D) \\ K\in \sI(l,q)}} \left(\sup_{t\in [0,1]} g_{j,J,K}(\gamma(t))\right)
    \left( \sum_{j=1}^q \Norm{a_j}[L^\infty([0,1])] \right)
     \left|\det_{l\times l} \Xa(\gamma(t))\right|^2.
\end{split}
\end{equation*}
We conclude,
\begin{equation*}
    \frac{d}{dt} \left|\det_{l\times l} \Xa(\gamma(t))\right|^2\leq C \left|\det_{l\times l} \Xa(\gamma(t))\right|^2,
\end{equation*}
for some constant $C$.  \Cref{Eqn::Appendix::ToShowDet0} follows by Gr\"onwall's inequality.
\end{proof}

\begin{prop}\label{Prop::Appendix::dimconst}
The map $x\mapsto \dim\Span\{X_1(x),\ldots,X_q(x)\}$, $M\rightarrow \N$ is constant.
\end{prop}
\begin{proof}
Since $M$ is connected, it suffices to show the map is locally constant.
Fix $x\in M$ and pick $\alpha\in \sA$ such that $x\in U_\alpha$.
Take $\delta>0$ so small that $B_X(x,\delta)\subset M\cap U_\alpha$ (here, we are
using \cref{Lemma::ProofImmerse::FinerTop}).
We wish to show $x\mapsto \dim\Span\{X_1(x),\ldots,X_q(x)\}$, $B_X(x,\delta)\rightarrow \N$
is constant.

Take $y\in B_X(x,\delta)$, so that $\exists \gamma:[0,1]\rightarrow \fM$,
$\gamma(0)=x$, $\gamma(1)=y$, $\gamma'(t)=\sum a_j(t) \delta X_j(\gamma(t))$,
$\|\sum |a_j(t)|^2\|_{L^\infty([0,1])}<1$.  Note,
$\forall t\in [0,1]$, $\gamma(t)\in B_X(x,\delta)\subseteq U_\alpha$.

Set $\gammat(t):=\phi_\alpha\circ\gamma(t)$.  $\gammat$ satisfies all the
hypotheses of \cref{Lemma::Appendix::Constdimongamma} and this shows
$$\dim \Span\{ \Xa_1(\phi_{\alpha}(x)),\ldots, \Xa_q(\phi_{\alpha}(x))\} = \dim \Span\{ \Xa_1(\phi_{\alpha}(y)),\ldots, \Xa_q(\phi_{\alpha}(y))\}.$$
Hence, $\dim\Span\{X_1(x),\ldots, X_q(x)\} = \dim\Span\{X_1(y),\ldots, X_q(y)\}$,
completing the proof.
\end{proof}

Set $n:=\dim\Span\{X_1(x),\ldots, X_q(x)\}$, $x\in M$ (by \cref{Prop::Appendix::dimconst},
$n$ does not depend on $x$).

\begin{lemma}\label{Lemma::Appendix::PickMaxLI}
Let $x\in M$ and $K=(k_1,\ldots, k_n)\in \sI(n,q)$ such that
$X_{k_1}(x),\ldots, X_{k_n}(x)$ are linearly independent.
Then, there is an open set $U\subseteq \fM$, containing $x$,
$J\in \sI(n,D)$, and $\delta>0$ such that the following hold:
\begin{enumerate}[(i)]
\item\label{Item::Appendix::PickdeltaX} $B_X(x,\delta)\subseteq U$.
\item $\exists \alpha\in \sA$, $U\subseteq U_\alpha$.
\item\label{Item::Appendix::infdetX} $\inf_{u\in \phi_\alpha(U)}\left| \det \Xa_{J,K}(u) \right|>0$.
\item\label{Item::Appendix::SpanXkj} $\forall y\in B_X(x,\delta)$, $\Span\{X_{k_1}(y),\ldots, X_{k_n}(y)\}=\Span\{X_1(y),\ldots, X_q(y)\}$.
\item\label{Item::Appendix::CommuteXkj} $\forall y\in B_X(x,\delta)$, $[X_{k_i}, X_{k_j}](y)\in \Span\{X_{k_1}(y),\ldots, X_{k_n}(y)\}$.
\item\label{Item::Appendix::writeXjInXjk} For $1\leq j\leq q$, $1\leq l\leq n$, $\exists b_{j}^l\in C^1(U)$,
$\Norm{b_j^l\circ \phi_\alpha^{-1}}[C^1(\phi_\alpha(U))]<\infty$, such that
$\forall y\in B_X(x,\delta)$,
\begin{equation}\label{Eqn::Appendix::XjInXkl}
    X_j(y)=\sum_{l=1}^n b_j^l(y) X_{k_l}(y).
\end{equation}
\end{enumerate}
\end{lemma}
\begin{proof}
Let $U\subseteq \fM$ be a neighborhood of $x$ which may shrink from line to line.
First, we may take $U$ so small that $U\subseteq U_\alpha$ for some $\alpha\in \sA$.
Since $\Xa_{k_1}(\phi_\alpha(x)),\ldots, \Xa_{k_n}(\phi_\alpha(x))$ are linearly independent,
by the hypotheses, $\exists J\in \sI(n,D)$ such that
\begin{equation*}
    \left|\det \Xa_{J,K}(\phi_\alpha(x))\right|>0.
\end{equation*}
By the continuity of the map $u\mapsto \left|\det \Xa_{J,K}(u)\right|$, we may shrink $U$
so that \cref{Item::Appendix::infdetX} holds.  We take $\delta>0$ so small that \cref{Item::Appendix::PickdeltaX} holds; here we
are using \cref{Lemma::ProofImmerse::FinerTop}.

Since $\forall u\in \phi_\alpha(U)$, $\left|\det \Xa_{J,K}(u)\right|>0$
we have $\forall y\in B_X(x,\delta)\subseteq U$, $\dim \Span\{X_{k_1}(y),\ldots, X_{k_n}(y)\}
=n=\dim \Span\{X_1(y),\ldots, X_q(y)\}$, proving \cref{Item::Appendix::SpanXkj}.

Since $[X_{k_i},X_{k_j}](y) \in \Span\{X_1(y),\ldots, X_q(y)\}$, $\forall y\in M$ (by assumption),
\cref{Item::Appendix::CommuteXkj} follows from \cref{Item::Appendix::SpanXkj}.

Finally, for \cref{Item::Appendix::writeXjInXjk}, set
\begin{equation*}
    b_j^l(y):=\frac{\det \Xa_{J,K_{j,l}}(\phi_\alpha(y))}{\det \Xa_{J,K}(\phi_\alpha(y))},
\end{equation*}
where $K_{j,l}$ is the same as $K$ but with $k_l$ repalced by $j$.
That $\Norm{b_j^l\circ\phi_\alpha^{-1}}[C^1(\phi_\alpha(U))]<\infty$ follows from
\cref{Item::Appendix::infdetX} and the fact that $X_1,\ldots, X_q\in C^1$.
\Cref{Eqn::Appendix::XjInXkl} follows from Cramer's rule.
\end{proof}

\begin{prop}\label{Prop::Appendix::MakeVs}
Let $x\in M$.  Then there exist an open set $U\subseteq \fM$, containing $x$,
$\delta>0$, and $C^1$ vector fields $V_1,\ldots, V_n$ on $U$ such that the following hold:
\begin{enumerate}[(i)]
\item\label{Item::Appendix::VBall} $B_X(x,\delta)\subseteq U$.
\item\label{Item::Appendix::VU} $\exists \alpha\in \sA$, $U\subseteq U_\alpha$.
\item\label{Item::Appendix::XjIntermVj} For $1\leq j\leq q$, $1\leq l\leq n$, $\exists f_{j}^l\in C^1(U)$,
$\Norm{f_j^l\circ \phi_\alpha^{-1}}[C^1(\phi_\alpha(U))]<\infty$
such that $\forall y\in B_X(x,\delta)$,
\begin{equation*}
    X_j(y) = \sum_{l=1}^n f_j^l(y) V_l(y).
\end{equation*}
\item\label{Item::Appendix::VLI} $\forall y\in B_X(x,\delta)$, $V_1(y),\ldots, V_n(y)$ are linearly independent.
\item\label{Item::Appendix::VjIntermsXj} For all $1\leq l\leq n$, $1\leq j\leq q$, $\exists g_l^j\in C^1(U)$,
$\Norm{g_l^j\circ \phi_\alpha^{-1}}[C^1(\phi_\alpha(U))]<\infty$, such that
$\forall y\in B_X(x,\delta)$,
\begin{equation*}
    V_l(y) =\sum_{j=1}^q g_l^j(y) X_j(y).
\end{equation*}
\item\label{Item::Appendix::VCommute} $\forall y\in B_X(x,\delta)$, $[V_j, V_k](y)=0$, $1\leq j,k\leq n$.
\end{enumerate}
\end{prop}
\begin{proof}
Take $K=(k_1,\ldots, k_n)\in\sI(n,q)$ such that $X_{k_1}(x),\ldots, X_{k_n}(x)$ are linearly independent
and let $J\in \sI(n,D)$, $U\subseteq \fM$, $\delta>0$ be as in \cref{Lemma::Appendix::PickMaxLI}.
Without loss of generality, we may reorder the vector fields and coordinates so that
$J=(1,\ldots, n)$, $K=(1,\ldots,n)$.

For $1\leq j \leq n$, $u\in \phi_\alpha(U)$, write
\begin{equation*}
    \Xa_j = \sum_{k=1}^D h_{j,k}\circ \phi_\alpha^{-1}(u) \diff{x_k},
\end{equation*}
and let $H(y)$ denote the $n\times n$ matrix $H(y)=(h_{j,k})_{1\leq j\leq n, 1\leq k\leq n}$.
Clearly, $\Norm{h_{j,k}\circ \phi_\alpha^{-1}}[C^1(\phi_\alpha(U))]<\infty$.
By \cref{Lemma::Appendix::PickMaxLI} \cref{Item::Appendix::infdetX},
$\inf_{u\in \phi_\alpha(U)} |\det H(u)|>0$.  Define $h^{j,k}$ by
$H(y)^{-1}=(h^{j,k}(y))_{1\leq j\leq n, 1\leq k\leq n}$, $y\in U$.
By the above comments, $\Norm{h^{j,k}\circ \phi_\alpha^{-1}}[C^{1}(\phi_\alpha(U))]<\infty$.
Set
$$V_j(y)=\sum_{k=1}^n h^{j,k}(y) X_k(y), \quad y\in U,$$
so that \cref{Item::Appendix::VjIntermsXj} holds, by definition.  Furthermore,
for $1\leq j\leq n$,
\begin{equation*}
    X_j(y)=\sum_{k=1}^n h_{j,k}(y) V_k(y),
\end{equation*}
so that \cref{Item::Appendix::XjIntermVj} holds for $1\leq j\leq n$.  For $n+1\leq j\leq q$,
\cref{Item::Appendix::XjIntermVj} follows from this and \cref{Lemma::Appendix::PickMaxLI} \cref{Item::Appendix::writeXjInXjk}.
Since $\forall y\in B_X(x,\delta)$, $\dim \Span\{X_1(y),\ldots, X_q(y)\}=n$, we see from \cref{Item::Appendix::XjIntermVj}
that $\dim\Span \{V_1(y),\ldots, V_n(y)\}=n$ and so \cref{Item::Appendix::VLI} follows.

It remains to prove \cref{Item::Appendix::VCommute}.  Let $\Va_k:=(\phi_\alpha)_{*} V_k$,
so that $\Va_k$ is a $C^1$ vector field on $\phi_\alpha(U)$.  By the construction of $\Va_k$,
$\forall u\in U$,
\begin{equation}\label{Eqn::Appendix::TheyAreUsualDerivs}
    \Va_k(\phi_\alpha(u)) \equiv \diff{u_k} \mod \left\{\diff{u_{n+1}},\ldots, \diff{u_D}\right\}.
\end{equation}
Also, by \cref{Item::Appendix::XjIntermVj,Item::Appendix::VjIntermsXj}, for $y\in B_X(x,\delta)$,
\begin{equation*}
    [\Va_j, \Va_k](\phi_\alpha(y))\in \Span\{\Xa_1(\phi_\alpha(y)),\ldots, \Xa_q(\phi_\alpha(y))\} = \Span\{\Va_1(\phi_\alpha(y)),\ldots, \Va_n(\phi_\alpha(y))\}.
\end{equation*}
Combining this with \cref{Eqn::Appendix::TheyAreUsualDerivs}, we have for $y\in B_X(x,\delta)$,
\begin{equation*}
    [\Va_j, \Va_k](\phi_\alpha(y))\in \Span\{\Va_1(\phi_\alpha(y)),\ldots, \Va_n(\phi_\alpha(y))\}\cap \Span\left\{ \diff{u_{n+1}},\ldots, \diff{u_D}\right\}=\{0\}.
\end{equation*}
\Cref{Item::Appendix::VCommute} follows, completing the proof.
\end{proof}

\begin{lemma}\label{Lemma::Appendix::CommuteLemma}
Let $W$ and $Z$ be $C^1$ vector fields on an open set $U\subseteq \R^d$.
Then, $\forall x\in U$, $t,s\in \R$ such that
$e^{-s Z}e^{-\tau W} e^{sZ}e^{\tau W}x$ makes sense for all $\tau \in [\min\{0,t\},\max\{0,t\}]$,
we have
\begin{equation*}
    e^{-sZ} e^{-tW} e^{sZ} e^{tW} x = x+\int_0^t\int_0^s \left( [W,Z](e^{-sZ}e^{-\tau X}e^{\sigma Z})\right)\left( e^{(s-\sigma) Z} e^{\tau W}(x)\right)\: d\sigma d\tau,
\end{equation*}
where we have written $([W,Z](f))(y)$ to denote the vector field $[W,Z]$ applied to the function $f$, then evaluated at the point $y$.
\end{lemma}
\begin{proof}
This is \cite[Lemma 4.1]{RampazzoSussmanCommutatorsOfFlowMapsOfNonsmoothVectorFields}.
\end{proof}

Fix $x\in M$ and let $\alpha\in \sA$, $U\subseteq U_\alpha$, $\delta>0$, and
$V_1,\ldots, V_n$ be as in \cref{Prop::Appendix::MakeVs}.
By \cref{Prop::Appendix::MakeVs} \cref{Item::Appendix::VjIntermsXj},
there exists $\delta_1>0$ such that $B_V(x,\delta_1)\subseteq B_X(x,\delta)$.
For $\epsilon=\epsilon(x)>0$ sufficiently small,\footnote{We allow $\epsilon>0$ to shrink, as needed, throughout the argument.}
 define the map $\Phi_x:B^n(\epsilon)\rightarrow \fM$ by
\begin{equation*}
    \Phi_x(t_1,\ldots, t_n)= e^{t_1 V_1} e^{t_2 V_2}\cdots e^{t_n V_n} x.
\end{equation*}
Note that for $t\in B^{n}(\delta_1/n)$, $\Phi_x(t)\in B_V(x,\delta_1)\subseteq B_X(x,\delta)\subseteq M$.

\begin{lemma}\label{Lemma::Appendix::PhiSymm}
For $\epsilon=\epsilon(x)>0$ sufficiently small and for any permutation $\sigma\in S_n$,
\begin{equation*}
    \Phi_x(t_1,\ldots, t_n) = e^{t_{\sigma(1)} V_{\sigma(1)}} e^{t_{\sigma(2)} V_{\sigma(2)}} \cdots  e^{t_{\sigma(n)} V_{\sigma(n)}} x, \quad \forall t\in B^n(\epsilon).
\end{equation*}
\end{lemma}
\begin{proof}
The minor difficulty in this lemma is that $V_1,\ldots, V_n$ are only known to commute on $B_X(x,\delta)$,
not on a neighborhood in $\fM$--since we do not yet know that $B_X(x,\delta)$ is a manifold, the lemma does not follow from standard results.  We prove the lemma with $\epsilon=\delta_1/4n$.
It suffices to show $\forall l\in \{1,\ldots, n-1\}$,
\begin{equation*}
    e^{t_1 V_1} e^{t_2 V_2}\cdots e^{t_l V_l} e^{t_{l+1} V_{l+1}} \cdots e^{t_n V_n} x =e^{t_1 V_1} e^{t_2 V_2} \cdots e^{t_{l-1} V_{l-1}} e^{t_{l+1} V_{l+1}} e^{t_l V_l} e^{t_{l+2} V_{l+2}}\cdots e^{t_n V_n} x,
\end{equation*}
as the result will then follow by repeated applications of this and by symmetry in the assumptions
on $V_1,\ldots, V_n$.
Since $e^{t_{l+2}V_{l+2}}\cdots e^{t_n V_n} x\in B_V(x,\delta_1/4)$ it suffices
to show $\forall (t_l, t_{l+1})\in B^2(\epsilon)$, $y\in B_V(x,\delta_1/4)$,
\begin{equation}\label{Eqn::Appendix::ToShowCommute}
    e^{t_lV_l} e^{t_{l+1}V_{l+1}} y =e^{t_{l+1} V_{l+1}}e^{t_l V_l }y.
\end{equation}
Note, $\forall (t_l,t_{l+1})\in B^2(\epsilon)$,
\begin{equation*}
    e^{-t_l V_l } e^{-t_{l+1} V_{l+1}} e^{t_l V_l} e^{t_{l+1} V_{l+1}} y \in B_V(x,\delta_1)\subseteq B_X(x,\delta).
\end{equation*}
Pushing this equation forward via $\phi_\alpha$ gives
\begin{equation*}
    e^{-t_l \Va_l } e^{-t_{l+1} \Va_{l+1}} e^{t_l \Va_l} e^{t_{l+1} \Va_{l+1}} \phi_\alpha(y).
\end{equation*}
Since $[\Va_l, \Va_{l+1}](u)=0$, $\forall u\in \phi_\alpha(B_V(x,\delta_1))\subseteq \phi_\alpha(B_X(x,\delta))$,
it follows from \cref{Lemma::Appendix::CommuteLemma} that
\begin{equation*}
    e^{-t_l \Va_l } e^{-t_{l+1} \Va_{l+1}} e^{t_l \Va_l} e^{t_{l+1} \Va_{l+1}} \phi_\alpha(y)=\phi_{\alpha}(y),
\end{equation*}
and so
\begin{equation*}
    e^{t_l \Va_l} e^{t_{l+1} \Va_{l+1}} \phi_\alpha(y) = e^{t_{l+1} \Va_{l+1}} e^{t_l \Va_l} \phi_\alpha(y).
\end{equation*}
\Cref{Eqn::Appendix::ToShowCommute} follows, completing the proof.
\end{proof}

\begin{lemma}\label{Lemma::Appendix::PropsPhi}
For $\epsilon=\epsilon(x)>0$ sufficiently small,
\begin{enumerate}[(i)]
\item\label{Item::Appendix::PhiOpen} $\Phi_x(B^n(\epsilon))\subseteq B_X(x,\delta)$ is an open set (and we give $\Phi_x(B^n(\epsilon))$ the
subspace topology).
\item\label{Item::Appendix::PhiHomeo} $\Phi_x:B^n(\epsilon)\rightarrow \Phi_x(B^n(\epsilon))$ is a homeomorphism.
\item\label{Item::Appendix::PhiC2} $\Phi_x:B^n(\epsilon)\rightarrow \fM$ is $C^2$ and $d\Phi_x(u)$ has full rank (i.e., rank $n$),
    $\forall u\in B^n(\epsilon)$.
\item\label{Item::Appendix::PushForwardPhixVj} $d\Phi_x(u)\diff{u_j} = V_j(\Phi_x(u))$.
\item\label{Item::Appendix::PhiExistY} There are $C^1$ vector fields $Y_1,\ldots, Y_q$ on $B^n(\epsilon)$ with
$\Norm{Y_j}[C^1(B^n(\epsilon);\R^n)]<\infty$ such that
$d\Phi_x(u) Y_j(u) = X_j(\Phi_x(u))$.
\end{enumerate}
\end{lemma}
\begin{proof}
We have already seen $\Phi_x(B^n(\epsilon))\subseteq B_V(x,\delta_1)\subseteq B_X(x,\delta)$.
Since $V_1,\ldots, V_n$ are $C^1$, standard proofs show that $\Phi_x$ is $C^1$.
Since $\diff{t_j}\big|_{t=0} \Phi_x(t)= V_j(x)$ and
$V_1(x),\ldots, V_n(x)$ are linearly independent
(\cref{Prop::Appendix::MakeVs} \cref{Item::Appendix::VLI}) the
Inverse Function Theorem shows that if $\epsilon>0$ is sufficiently small,
$\Phi_x:B^{n}(\epsilon)\rightarrow \fM$ is injective and $d \Phi_x(u)$ has full rank (i.e., rank $n$)
$\forall u\in B^n(\epsilon)$.

By the definition of $\Phi_x$, $\diff{t_1}\Phi_x(t) = V_1(\Phi_x(t))$,
and by \cref{Lemma::Appendix::PhiSymm}, $\Phi_x$ is symmetric in $V_1,\ldots, V_n$
and so \cref{Item::Appendix::PushForwardPhixVj} follows for $\epsilon>0$ sufficiently small.

Let $S\subseteq B^n(\epsilon)$ be open.  We claim $\Phi_x(S)\subseteq B_X(x,\delta)$ is open.
Indeed, take $\Phi_x(u)\in \Phi_x(S)$.  Let $\epsilon_0>0$ be so small  that $B^n(u,\epsilon_0)\subseteq S$.
Then $\Phi_x(B^n(u,\epsilon_0))\subseteq \Phi_x(S)$.  And so
$B_V(\Phi_x(u),\epsilon_0)=\Phi_x(B^n(u,\epsilon_0))\subseteq \Phi_x(S)$.\footnote{To conclude $B_V(\Phi_x(u),\epsilon_0)=\Phi_x(B^n(u,\epsilon_0))$, we
have used $d\Phi_x(t) \diff{t_j} = V_j(\Phi_x(t))$ and the definition of $B_V(\Phi_x(y), \epsilon_0)$.}
By \cref{Prop::Appendix::MakeVs} \cref{Item::Appendix::XjIntermVj} $\exists \epsilon_1>0$
with $B_X(\Phi_x(u),\epsilon_1)\subseteq B_V(\Phi_x(u),\epsilon_0)=\Phi_x(B^n(u,\epsilon_0))\subseteq \Phi_x(S)$.
Thus, $\Phi_x(S)\subseteq B_X(x,\delta)$ is open.  In particular $\Phi_x(B^n(\epsilon))\subseteq B_X(x,\delta)$ is open.  This proves
\cref{Item::Appendix::PhiOpen}.

Since $\Phi_x$ is an injective open map, to prove it is a homeomophism it suffices to prove it is continuous.
Let $u\in B^n(\epsilon)$ and let $S\subseteq B_X(x,\delta)$ be an open set such that $\Phi_x(u)\in S$.
We wish to show that there is an open set $O\subseteq B^n(\epsilon)$, $u\in O$, $\Phi_x(O)\subseteq S$.

Take $\epsilon_0>0$ so small that $B_X(\Phi_x(u),\epsilon_0)\subseteq S$.  Then by \cref{Prop::Appendix::MakeVs} \cref{Item::Appendix::XjIntermVj}  $\exists \epsilon_1>0$ with $B_V(\Phi_x(u),\epsilon_1)\subseteq B_X(\Phi_x(u),\epsilon_0)\subseteq S$.  But $\Phi_x(B^n(u,\epsilon_1))= B_V(\Phi_x(u),\epsilon_1)$;
thus $O=B^n(u,\epsilon_1)$ is our desired neighborhood of $u$.  This proves \cref{Item::Appendix::PhiHomeo}.

Taking $f_j^l$ as in \cref{Prop::Appendix::MakeVs} \cref{Item::Appendix::XjIntermVj},
and setting $Y_j(u) = \sum_{l=1}^n f_{j}^l\circ \Phi_x(u) \diff{u_l}$,
\cref{Item::Appendix::PhiExistY} follows.

For \cref{Item::Appendix::PhiC2}, we already know $\Phi_x$ is $C^1$.  That $\Phi_x$ is $C^2$ follows from
\cref{Item::Appendix::PushForwardPhixVj} and the fact that $V_1,\ldots, V_n$ are $C^1$.
We have already shown $d\Phi_x(u)$ has full rank, $\forall u\in B^n(\epsilon)$.
\end{proof}

In the previous discussion, $\epsilon>0$ implicitly depended on $x$.  We now make this dependance
explicit and write $\epsilon_x>0$.
We consider a family of functions and open sets on $M$ given by
\begin{equation*}
    \left\{ (\Phi_x^{-1}, \Phi_x(B^n(\epsilon_x)) \right\}_{x\in M}.
\end{equation*}
The proof of the existence of the $C^2$ structure in  \cref{Prop::ResQual::Mmanif} is completed by the next proposition.

\begin{prop}
The above maps yield a $C^2$ atlas on $M$.  With this manifold structure $X_1,\ldots, X_q$ are $C^1$ vector fields
on $M$, and the inclusion map $M\hookrightarrow \fM$ is a $C^2$ injective immersion.
\end{prop}
\begin{proof}
The main point is to show that the collection of maps gives a $C^2$ atlas.  Once this is shown,
that $X_1,\ldots, X_q$ are $C^1$ on this manifold follows from \cref{Lemma::Appendix::PropsPhi} \cref{Item::Appendix::PhiExistY}.  That the inclusion map is a $C^2$ injective immersion
follows from \cref{Lemma::Appendix::PropsPhi} \cref{Item::Appendix::PhiC2}.

We turn to showing the collection is a $C^2$ atlas.  Set
$W=\Phi_{x_1}(B^n(\epsilon_{x_1}))\cap \Phi_{x_2}(B^n(\epsilon_{x_2}))$.
We want to show $\Phi_{x_1}^{-1}\circ \Phi_{x_2}:\Phi_{x_2}^{-1}(W)\rightarrow B^n(\epsilon_{x_1})$
is $C^2$.  Since $\Phi_{x_1}:B^n(\epsilon_{x_1})\rightarrow \fM$ is injective, $C^2$,
and has injective differential (\cref{Lemma::Appendix::PropsPhi} \cref{Item::Appendix::PhiHomeo} and \cref{Item::Appendix::PhiC2}) we have
\begin{equation*}
    \Phi_{x_1}^{-1}\circ \Phi_{x_2}\text{ is }C^2\Leftrightarrow \Phi_{x_1}\circ \Phi_{x_1}^{-1}\circ \Phi_{x_2}\text{ is } C^2.
\end{equation*}
But $\Phi_{x_1}\circ \Phi_{x_1}^{-1}\circ \Phi_{x_2}=\Phi_{x_2}$ is $C^2$ by \cref{Lemma::Appendix::PropsPhi} \cref{Item::Appendix::PhiC2}, completing the proof.
\end{proof}

Finally,  the uniqueness of the $C^2$ structure in \cref{Prop::ResQual::Mmanif} follows immediately from the next lemma and \cref{Lemma::ProofImmerse::FinerTop}.
\begin{lemma}
Let $\fM$ be a manifold and let $M\subseteq \fM$ be a subset.  Give $M$
any topology which is finer\footnote{Not necessarily strictly finer.} than the subspace topology induced by $\fM$.  Then, there is at most one $C^2$ manifold structure on $M$, compatible with this topology,
such that the inclusion map $M\hookrightarrow \fM$ is an injective immersion.
\end{lemma}
\begin{proof}
Suppose there are two such $C^2$ structures on $M$; denote the corresponding $C^2$ manifolds by $M_1$ and $M_2$.
We wish to show that the identity map $M_1\rightarrow M_2$ is a $C^2$ diffeomorphism.
Let $i_1:M_1\hookrightarrow \fM$, $i_2:M_2\hookrightarrow \fM$ be the inclusion maps (on
the underlying space $M$, $i_1=i_2$).
Since $i_1$ and $i_2$ are assumed to be injective immersions, for all $x\in M$,
there is a neighborhood $U\subseteq M$ of $x$ such that
\begin{equation*}
    i_1|_U :M_1\cap U\rightarrow \fM\cap U, \quad i_2|_U:M_2\cap U\rightarrow \fM\cap U
\end{equation*}
are $C^2$ diffeomorpisms, where $\fM\cap U$ is given the $C^2$ structure as a submanifold of $\fM$.
Hence, the idenitity map $U\cap M_1\rightarrow U\cap M_2$ is a $C^2$ diffeomorphism.
Since the idenitity map $M_1\rightarrow M_2$ is a homeomorphism which is locally a $C^2$
diffeomorphism, we conclude that it is a global $C^2$ diffeomorphism, as desired.
\end{proof}